\documentclass[letter,11pt, reqno]{amsart}
\usepackage[margin=3cm]{geometry}
\usepackage{amsthm}

 \newtheorem{theorem}{Theorem}
 \newtheorem{proposition}{Proposition}
 \newtheorem{rem}{Remark}
 
  \newtheorem{lemma}{Lemma}
  \newtheorem{corollary}{Corollary}
\usepackage{algorithm}
\usepackage{algpseudocode}
\usepackage{tikz}

\usepackage{amsmath, amssymb,amsfonts}
\usepackage{mathtools}
\usepackage{booktabs}
\usepackage{graphicx}
\usepackage{xcolor}
\usepackage{natbib}
\usepackage{adjustbox}
\usepackage{bm}
\usepackage{multirow}
\usepackage{tikz}
\usetikzlibrary{arrows.meta,positioning,calc,fit,shapes.geometric,decorations.pathreplacing}

\usepackage{hyperref}
\hypersetup{
  colorlinks=true,
  linkcolor=blue,
  citecolor=blue,
  urlcolor=blue
}

\usepackage{placeins}

\def\P{\mathbb{P}}
\def\E{\mathbb{E}}

\def\A{\mathcal{A}}
\def\RR{\mathcal{R}}

\definecolor{punkpink}{RGB}{255, 20, 147}

\usepackage{graphicx}
\usepackage{subcaption}

\usepackage{algorithm}
\usepackage{algpseudocode}

\algnewcommand\algorithmicinput{\textbf{Input:}}
\algnewcommand\algorithmicoutput{\textbf{Output:}}
\algnewcommand\Input{\item[\algorithmicinput]}
\algnewcommand\Output{\item[\algorithmicoutput]}

\usepackage{booktabs}
\usepackage{longtable}
\usepackage{makecell}

\let\origmaketitle\maketitle
\def\maketitle{
	\begingroup
	\def\uppercasenonmath##1{} 
	\let\MakeUppercase\relax 
	\origmaketitle
	\endgroup
}

\begin{document}

\title[]{\huge Exponential Conic Optimization for Multi-Regime Service System Design under Congestion and Tail-Risk Control}

\author[V. Blanco, M. Mart\'inez-Ant\'on \MakeLowercase{and} J. Puerto]{
{\large V\'ictor Blanco$^{\dagger}$, Miguel Mart\'inez-Ant\'on$^{\ddagger}$, and Justo Puerto$^{\ddagger}$}\medskip\\
$^\dagger$Institute of Mathematics (IMAG), Universidad de Granada\\
\texttt{vblanco@ugr.es}\\
$^\dagger$Institute of Mathematics (IMUS), Universidad de Sevilla\\
\texttt{mmartinez31@us.es}, \texttt{puerto@us.es}
}

\maketitle

\begin{abstract}
We study the design of single-facility service systems operating under multiple recurring regimes with service-level constraints on response times. Regime-dependent arrival and service rates induce hyperexponential response-time distributions, and the design problem selects regime-specific capacities to balance cost, congestion, fairness, and reliability. We propose a mixed-integer exponential conic optimization framework integrating SLA chance constraints, conflict-graph design restrictions, and CVaR-based tail-risk control. Although NP-hard, the problem admits an efficient decomposition scheme and tractable special cases. Computational experiments and a large-scale urban case study show substantial improvements over the current system, quantifying explicit trade-offs between efficiency, congestion control, fairness, and robustness. The framework provides a practical tool for congestion-aware and tail-control service system design.
\end{abstract}

\keywords{Hyperexponential queueing system;  Service-level agreement; Tail-risk control;  Conic optimization; Stochastic optimization; Integer programming}

\section{Introduction}

Service systems in both private and public sectors increasingly operate under
explicit performance requirements that regulate not only average congestion
levels but also the reliability of user response times.
Such requirements arise in customer support operations, logistics and
transportation services, emergency response systems, healthcare delivery, and
cloud computing infrastructures.
Service-level agreement (SLA), in these contexts, typically imposes
probabilistic guarantees, requiring congestion-induced delays to remain below
prescribed thresholds with high confidence.
Designing service systems that satisfy reliability requirements at minimum
operational cost is therefore a central problem in operations research.

A defining feature of modern service systems is demand heterogeneity.
Demand elements differ in their exogenous access characteristics, such as
network latency, travel time, or dispatch delay, while sharing access to
capacity-constrained service resources.
Accordingly, a user’s total response time model naturally decomposes into a
deterministic access component and a stochastic congestion component arising
from waiting and service at the facility.
Meaningful reliability guarantees must therefore account for heterogeneous
demand rather than relying solely on aggregate averages.

In addition, real systems rarely operate as a single homogeneous queue.
Instead, service capacity is often structured into a certain number of recurring
operating regimes that differ systematically in congestion conditions and
processing speed.
These regimes may represent alternative staffing configurations, distinct
service modes, specialized resources, or routing policies.
Examples include differentiated support tiers, parallel processing pipelines
in cloud systems, fast and slow treatment tracks in healthcare, and multiple
dispatch channels in emergency response.
Each regime behaves as a stable queueing system with its own arrival intensity
and service capacity.

We model such systems as collections of parallel
$\mathrm{M}/\mathrm{M}/1$ queues with regime-dependent arrival rates and
service capacities, where incoming demand is probabilistically routed across
regimes.
Conditional on the selected regime, response times follow the classical
exponential behavior of stable queues. However,
from the perspective of an arbitrary arrival, regime assignment is
random, and the overall response-time distribution becomes a mixture of
exponential components.
This hyperexponential structure arises endogenously from regime-based routing
and captures the empirically observed coexistence of typical performance and
occasional congestion episodes.
At the same time, it preserves analytical tractability, admitting closed-form
expressions for cumulative probabilities and expectations.

The goal of this paper is to exploit this structure for service system design.
We develop a flexible optimization framework in which regime-specific service
rates and routing-induced response-time distributions are decision variables.
SLA enforcement is modeled selectively: binary activation variables determine
which demand elements are protected by probabilistic response-time guarantees,
while coverage and fairness requirements are imposed through additional
constraints.
Beyond individual SLA compliance, the framework incorporates explicit control
of tail performance through convex risk measures, allowing the decision-maker to
limit the average response time experienced by the worst-performing fraction
of demand.
This provides protection against extreme degradation without imposing overly
conservative worst-case guarantees.

All these features are embedded in a unified mixed-integer conic optimization formulation
that leverages the analytical structure of exponential queues.
Exact expressions for response-time distributions lead to constraints that are
representable with exponential cones, yielding tractable models that jointly
enforce stability, selective SLA, congestion-sensitive loss functions,
and tail-risk control.
We show that the general design problem is NP-hard, while identifying
structured special cases, such as uniform service-level confidence parameters, admitting
polynomial-time solution procedures.

The contribution of this paper is thus a unified optimization framework for
the reliability-oriented design of regime-based Poisson service systems.
By integrating hyperexponential response-time structure, selective SLA
coverage, and risk-aware performance criteria within an exponential cone-representable formulation, the proposed approach provides both theoretical
insight and implementable decision-support tools.
Computational experiments illustrate the trade-offs between capacity
investment, congestion, and service reliability across alternative design
paradigms.

\subsection{Related Work}

This paper relates to several streams of the operations research literature on
queueing system design, service-level guarantees, and risk-aware optimization.
Our emphasis is on optimization-based approaches that explicitly embed queueing
performance measures into system design decisions.

Queueing theory form the backbone of service system analysis, capturing both
capacity limitations and the temporal interaction between arrivals and service
completions.
Classical models such as $\mathrm{M}/\mathrm{M}/1$ and $\mathrm{M}/\mathrm{D}/1$
yield tractable expressions for key performance measures and underpin a large
body of work on staffing, capacity planning, and control
\citep{Gross2008}.
However, empirical evidence consistently points to substantial heterogeneity and
variability in service times, often leading to heavy-tailed or multimodal
response-time distributions that are poorly approximated by single exponential
models \citep{hokstad78}.

Hyperexponential and, even generally, phase-type distributions address this
limitation by representing service times as mixtures of exponential regimes.
These models have been widely used in telecommunications, computer systems, and
cloud infrastructures to capture bursty behavior and rare congestion events
\citep{Asmussen2003,FeldmannWhitt1998}.
Despite their analytical tractability, hyperexponential models have been used
primarily for performance evaluation, with relatively limited attention devoted
to their systematic integration into optimization-based design frameworks.

Optimization models for service system design have a long tradition, particularly
in call centers and service operations.
Early and recent contributions study how staffing levels, service rates, and
routing policies can be optimized to balance operating costs and delay-related
performance measures
\citep{Borst2004,Garnett2002,goez2017second,GurvichWhitt2009,HarrisonZeevi2005,hassin2003queue,Long2024generalized,Tsitsiklis2017flexible,Zychlinski2023managing}.
These works typically rely on exponential service-time assumptions, diffusion
approximations, or asymptotic regimes, and they often focus on average delays or
probabilities of delay.
In contrast, the explicit optimization of regime-dependent parameters in
hyperexponential queueing systems, especially under tail-sensitive performance
criteria, remains relatively uncharted.

Several works have also analyzed service network design  and location problems under congestion effects, explicitly incorporating queueing performance into spatial decision models. Early contributions such as \citep{berman1985} study optimal server location on networks operating as $\mathrm{M}/\mathrm{G}/1$ queues, embedding waiting-time considerations into location decisions. Subsequent research has developed broader location frameworks with congestion, including models that account for interaction between spatial demand allocation and service delays \citep{berman2003location}, as well as location--allocation problems on congested networks \citep{aboolian2008location}. More recent contributions consider probabilistic covering formulations under congestion \citep{aboolian2022probabilistic} and integrated service system design models with stochastic demand, congestion, and consumer choice \citep{aboolian2022efficient}. While these studies incorporate congestion into location decisions, they typically rely on steady-state queueing metrics or average-delay approximations. In contrast, our approach embeds exact regime-dependent response-time distributions directly into the design problem and explicitly controls reliability and tail performance through risk-sensitive optimization.

Service-level agreement provides a formal mechanism to encode performance
requirements in many application domains, including telecommunications, cloud
computing, logistics, and healthcare
\citep{HarcholBalter2013,jain2002sla}.
In queueing-based optimization models, SLA is commonly imposed as probabilistic
constraints on response times or as constraints on the fraction of customers
served within a target time window
\citep{cheng2013staffing,soh2017call}.
Most existing approaches enforce these constraints uniformly across all demand
classes and rely on approximations or heuristics to maintain tractability.

Recent work has begun to relax uniformity assumptions by allowing selective or
prioritized service guarantees, motivated by heterogeneity in demand value or
delay sensitivity and differentiated service menus
\citep{Bertsimas2018Fairness,Caldentey2025designing,Daskin2013}.
Nevertheless, such approaches are rarely combined with explicit continuous-time
queueing dynamics and response-time chance constraints.

Tail-risk control has also received growing attention in optimization, notably
through coherent risk measures as conditional value-at-risk (CVaR)
\citep{RockafellarUryasev2000,RockafellarUryasev2002}.
CVaR-based formulations admit tractable convex representations and have been
successfully applied in finance \citep{Cesarone2025}, supply chain management, and network design
\citep{Ogryczak2003,puerto2017revisiting}.
Their application to queueing system design remains limited, and prior works
typically address worst-case or average performance rather than directly
regulating the distribution of response times across users.

Finally, recent advances in conic optimization have enabled exact convex
representations of nonlinear performance measures and risk constraints.
Exponential cone formulations, in particular, allow logarithmic and exponential
relationships to be embedded directly into optimization models
\citep{BenTalNemirovski2001,boydvandenberghe2004,Chares2009}.
While these tools have been exploited in congestion control and robust
optimization, their potential for queueing system design with explicit
response-time distributions has not been fully unraveled.

The present paper contributes to this literature by developing a unified
optimization framework that integrates hyperexponential queueing structure,
selective SLA enforcement, congestion-sensitive cost modeling, and CVaR-based
tail-risk control within a single conic formulation.
Unlike existing approaches, regime-dependent service capacities are treated as
endogenous design variables, exact response-time expressions are embedded through
exponential cones, and coverage and tail-risk characteristics can be tailored
without altering the underlying model structure.

\subsection{Contributions and Organization}

This paper develops an optimization framework for the design of
reliability-oriented service systems operating under multiple regimes.
Rather than focusing on congestion evaluation, we embed the exact
response-time structure of regime-based queues directly into
capacity-design decisions, enabling explicit control of coverage,
reliability, and tail risk.

Our contributions are fourfold.
First, we model regime-dependent response-time distributions explicitly
and treat service reliability and tail performance as primary design
objectives, moving beyond average-delay or steady-state analyses.
Second, we introduce endogenous selective service-level agreement,
allowing the protected demand set to be determined within the
optimization.
Third, we incorporate CVaR-type tail control directly on response times
and show that the resulting design problems admit exact
mixed-integer exponential cone formulations.
Fourth, we analyze the computational complexity of the model,
prove its NP-hardness, and derive a
polynomial-time algorithm for relevant cases.
Together, these results yield a unified and tractable framework that
transparently captures trade-offs between capacity investment,
congestion, coverage, and extreme-delay protection.

The remainder of the paper is organized as follows.
Section~\ref{sec:performance_structure} derives the regime-based response-time expressions, characterizing the exact distributional structure induced by multiple operating regimes and establishing the analytical foundations that support the subsequent design models.
Section~\ref{sec:design_framework} introduces the optimization framework, formulates the reliability-oriented capacity design problems, establishes their NP-hardness, and identifies relevant special cases that admit polynomial-time solution algorithms.
Section~\ref{sec:experiments} presents a detailed computational study based on a New York City EMS case study, illustrating the tractability of the proposed mixed-integer exponential conic formulations and quantifying the trade-offs between capacity investment, congestion, coverage, and tail-risk protection.
Finally, Section~\ref{sec:conclusions} concludes the paper by summarizing the main findings, discussing managerial implications, and outlining promising directions for future research.

\section{Performance Structure of Multi-Regime Parallel Service Systems}
\label{sec:performance_structure}

We consider a service system motivated by a wide range of management science and
operations research applications, including facility location, emergency response
systems, centralized service centers, and capacity planning under uncertainty.
Demand originates from a finite set of points $\mathcal{A}$, representing
geographic zones, customer segments, or origin locations in a network. Service is
provided through a finite collection of service regimes indexed by
$\mathcal{R}:=\{1,\ldots,R\}$.

Each regime represents an independent service channel, operating as an
$\mathrm{M}/\mathrm{M}/1$ queue with its own congestion characteristics. Customers
are assigned to exactly one regime \emph{prior to entering any queue}, and all
waiting and service experienced by a customer take place exclusively within the
queue associated with the selected regime. Therefore, regimes do not represent
time-varying operating states of a single server, but rather parallel service
channels that coexist simultaneously. This modeling structure arises naturally in many service settings where demand is
pre-classified or routed before processing. Examples include customer support
systems with distinct service tiers or specialized teams, healthcare operations
where patients are triaged into separate treatment tracks, cloud computing
platforms that route requests to different processing pipelines, and emergency
response systems where incidents are dispatched to alternative resources based on
location or severity. In all these cases, customers commit to a specific service
channel before experiencing congestion, and performance is governed by the load
and capacity of the selected channel rather than by time-varying system states.

For each demand point $a\in\mathcal{A}$ and each service regime $r\in\mathcal{R}$,
customers originating from $a$ and routed to regime $r$ arrive according to an
independent Poisson process with rate $\lambda_{ar}\ge 0$. These arrival processes
are assumed independent across demand points and regimes. The total arrival rate
associated with demand point $a$ is $\lambda_a := \sum_{r\in\mathcal{R}} \lambda_{ar}$,
while the aggregate arrival rate to regime $r$ is $\Lambda_r := \sum_{a\in\mathcal{A}} \lambda_{ar}$.

By the superposition property of Poisson processes, arrivals to each regime form a
Poisson process with rate $\Lambda_r$. This viewpoint allows the composition of
service regimes to vary across demand points and captures heterogeneity in routing
patterns and service requirements.

The effective arrival rates $(\Lambda_r)_{r\in\RR}$ depend on how demand is distributed across operating regimes.
From a practical standpoint, the proposed methodology is deliberately flexible and can accommodate
different system architectures, routing mechanisms, and levels of operational control.
Let each demand point $a\in\A$ generate an exogenous demand rate $\lambda_a \ge 0$.
In applied settings, the following arrival rate structures naturally arise.

We consider a regime-based structure in which the system operates under regime 
$r \in \mathcal R$ with probability $\pi_r \ge 0$, where 
$\sum_{r\in\mathcal R} \pi_r = 1$. In this setting, regime $r$ receives an arrival inflow $\Lambda_r = \pi_r \sum_{a\in\mathcal A} \lambda_a$, so that the overall arrival rate $\Lambda := \sum_{a\in\mathcal A} \lambda_a$ 
is proportionally distributed across regimes according to the regime probabilities. 

This formulation captures environments in which regimes represent recurring  operational conditions—such as peak and off-peak periods, staffing configurations, or congestion states—rather than user-dependent routing policies. Each regime 
therefore corresponds to a distinct service configuration, and the overall  response-time distribution arises as a mixture of regime-specific performance  measures weighted by $\pi_r$.

Each arriving customer must first complete a deterministic transportation or
access phase before entering the service system. For a customer originating at
demand point $a$, this transportation time is denoted by $t_a\ge 0$ and is assumed to be
known. This component captures physical travel time, communication latency, or any
other exogenous delay that is independent of congestion at the service regimes.
Transportation time contributes additively to the total response time but does not
affect the queueing dynamics within the service system.

Upon completion of transportation, each customer is routed to exactly one service
regime $r$ according to the routing structure encoded by the rates
$\lambda_{ar}$. The customer then joins the corresponding queue, which operates
under a first-come, first-served discipline with unlimited waiting space. Conditional
on regime $r$, service times are exponentially distributed with rate $\mu_r>0$.
Thus, regime $r$ behaves as a stable $\mathrm{M}/\mathrm{M}/1$ queue provided that
$\Lambda_r < \mu_r$.

Figure~\ref{diagram} illustrates the system structure. Customers originate at demand
points, incur deterministic transportation times, are routed to one of the service
regimes, and then experience waiting and service according to the congestion level
of the selected regime.
\begin{figure}
{\begin{tikzpicture}[
    >=Latex,
    scale=1.2,
    demand/.style   = {draw,fill=blue!30,minimum size=6mm},
    inqueue/.style  = {draw,fill=blue!20,minimum width=4mm,minimum height=4mm},
    queuebox/.style = {draw,fill=blue!5,rounded corners=1pt},
    regimebox/.style= {draw,fill=green!5,rounded corners=1pt},
    label/.style    = {font=\tiny},
    thinarr/.style  = {->,thin},
    thickarr/.style = {->,line width=0.9pt},
    every node/.style = {align=center}
]
\coordinate (D0)  at (0,0);
\coordinate (Split) at (1,0);   
\coordinate (Q1)  at (5.1,1.6);
\coordinate (Q2)  at (5.1,0.0);
\coordinate (Q3)  at (5.1,-1.6);
\coordinate (Out) at (8.6,0);
\node[demand] (d1) at ($(D0)+(-0.8,0.9)$) {};
\node[demand] (d2) at ($(D0)+(-1.2,0.15)$) {};
\node[demand] (d3) at ($(D0)+(-0.9,-0.6)$) {};
\node[demand] (d4) at ($(D0)+(-1.6,-1.25)$) {$a$};
\node[label,anchor=east] at ($(D0)+(-1.2,1.4)$) {$\mathcal{A}$};
\node[label,anchor=west] at ($(D0)+(0.05,-1.1)$) {$t_a$};
\node[label,anchor=east] at ($(D0)+(1,-2.3)$) {Transportation Time};
\node[label] at ($(D0)+(4,-2.3)$) {Waiting Time};
\node[label] at ($(D0)+(6.8,-2.3)$) {Service Time};
\node[label] at ($(D0)+(5.6,2.2)$) {admission};
\node[draw,rounded corners=1pt,fill=gray!10,inner sep=5pt, minimum size=3.5mm] (router) at (Split) {\hspace*{0.2cm}};
\node[label,above,yshift=4pt] at (Split) {routing\\to\\regime};
\foreach \i/\y in {1/0.65, 2/0.1, 3/-0.35, 4/-0.9}{
  \draw[thinarr] (d\i.east) .. controls ($(d\i.east)+(0.35,\y)$) and ($(router.west)+(-0.45,\y)$) .. (router.west);
}
\draw[queuebox] ($(Q1)+(-2.2,-0.35)$) rectangle ($(Q1)+(0,0.35)$);
\node[label] at ($(Q1)+(-1.1,0.7)$) {Queue $r=1$};
\foreach \x in {-1.85,-1.35,-0.85,-0.35}{
  \node[inqueue] at ($(Q1)+(\x,0)$) {};
}
\draw[regimebox] ($(Q1)+(1,-0.38)$) rectangle ($(Q1)+(2.5,0.38)$);
\node[label] at ($(Q1)+(1.75,0)$) {$\text{M/M/1}$ \\ $\Lambda_1, \mu_{1}$};
\draw[queuebox] ($(Q2)+(-2.2,-0.35)$) rectangle ($(Q2)+(0,0.35)$);
\node[label] at ($(Q2)+(-1.1,0.7)$) {Queue $r=2$};
\foreach \x in {-1.85,-1.35,-0.85}{
  \node[inqueue] at ($(Q2)+(\x,0)$) {};
}
\draw[regimebox] ($(Q2)+(1,-0.38)$) rectangle ($(Q2)+(2.5,0.38)$);
\node[label] at ($(Q2)+(1.75,0)$) {$\text{M/M/1}$ \\ $\Lambda_2, \mu_2$};
\draw[queuebox] ($(Q3)+(-2.2,-0.35)$) rectangle ($(Q3)+(0,0.35)$);
\node[label] at ($(Q3)+(-1.1,0.7)$) {Queue $r=R$};
\foreach \x in {-1.85,-1.35}{
  \node[inqueue] at ($(Q3)+(\x,0)$) {};
}
\draw[regimebox] ($(Q3)+(1,-0.38)$) rectangle ($(Q3)+(2.5,0.38)$);
\node[label] at ($(Q3)+(1.75,0)$) 
{$\text{M/M/1}$ \\ $\Lambda_R, \mu_{R}$};
\node[label] at ($(Q2)+(-3.0,-0.3)$) {$\vdots$};
\draw[thickarr] ($(Q1)+(0,0)$) -- ($(Q1)+(1.0,0)$);
\draw[thickarr] ($(Q2)+(0,0)$) -- ($(Q2)+(1.0,0)$);
\draw[thickarr] ($(Q3)+(0,0)$) -- ($(Q3)+(1.0,0)$);
\draw[thickarr]
  (router.east)
  .. controls ($(router.east)+(0.8,0.9)$)
  .. node[midway,above,label] {$\Lambda_1$}
 ($(Q1)+(-2.2,0)$);
\draw[thickarr]
  (router.east)
  -- node[midway,above,label] {$\Lambda_2$}
 ($(Q2)+(-2.2,0)$);
\draw[thickarr]
  (router.east)
  .. controls ($(router.east)+(0.8,-0.9)$) 
  .. node[midway,below,label] {$\Lambda_R$}
  ($(Q3)+(-2.2,0)$);
\node[inqueue,minimum size=5.5mm,fill=red!35] (end) at (Out) {};
\node[label] at ($(end)+(0,0.6)$) {departures};
\draw[thickarr] ($(Q1)+(2.5,0)$) .. controls ($(Q1)+(3.2,0)$) and ($(end)+(-1.2,1.2)$) .. (end.west);
\draw[thickarr] ($(Q2)+(2.5,0)$) -- (end.west);
\draw[thickarr] ($(Q3)+(2.5,0)$) .. controls ($(Q3)+(3.2,0)$) and ($(end)+(-1.2,-1.2)$) .. (end.west);
\end{tikzpicture}}
{Service system structure with routing to parallel regime-specific $\mathrm{M}/\mathrm{M}/1$ queues.\label{diagram}}
{}
\end{figure}
For each regime $r\in\mathcal{R}$, the utilization factor is $0\leq \rho_r := \frac{\Lambda_r}{\mu_r}<1$. Under regime $r$, an arriving customer enters service immediately with probability
$1-\rho_r$ and experiences a positive waiting time with probability $\rho_r$. 

The following result states the distribution of the total response time of the system for each demand point $a\in\A$.

\begin{theorem}\label{th:hyperexponential}
The total (end-to-end) response time $R_a$ of the system for a demand point $a\in \A$ follows a shifted hyperexponential with cumulative distribution function
$$
\mathbb{P}(R_a \le t)
=
\begin{cases}
0 & \mbox{if $t < t_a$}, \\[6pt]
1-\displaystyle\sum_{r\in\mathcal{R}}
\pi_r \, e^{-(\mu_r-\Lambda_r)(t-t_a)} & \mbox{if $t \ge t_a$},
\end{cases}
$$
and expectation
$$
\mathbb{E}[R_a]
=
t_a + \sum_{r\in\mathcal{R}} \frac{\pi_r}{\mu_r-\Lambda_r},
$$
where mixture weights 
$\pi_r := \frac{\Lambda_r}{\sum_{r'\in\mathcal{R}} \Lambda_{r'}}$ for all $r\in\RR.$
\end{theorem}

\begin{proof}{Proof.}
A classical result for the $\mathrm{M}/\mathrm{M}/1$ queue implies that,
conditional on regime $r$, the stationary sojourn time at the facility, defined as the sum of the waiting and service times, is exponentially distributed with rate
$\mu_r-\Lambda_r$~\citep[see, e.g.,][]{kleinrock1975,Wolff1989}. Since an arbitrary arrival is routed to regime $r$ with probability
$$
\pi_r= \frac{\Lambda_r}{\sum_{r'\in\mathcal{R}} \Lambda_{r'}},
$$
the unconditional sojourn time $T$ of a random demand point follows a hyperexponential distribution with
mixture weights $\pi_r$ and exponential rates $\mu_r-\Lambda_r$. Therefore, it follows that the cumulative distribution function of $T$ is
$$
\mathbb{P}(T \le t)
=
1-\sum_{r\in\mathcal{R}} \pi_r \, e^{-(\mu_r-\Lambda_r)t}  \quad \text{ if }
\quad t \ge 0,
$$
and $0$ otherwise; and its expectation is $
\mathbb{E}[T]
=
\sum_{r\in\mathcal{R}} \frac{\pi_r}{\mu_r-\Lambda_r}$.

Finally, the total response time model of the system is defined as the sum of the deterministic transportation time plus the stochastic sojourn time ($R_a = t_a + T$). 
Since the transportation time $t_a$ is deterministic, the claim follows from shifting by $t_a$ the cumulative distribution
function and the expectation of $T$.
\end{proof}

In particular, the dependence of $\mathbb{E}[R_a]$ on $a$ arises exclusively
through the transportation time, $t_a$, while congestion effects are captured by
the regime-level service slack $\mu_r-\Lambda_r$ and the routing weights $\pi_r$.

This performance structure makes explicit how demand-specific arrival rates ($\lambda_{ar}$), service capacities ($\mu_r$), and transportation times ($t_a$) jointly
determine congestion and response-time distributions. The resulting closed-form
expressions for utilization and response-time probabilities provide a rigorous and
tractable foundation for optimization models that balance efficiency, equity, and
risk in regime-structured service system design.

Service-level agreement imposes probabilistic performance guarantees on the total
response time experienced by demand elements. For every demand $a\in\A$, fix a tolerance level $\alpha_a\in[0,1]$ and
a response-time threshold $t_a^\star\ge t_a$. The SLA requirement for element $a$
is expressed as a chance constraint of the form
\begin{equation}
\P(R_a \le t_a^\star) \ge 1-\alpha_a,
\label{eq:sla_chance}
\end{equation}
which enforces that the response-time target is met with high probability.

\begin{corollary}\label{coro:sla_explicit}
    Let $a\in\A$ be a demand point. Given a tolerance level  $0\le\alpha_a\le 1$ and a total response time threshold $t_a^\star\ge t_a$ for $a$, the corresponding {\rm SLA} requirement \eqref{eq:sla_chance} of the system admits an explicit representation as
    \begin{equation}
\sum_{r\in\mathcal{R}}
\pi_r\,e^{-(\mu_r-\Lambda_r)(t_a^\star-t_a)}
\;\le\;
\alpha_a.
\label{eq:sla_explicit}
\end{equation}
\end{corollary}
\begin{proof}{Proof.}
    The result is a direct consequence of Theorem~\ref{th:hyperexponential}.
\end{proof}

Note that the left-hand side in \eqref{eq:sla_explicit} captures the contribution of each service regime to the tail probability
of excessive delay.

In the case of a single regime, $|\RR|=1$, the expression \eqref{eq:sla_explicit} simplifies to the following lower bound for the service rate of the queue
\begin{align}\label{eq:sla_1reg}
\mu_1 \geq \Lambda_1 - \frac{\log(\alpha_a)}{t_a^\star-t_a}.
\end{align}

In many practical settings, such guarantees are not required uniformly across the
entire demand population. Instead, service-level agreement is enforced only for
a selected subset of demand elements, reflecting contractual obligations,
regulatory requirements, or prioritization policies. Let
$S\subseteq\mathcal{A}$ denote the set of demand points for which the SLA condition
\eqref{eq:sla_chance} is imposed, while no probabilistic guarantee is required
for elements in $\mathcal{A}\setminus S$. This selective enforcement mechanism
plays a central role in the design framework developed in the next section, where
the protected set $S$ is determined endogenously as part of the optimization
process.

\begin{rem}
The framework developed in this paper exploits the analytical structure induced by
Poisson arrivals and exponential service within each regime, which yields closed-form
expressions for response-time distributions and enables compact exponential cone
representations of stability, service-level, and tail-risk constraints. This choice
is motivated by tractability and by the fact that exponential service provides the
simplest setting in which congestion and routing-induced variability interact.

The approach extends naturally to regimes with Erlang-$k$ service times. In this
case, each regime remains an $\mathrm{M}/\mathrm{G}/1$ queue, expected waiting and
sojourn times follow from standard Pollaczek--Khinchine results, and the
sojourn-time distribution admits a phase-type representation that can be expressed
as a finite combination of exponential and polynomial--exponential terms~\citep[see, e.g.,][]{Asmussen2003,Gross2008}. While the
resulting constraints are algebraically more involved, they preserve the essential
structure required for convex and conic optimization. Accordingly, the
exponential-service assumption should be viewed as a modeling choice rather than a
limitation.
\end{rem}

The design problems studied in this paper are formulated as optimization models
that explicitly exploit the analytical structure of the response-time
distributions derived above. In particular, the probabilistic service-level
constraints and congestion-sensitive cost terms involve exponential and
logarithmic relationships that arise naturally from the hyperexponential
sojourn-time expressions. These relationships are represented exactly using
exponential cone programming. The exponential cone
$\mathcal{K}_{\exp}\subset\mathbb{R}^3$ is defined as
$$
\mathcal{K}_{\exp}
:=
\left\{
(x,y,z)\in\mathbb{R}^3 :
x \ge y e^{z/y},\; y>0
\right\}
\;\cup\;
\left\{
(x,0,z)\in\mathbb{R}^3 :
z\leq 0\leq x
\right\},
$$
which is a proper cone that provides an exact conic representation of exponential and logarithmic inequalities. Its dual cone is given by
$$
\mathcal{K}_{\exp}^\ast
=
\left\{
(u,v,w)\in\mathbb{R}^3 :
eu \geq - w e^{v/w},\; w < 0\right\}
\;\cup\;
\left\{
(u,v,0): u,v\ge 0
\right\}.
$$
Within this framework, the service system design problem is expressed using
linear constraints, exponential cone constraints, and a limited number of
integer restrictions. Continuous decision variables encode quantities such as
service rates, congestion margins, and auxiliary variables associated with
response-time bounds, whereas binary variables are used to model the selective
activation of service-level guarantees. Optimization models of this nature belong
to the class of mixed-integer exponential cone programs, which generalize mixed-integer linear programming by allowing exponential
relationships to be represented without approximation. Modern conic optimization solvers, as \cite{mosek},
support $\mathcal{K}_{\rm exp}$ directly, enabling the exact and tractable
integration of queueing-based performance constraints, stability conditions, and
tail-risk controls into a unified optimization model.

\section{Design Framework and Optimization Model}
\label{sec:design_framework}

This section introduces a unified optimization framework for the design of a multi-regime parallel service system under a conflict-aware reliable service coverage, concerned about cost-congestion tradeoffs and least-favored users' tail care. The framework
supports the joint selection of regime-dependent service capacities and the service quality-driven protected segment of the demand population. These decisions are made while explicitly enforcing stability conditions and controlling response-time reliability, and tail-risk performance at the population level. \\

\noindent{\bf Decision Variables}\\

\noindent In what follows, we describe the decision variables used in our model.
\begin{itemize}
\item[-] \textbf{Service rates.} The primary design variables are the regime-dependent service rates
$\mu_r$, $r\in\mathcal{R}$, which represent staffing levels, server speeds, or
effective processing capacities assigned to each service regime. These variables
directly determine congestion through the service slack $\mu_r-\Lambda_r$ and
therefore govern both expected response times and tail performance.
\item[-] \textbf{SLA-protected set.} To model the selective enforcement of service-level agreement, we introduce
binary activation variables $s_a\in\{0,1\}$ for each demand point
$a\in\mathcal{A}$. The value $s_a=1$ indicates that demand point $a$ belongs to the
protected set $S\subseteq\mathcal{A}$ for which probabilistic response-time
guarantees are imposed, whereas $s_a=0$ points out that no SLA constraint is required for that element.\\
\end{itemize}

\noindent {\bf Feasibility Requirements}\\

\noindent Feasibility requires stability in every service regime and satisfaction of the
service-level constraints for the protected set $S$. In what follows, we describe how these conditions are modeled in our formulation.
\begin{itemize}
\item[-] \textbf{Stability.}
This mandatory condition is enforced through a strict slack condition ensuring the utilization factor in each regime, $\rho_r$, is strictly under $1$,
\begin{equation}
\mu_r - \Lambda_r \ge \varepsilon,
\qquad \forall r\in\mathcal{R},
\label{eq:stability_eps}
\end{equation}
where $\varepsilon>0$ is a small numerical parameter chosen consistently with
solver tolerances. This condition guarantees that each regime operates in a
stable $\mathrm{M}/\mathrm{M}/1$ regime and avoids numerical issues associated
with boundary solutions.

\item[-] \textbf{Selective SLA enforcement.}
Let $t_a^\star$ denote the response-time threshold for demand point $a$ and let
$\alpha_a\in[0,1]$ denote it admissible violation probability. To enforce the SLA
condition \eqref{eq:sla_chance} only when $s_a=1$, we show the following result.

\begin{lemma}\label{lemma:sla}
    Let $\zeta_{ar}\geq 0$ be nonnegative auxiliary variables for $a\in \A$ and $r\in\RR$. The following set of mixed-integer exponential cone constraints,
    \begin{gather}
        \sum_{r\in\mathcal{R}} \pi_r \zeta_{ar}\le 1 + (\alpha_a-1)s_a,\label{eq:sla_switch}\\
        (\zeta_{ar}, s_a, -(\mu_r-\Lambda_r)(t_a^\star-t_a))\in \mathcal{K}_{\exp}, \label{eq:exp_link_zeta}
    \end{gather}
for all $a\in\A$ and $r\in\RR$, ensures the {\rm SLA}~\eqref{eq:sla_chance} is just required for the elements of $S := \{a \in \A: s_a=1\}$ in the system.

In the single regimen case, $|\RR|=1$, the stability~\eqref{eq:stability_eps} and {\rm SLA} enforcement \eqref{eq:sla_switch}-\eqref{eq:exp_link_zeta} can be compactly simplified to
$$
\mu_1 \geq  \left(\max\left\{\varepsilon, - \frac{\log(\alpha_a)}{t_a^\star-t_a}\right\} - \varepsilon\right) s_a + \Lambda_1 + \varepsilon, \quad \forall a \in \A.
$$
\end{lemma}

\begin{proof}{Proof.}
    Suppose first that $a\notin S$, then $s_a=0$. Hence, constraint \eqref{eq:exp_link_zeta} becomes
    $$
    (\zeta_{ar},0, -(\mu_r-\Lambda_r)(t_a^\star-t_a))\in \mathcal{K}_{\exp}, \qquad \forall r\in\RR
    $$
    which is trivially verified by nonnegativity of $\zeta_{ar}$ and nonpositivity of $-(\mu_r-\Lambda_r)(t_a^\star-t_a)$, where this latter is due to the stability condition \eqref{eq:stability_eps}. Therefore, constraint \eqref{eq:sla_switch}, which reduces to
$\sum_{r\in\RR}\pi_r \zeta_{ar}\le 1$, becomes nonbinding. Thus, no condition is imposed to demand $a$. Otherwise, if $a\in S$, then $s_a=1$. Hence, constraint \eqref{eq:exp_link_zeta} becomes
$$
e^{-(\mu_r-\Lambda_r)(t_a^\star-t_a)}\le \zeta_{ar}, \qquad \forall r\in\RR
$$
and constraint \eqref{eq:sla_switch} reduces to $\sum_{r\in\RR}\pi_r \zeta_{ar}\le \alpha_a$, what together implies that
$$
\alpha_a \ge \sum_{r\in\RR}\pi_r \zeta_{ar} \ge \sum_{r\in\RR}\pi_r e^{-(\mu_r-\Lambda_r)(t_a^\star-t_a)}.
$$
Then, \eqref{eq:sla_explicit} holds. Thus, by Corollary~\ref{coro:sla_explicit}, SLA condition \eqref{eq:sla_chance} also holds for $a$.

For the case $|\RR|=1$, the result follows straighforward by the observation in \eqref{eq:sla_1reg}.
\end{proof}

\end{itemize}

Apart from stability and selective SLA enforcement, no additional feasibility
requirements are imposed in the base model. Note that in such base model the best decision is set $s_a=0$ for all demand $a\in\A$ since there is no binding with other system criteria to enforce the service-level agreement and avoiding the emptyness of the protected set $S$. Thereby, a minimum coverage requirement will be considered below, along with a cost-congestion tradeoff, and tail-risk control. Such assessment criteria will be incorporated to the base model through objective functions and additional constraints, as described next.\\

\noindent{\bf Assessment Criteria}\\

\noindent The proposed framework supports a flexible family of objectives and constraints
that capture economic efficiency, congestion robustness, service coverage, and
tail-risk control within a single optimization model.
\begin{itemize}

\item[-] \textbf{Cost and congestion robustness.}
Operating capacity is penalized through a linear cost term $c^t \mu:=\sum_{r\in\mathcal{R}} c_r \mu_r$, where
$c\in\mathbb{R}^{R}_+$ represents regime-specific capacity costs. To
discourage operation close to instability, we also incorporate a logarithmic
congestion penalty of the form $-\log(\mu_r-\Lambda_r)$, which increases rapidly
as the service slack approaches zero. Introducing auxiliary variables $\nu_r$, this penalty can be modeled via the exponential cone constraints
\begin{equation}
(\mu_r-\Lambda_r,\,1,\,\nu_r)\in\mathcal{K}_{\exp},
\qquad \forall r\in\mathcal{R}.
\label{eq:exp_log}
\end{equation}
These expressions mean $\mu_r-\Lambda_r\ge e^{\nu_r}$, so that $\log(\mu_r-\Lambda_r)\ge \nu_r$ for all regime $r\in\RR$. Thus, $-\nu_r$ provides an upper bound on the logarithmic
congestion penalty $-\log(\mu_r-\Lambda_r)$. A
representative objective function is therefore given by
\begin{equation}
\min \quad c^{t}\mu - \kappa \sum_{r\in\mathcal{R}} \nu_r,
\label{eq:obj_tradeoff}
\end{equation}
where $\kappa\ge 0$ controls the emphasis placed on congestion robustness relative
to capacity cost.

\item[-] \textbf{Coverage.}
To prevent service guarantees from being concentrated on a negligible subset of
the demand population, we impose a minimum coverage requirement of the form
\begin{equation}
\sum_{a\in\mathcal{A}} s_a \;\ge\; \lceil \beta |\mathcal{A}| \rceil,
\label{eq:coverage}
\end{equation}
where $\beta\in(0,1]$ specifies the minimum fraction of demand points for which
probabilistic service-level agreements must be enforced. The parameter $\beta$
thus provides a direct and interpretable control on the breadth of service
commitments. Larger values of $\beta$ correspond to near-uniform protection
across the population, whereas smaller values allow resources to be focused on a
core subset of demand elements where guarantees can be delivered most
cost-effectively.

The framework readily accommodates extensions of \eqref{eq:coverage} to account
for heterogeneity in demand importance or economic value. For example, letting
$w_a>0$ denote a weight associated with element $a$, a weighted coverage
constraint of the form
$$
\sum_{a\in\mathcal{A}} w_a s_a \;\ge\; \beta \sum_{a\in\mathcal{A}} w_a
$$
ensures protection of a prescribed fraction of total weighted demand. Moreover,
$\beta$ itself may be treated as a decision variable, enabling the model to
determine the maximum achievable coverage under a fixed budget or congestion
limit.

\item[-] \textbf{Conflict-aware service guarantees.}
Meanwhile the coverage constraint \eqref{eq:coverage} ensures that a minimum fraction of demand elements is protected, in practice it may not be operationally credible to enforce strict service-level guarantees for all selected elements simultaneously. In congested service systems, guarantees compete for localized and highly correlated resources, such as nearby servers, dispatch units, or priority capacity, particularly under adverse demand realizations that drive tail-risk
measures.

To prevent over-commitment of such resources, we introduce conflict constraints
that restrict simultaneous protection of incompatible demand elements. These
constraints are modeled via a conflict graph $G=(\A,\mathcal{E})$, where
each node $a\in\A$ represents a demand element and an edge
$\{a,a'\}\in\mathcal{E}$ indicates that elements $a$ and $a'$ cannot both be
assigned strict service guarantees without violating operational feasibility.
The conflict constraints take the simple linear form
\begin{equation}
s_a + s_{a'} \;\le\; 1,
\qquad \forall \{a,a'\}\in\mathcal{E}.
\label{eq:conflict}
\end{equation}
The conflict graph can be constructed using empirical information on spatial
proximity, temporal overlap, shared service resources, or correlation of response
times under congestion. 

By explicitly modeling such incompatibilities, constraint \eqref{eq:conflict}
ensures that the protected set of demand elements remains operationally credible,
preserves the integrity of the service guarantees under tail-risk control, and
prevents the model from relying on unrealistic concentration of guarantees in
highly congested regions of the system, in line with the congestion interaction
effects discussed in \cite{feldman2012conflicting}.

\item[-] \textbf{Tail-risk control.}
    Note that although \eqref{eq:coverage} is selectively enforced on a $\beta$ fraction of the users, it might also be important to control the performance experienced by the remaining $1-\beta$ proportion who may face consistently high response times.
To prevent unbounded degradation among these users, one may incorporate a complementary robust condition based on the average of the expected time of those customers who could \emph{violate} the {\rm SLA} constraints ($\A\setminus S)$. To adapt the formulation to this feature, it considers a time threshold value $\Gamma \geq 0$ which will upper bound the expected response times of the users in $\A\setminus S$, and adds the nonlinear constraint
$$
\frac{1}{\lfloor(1-\beta)|\A|\rfloor}\sum_{a\in\A}(1-s_a)\left(t_a+\sum_{r\in\RR}\frac{\pi_r}{\mu_r-\Lambda_r}\right)\leq \Gamma.
$$
Despite its nonlinearity, it can be replaced by $|\RR|$ exponential cone constraints, $3|\A||\RR|+1$ linear constraints, and $(|\A|+1)|\RR|$ continuous auxiliary variables, keeping convexity in its continuous relaxation. Nevertheless, this constraint just ensures proper robustness in case all demand in $\A\setminus S$ actually violate the {\rm SLA} constraints, i.e., $\mathbb{P}(R_a\le t^\star_a)<1-\alpha_a, \forall a\in\A\setminus S$. A careful reader will notice this forced-violation requirement would suppose a nonsense for the aim of guaranteeing the quality of the service. Because of that, we will tackle robustness in expected response times via conditional value-at-risk ({\rm CVaR}), which allows us to split the coverage of the strictly satisfied demand and the tail-risk control. This is performed through the notion of CVaR
$$
{\rm CVaR}_\gamma\left(\E[R_a]_{a\in\A}\right) : = \frac{1}{\lfloor (1-\gamma)|\A|\rfloor} \sum_{i=1}^{\lfloor (1-\gamma)|\A|\rfloor} \E[R_{(i)}]
$$
where $\E[R_{(1)}] \geq \E[R_{(2)}] \geq \cdots \geq \E[R_{(|\A|)}]$ is the (non-increasing) sorted list of expected total response times of each of the demands, and $0< 1-\gamma\leq 1$ is the proportion of the demand with the worst expected total response times ($\gamma$ could be $\beta$ or could be decoupled from the coverage condition). That means CVaR$_\gamma\circ \E$ is the average of the expected response time of the demands whose expected total response are among the $1-\gamma$ worst fraction.

\begin{lemma}\label{lemma:cvar}
    Let $\A$ be the set of demands, $0\leq \gamma <1$ be a proportion, and $\Gamma\geq 0$ be a time threshold. The condition ${\rm CVaR}_\gamma\left(\E[R_a]_{a\in\A}\right)\leq \Gamma$ can be expressed with the following set of exponential cone constraints
    \begin{align}
  r_a \geq t_a + \sum_{r\in\mathcal{R}} \pi_r\tau_r, && \forall a\in\mathcal{A},
\label{eq:r_linear}\\
U_a \ge r_a-\eta, && \forall a\in\A\label{eq:var}\\
\eta + \frac{1}{\lfloor (1-\gamma)|\A|\rfloor}\sum_{a\in\mathcal{A}} U_a \le \Gamma, &&
\label{eq:cvar_constraint}\\
\nu_r + \ell_r\ge 0, &&\forall r\in\mathcal{R}, \label{eq:log_sum}\\
(\tau_r,\,1,\,\ell_r)\in \mathcal{K}_{\exp}, &&\forall r\in\mathcal{R},
\label{eq:log_v}
\end{align}
where $(r,\tau, \eta,\ell, U)$ are nonnegative auxiliary variables.
\end{lemma}

\begin{proof}{Proof.}
First, we show constraints \eqref{eq:log_sum} and \eqref{eq:log_v} ensure $(\mu_r-\Lambda_r)\tau_r\geq 1$. On the one hand, by \eqref{eq:log_v} we have $\tau_r\ge e^{\ell_r}$ therefore $\log(\tau_r)\ge \ell_r$ and by definition of the auxiliary variable $\nu_r$ \eqref{eq:exp_log} $\log(\mu_r-\Lambda_r)\ge \nu_r$. It implies using \eqref{eq:log_sum} that $\log(\mu_r-\Lambda_r)+\log(\tau_r)\ge 0$, hence $(\mu_r-\Lambda_r)\tau_r\geq 1$ holds.

On the other hand, we have that,
\begin{equation}\label{eq:auxiliary_expectations}
r_a \geq t_a + \sum_{r\in\mathcal{R}} \pi_r\tau_r \geq t_a + \sum_{r\in\mathcal{R}} \frac{\pi_r}{\mu_r-\Lambda_r}=\E[R_a],
\end{equation}
for all $a\in\A$, where the first inequality is \eqref{eq:r_linear}, the second one is due to $(\mu_r-\Lambda_r)\tau_r\geq 1$, and the last equality comes from Theorem\ref{th:hyperexponential}.

Finally, constraints \eqref{eq:var} and \eqref{eq:cvar_constraint} bound the largest $1-\gamma$ proportion of $r_a$ variables \cite[see, e.g.,][for the formulation of the $k$-sum]{Ogryczak2003}. Thus, we obtain
\begin{align*}
    & \Gamma \geq  \eta + \frac{1}{\lfloor (1-\gamma)|\A|\rfloor}\sum_{a\in\mathcal{A}} U_a, && [\text{by }~\eqref{eq:cvar_constraint}]\\
    & \geq  \eta + \frac{1}{\lfloor (1-\gamma)|\A|\rfloor}\sum_{\substack{a\in\mathcal{A}:\\ r_a\ge \eta}}(r_a-\eta), && [\text{by } U\ge 0, \eqref{eq:var}] \\
    & = \frac{1}{\lfloor (1-\gamma)|\A|\rfloor} \sum_{i=1}^{\lfloor (1-\gamma)|\A|\rfloor} r_{(i)} && \\
    & \geq \frac{1}{\lfloor (1-\gamma)|\A|\rfloor} \sum_{i=1}^{\lfloor (1-\gamma)|\A|\rfloor} \E[R_{(i)}], && [\text{by }~\eqref{eq:auxiliary_expectations}]\\
    & =: {\rm CVaR}_\gamma\left(\E[R_a]_{a\in\A}\right) &&
\end{align*}
as it was claimed.
\end{proof}

\end{itemize}

The parameters $(\kappa, \beta, \gamma, \Gamma)$ jointly control congestion robustness, service coverage, and tail-risk aversion, allowing the designer to explore a broad
range of reliability and efficiency trade-offs within a single unified
optimization framework.\\

\noindent{\bf Mathematical Optimization Model}\\

\noindent Summarizing the modeling and design developments introduced above, the flexible
service system design problem can be expressed as a mixed-integer exponential
cone optimization model that integrates stability~\eqref{eq:stability_eps}, selective service-level
guarantees~\eqref{eq:sla_switch}-\eqref{eq:exp_link_zeta}, cost-congestion robustness~\eqref{eq:exp_log}-\eqref{eq:obj_tradeoff}, minimum coverage requirement~\eqref{eq:coverage}, conflict-aware service~\eqref{eq:conflict}, and tail-risk control~\eqref{eq:r_linear}-\eqref{eq:log_v} 
within a single formulation. The decision variables include regime-dependent
service capacities, SLA activation indicators, and
auxiliary variables required to represent queueing performance and risk measures
exactly through linear and exponential cone constraints. The full resulting model is denoted as  \ref{SSD(C/TR)-ECP} and can be explicitly found in Appendix~\ref{ap:fullmodel}.

The parameters $((\alpha_a)_{a\in \A},\kappa,\beta,\gamma,\Gamma)$ provide interpretable levers that
shape the resulting system design and allow the decision-maker to explore
trade-offs between coverage, reliability, congestion robustness, and tail-risk
control. The coverage parameter $\beta$ controls the minimum fraction of demand
points for which service-level agreement is enforced. Larger values of $\beta$
promote more uniform protection across the demand population, typically requiring
higher service capacities and leading to more conservative designs, while smaller
values allow the model to concentrate resources on a core subset of demand points
for which guarantees can be delivered most cost-effectively.

The SLA tolerance parameters $\alpha_a$ governs the strictness of the probabilistic
response-time guarantees. Smaller values of $\alpha_a$ impose tighter reliability
requirements by limiting the allowable probability of SLA violation for demand $a$, which in
turn increases the required service slack and capacity investment. Larger values
of $\alpha_a$ relax these guarantees, allowing higher violation probabilities and
yielding less conservative designs.

Tail-risk considerations are controlled through the parameters $\gamma$ and $\Gamma$. The first one sets the proportion of the out-of-tail population, considering the remaining $1-\gamma$ proportion as the least-favored users to be cared for (typically, one could link the coverage with the tail control taking $\gamma=\beta$). The second one bounds the conditional value-at-risk of the expected response times across the
demand population. Smaller values of $\Gamma$ enforce stricter control of extreme
delays by limiting the average response time experienced by the worst-performing
subset of demand points, whereas larger values permit greater variability in tail
performance and prioritize cost efficiency.

The congestion-robustness parameter $\kappa$ balances operating cost against
proximity to instability. Increasing $\kappa$ penalizes designs that operate with
small service slack, encouraging capacity buffers and smoother performance under
demand fluctuations, while $\kappa=0$ yields cost-minimizing solutions that may
operate closer to critical utilization levels.

Together with the stability margin $\varepsilon$, which enforces a strict buffer
away from instability and improves numerical robustness, these parameters define
a flexible design space. By varying $((\alpha_a)_{a\in \A},\kappa,\beta,\gamma,\Gamma)$, the framework
can generate a spectrum of solutions ranging from cost-efficient designs with
limited coverage and relaxed reliability requirements to highly robust systems
with broad service guarantees and controlled tail performance. This parametric
structure supports sensitivity analysis and policy evaluation, enabling decision
makers to align operational performance with strategic reliability objectives.

\begin{theorem}\label{th:np-hardness}
    If $\mathcal{E}\neq \emptyset$, the mixed-integer exponential cone program \ref{SSD(C/TR)-ECP} is {\rm NP}-hard, even for a single regime, $|\RR|=1$.
\end{theorem}
 
\begin{proof}{Proof.}
    Let us consider the simplified version of the problem where $\RR=\{1\}$ with
$\pi_1=1$. One can set $\varepsilon=-\min_{a\in\A}\frac{\log(\alpha_a)}{t^\star_a-t_a}>0$, then take $\mu_1=\Lambda_1+\varepsilon$ and enforce $\zeta_{a1}=e^{-\varepsilon(t^\star_a-t_a)}$. It is easy to check this choice is feasible for the problem. Finally, choose $\kappa=c_1=0$ and pick $(\tau_1,\ell_1,\nu_1,\eta,r,U)$ so that the remaining constraints are trivially satisfied (e.g., take $\Gamma$ sufficiently large and $\tau_1=0$).

Under these parameter choices, the only nontrivial constraints on $s$ are exactly
$$
\sum_{a\in\mathcal A}s_a\ge \left\lceil\beta |\A|\right\rceil 
\qquad\text{and}\qquad
s_a+s_{a'}\le 1,\ \ \forall \{a,a'\}\in\mathcal{E},
$$
so a feasible solution exists if and only if $G=(\A,\mathcal{E})$ contains an independent set of
size at least $\left\lceil\beta |\A|\right\rceil$. Hence, the problem reduces to the independent set problem, whose feasibility is well-known to be NP-complete~\citep[see, e.g.,][]{garey2002computers}. Thus, the
optimization problem is NP-hard.
\end{proof}

\subsection{Solution Approach}\label{sec:benders}

The problem described above admits a natural decomposition structure in which
the binary variables $s$ determine the subset of demand elements for which
service-level guarantees are enforced, while all remaining variables enter
through a continuous convex optimization problem.
Specifically, the problem can be equivalently written as
\begin{align*}
\min \ & \theta\\
\text{s.t. } 
& \sum_{a\in\mathcal{A}} s_a \ge \lceil \beta |\mathcal{A}| \rceil,\\
& s_a + s_{a'} \le 1, \qquad \forall \{a,a'\}\in\mathcal{E};\\
& \theta \ge \Theta(s),\\
& s_a \in \{0,1\}, \qquad \forall a\in\mathcal{A};
\end{align*}
where $\Theta(s)$ denotes the optimal value of a (continuous) exponential cone
program obtained by fixing $s$ (see Appendix~\ref{app:benders}).

For any fixed $s$, the subproblem defining $\Theta(s)$ is a convex optimization
problem involving linear constraints and exponential cones. Under standard
regularity conditions ensuring feasibility and strong duality, the value
function $\Theta(\cdot)$ is convex over the domain $[0,1]^{\mathcal{A}}$.
This structure enables the use of Benders decomposition, in which the value function $\Theta(\cdot)$ 
is approximated from below by a sequence of supporting hyperplanes derived from
dual information of the subproblem.

The resulting algorithm alternates between a mixed-integer linear  master problem in the
$s$-variables and a continuous exponential cone subproblem.
At each iteration, the master problem proposes a candidate protection pattern,
which is evaluated by solving the subproblem. If the subproblem is feasible,
dual multipliers yield a valid Benders optimality cut that is added to the master.
If the subproblem is infeasible, a feasibility cut excluding the current
solution is generated. The procedure terminates when the optimality gap falls
below a prescribed tolerance.

The detailed formulation of the subproblem and the derivation of Benders cuts are provided in Appendix~\ref{app:benders}.



Observe that the master problem has a finite number of feasible binary vectors
$s\in\{0,1\}^{\mathcal A}$.
In the described decomposition procedure, each iteration adds a cut that excludes the current solution from being optimal
unless it is globally optimal.
Hence, the procedure cannot revisit the same solution indefinitely and must
terminate after finitely many iterations.

\subsection{A Polynomial Time Procedure for the Conflict-Free Case}

As shown in Theorem~\ref{th:np-hardness}, problem \ref{SSD(C/TR)-ECP} is NP-hard whenever
the conflict graph satisfies $\mathcal E \neq \emptyset$.
In this section, we explore a structurally important special case in which no
conflict constraints are present, that is, $\mathcal E = \emptyset$.
Although simpler, this setting remains relevant in practice and provides
valuable insight into the structure of optimal service guarantee policies.

In many service systems, including emergency medical services, explicit conflicts
between demand elements arise primarily due to localized congestion effects.
When demand is sufficiently dispersed in time or space, or when the planning
horizon is aggregated at a high level, it is reasonable to ignore such pairwise
incompatibilities and focus instead on system wide capacity and tail-risk
tradeoffs.
The conflict-free formulation thus constitutes a useful benchmark and enables
analytical characterization of optimal protection strategies. So, throughout this section, we assume $\mathcal E = \emptyset$.

We begin by establishing a monotonicity property of the value function
$\Theta(\cdot)$.

\begin{proposition}\label{prop:monotonicity}
Let $s,s' \in \{0,1\}^{\mathcal A}$ be feasible vectors such that
$s \le s'$ componentwise. Then $\Theta(s) \le \Theta(s')$.
\end{proposition}

\begin{proof}{Proof.}
Assuming that $\mathcal{E}=\emptyset$, the only constraints in the subproblem defining $\Theta(\cdot)$ that depend on $s$
are the coverage constraint
$$
\sum_{a \in \mathcal A} s_a \ge \left\lceil \beta |\A|\right\rceil,
$$
and the service-level agreement constraints
$$
\sum_{r \in \mathcal R} \pi_r e^{-(\mu_r-\Lambda_r)\Delta_a}
\le 1 + (\alpha_a - 1)s_a,
\qquad \forall a \in \mathcal A,
$$
where $\Delta_a := t_a^\star - t_a \ge 0$.
Since $\alpha_a < 1$, replacing $s$ by $s'$ tightens these constraints
componentwise. Therefore, the feasible region of the subproblem under $s'$ is
contained in that under $s$, and the optimal value cannot decrease.
\end{proof}

An immediate consequence of this monotonicity is that the coverage constraint can
be enforced at equality.

\begin{corollary}
There exists an optimal solution satisfying
$$
\sum_{a \in \mathcal A} s_a = \left\lceil \beta |\A|\right\rceil.
$$
\end{corollary}

\begin{proof}{Proof.}
This is a direct consequence of Proposition~\ref{prop:monotonicity}.
\end{proof}

In what follows, we additionally assume that $\alpha_a=\alpha$ for all
$a\in\mathcal A$.
That means all demand elements are subject to a common service-level confidence.

This assumption is standard in many practical service systems, where
service-level agreement is defined by uniform policy targets rather than
individualized confidence requirements.
In particular, in emergency service systems and public-sector operations,
response-time guarantees are typically specified by system-wide confidence
levels reflecting regulatory standards, contractual obligations, or equity
considerations.
Under this uniform-confidence setting, the conflict-free formulation admits a
polynomial-time solution and allows for a clear structural characterization of
optimal protection policies.

\begin{theorem}
When $\mathcal E=\emptyset$ and $\alpha_a=\alpha$ for all $a\in\mathcal A$,
problem \ref{SSD(C/TR)-ECP} is solvable in polynomial time.
\end{theorem}

\begin{proof}{Proof.}
Let $\Delta_a:=t_a^\star-t_a$ for all $a\in\mathcal A$, and let $\sigma$ be a
permutation of $\mathcal A$ such that
$$
\Delta_{\sigma(1)} \ge \Delta_{\sigma(2)} \ge \cdots \ge \Delta_{\sigma(|\mathcal A|)}.
$$
The permutation $\sigma$ can be computed in
$\mathcal O(|\mathcal A|\log|\mathcal A|)$ time.

For any feasible service rate vector $\mu$ with $\mu_r>\Lambda_r$ for all
$r\in\mathcal R$, the function
$$
f(\Delta):=\sum_{r\in\mathcal R}\pi_r e^{-(\mu_r-\Lambda_r)\Delta}
$$
is non-increasing in $\Delta$.
Hence, if the service-level agreement constraint is satisfied for some index $a$,
it is also satisfied for any index $a'$ with $\Delta_{a'}\ge\Delta_a$.

It follows that, among all selections of
$\lceil\beta|\mathcal A|\rceil$ protected elements, the least restrictive choice
is obtained by selecting the indices with the largest values of $\Delta_a$.
Therefore, an optimal solution is obtained by fixing
$$
s^\star_a=
\begin{cases}
1, & \text{if } \sigma(a)\le \lceil\beta|\mathcal A|\rceil,\\
0, & \text{otherwise},
\end{cases}
$$
and solving the resulting continuous exponential cone program with $s=s^\star$.
Since this subproblem is convex, it can be solved in polynomial time by interior
point methods.
\end{proof}

\begin{algorithm}[t]
\caption{Polynomial time solution for the conflict-free case and uniform confidence levels.}
\label{alg:poly_no_conflicts}
\begin{algorithmic}[1]
\Input Set $\mathcal A$, integer $k=\lceil \beta|\A|\rceil$, slacks
$\Delta_a=t_a^\star-t_a\ge 0$ for all $a\in\mathcal A$, and the exponential cone
subproblem defining $\Theta(s)$.

\State Compute $\Delta_a=t_a^\star-t_a$ for all $a\in\mathcal A$.
\State Sort $\mathcal A$ in non-increasing order of $\Delta_a$ and let $\sigma$ be
the resulting permutation.
\State Define $s^\star\in\{0,1\}^{\mathcal A}$ by
$$
s^\star_{\sigma(i)}=
\begin{cases}
1, & i=1,\ldots,k,\\
0, & i=k+1,\ldots,|\mathcal A|.
\end{cases}
$$
\State Solve the exponential cone program defining $\Theta(s^\star)$
to obtain $\mu^\star$.
\State \Return $(s^\star,\mu^\star)$.
\Output An optimal solution $(s^\star,\mu^\star)$ to
\ref{SSD(C/TR)-ECP}.
\end{algorithmic}
\end{algorithm}

Note that the ordering argument above relies critically on the assumption of a uniform
confidence level.
When the parameters $\alpha_a$ are heterogeneous, the restrictiveness of a
service-level agreement depends jointly on the temporal slack $\Delta_a$ and the
tolerance parameter $\alpha_a$.
While larger values of $\Delta_a$ make a constraint easier to satisfy, smaller
values of $\alpha_a$ impose stricter probabilistic guarantees.
As a result, neither parameter alone induces a total ordering, and only a
partial order based on simultaneous dominance in both $\Delta_a$ and $\alpha_a$
can be established.
Consequently, in the heterogeneous confidence case there is no universal scalar
sorting rule that determines the optimal protected set, and the polynomial-time
selection result no longer applies. However, if there exists a subset of $\lceil \beta|\A|\rceil$ demand elements whose pairs
$(\Delta_a,\alpha_a)$ dominate those of all remaining elements componentwise,
then the same polynomial-time procedure applies by selecting this dominating
subset, since their associated service-level agreement constraints are uniformly
less restrictive.

Summarizing, note that the result above highlights a sharp contrast between the general conflict-aware
formulation and the conflict-free case.
When no incompatibilities are present, the optimal protection policy admits a
simple threshold structure based on the temporal slack $\Delta_a$.
Introducing conflict constraints destroys this ordering property and leads to
combinatorial selection effects that render the problem NP-hard, as established
earlier.

 \section{Numerical Experiments: NYC EMS Case Study}
\label{sec:experiments}

We assess the proposed exponential conic optimization framework using real incident-level
data from the New York City Emergency Medical Services (EMS) system (\url{data.cityofnewyork.us}). The goal of the experiments
is twofold. First, we evaluate the computational performance of the proposed compact
formulation against the Benders decomposition algorithm developed in this paper.
Second, we analyze the managerial implications of the optimized service rates under
different congestion and tail-risk profiles.

The dataset is obtained from the NYC Open Data EMS Incident Dispatch repository from 2025 and
contains individual emergency calls with detailed information about each incident. We construct our input dataset and parameters for our approach based on the following elements.
\begin{itemize}
    \item[-] \textbf{Boroughs.} We partition the whole dataset based on the borough where the incident occurred. We focus our analysis on the four most populated NYC boroughs: Manhattan, Bronx, Brooklyn, and Queens. 
    \item[-] \textbf{Regimes.} We consider that the regimes represent different initial call types from the EMS incidents, namely \texttt{CARDBR} (difficulty breathing with chest pain), \texttt{INJURY} (non-critical), \texttt{SICK}, and \texttt{UNC}	(unconscious patient).
    \item[-] \textbf{Time window scenarios.} We restrict the datasets to different time window scenarios within 2025 to generate datasets of different sizes and characteristics. In particular, we focus on: two scenarios based on incidents that occurred in night/peak hours within a single day; three scenarios based on single but whole days; and three scenarios for three days in a row, a week, and a year, respectively. The precise starting and ending instants (date/time) of each time window are detailed in Table~\ref{tab:profiles_and_scenarios} (right) and the sizes of each of the constructed datasets is shown in Table \ref{tab:ems_pi_mu_rho} (column $|\A|$).
    \item[-] \textbf{Parameter profiles.} We consider several service-level profiles reflecting different trade-offs 
between coverage, congestion, and tail-risk aversion. 
Each profile is characterized by: 
(i) a target coverage level $\beta$, which also determines the proportion of 
out-of-tail demand used in the conditional value-at-risk (CVaR), i.e., 
we set $\gamma = \beta$; 
(ii) a risk-aversion parameter $\psi$ controlling the CVaR threshold 
$\Gamma(\psi)$; 
(iii) a slack parameter $\alpha$ regulating the tolerance on response-time 
compliance; 
(iv) a regularization parameter $\kappa$ penalizing excessive service capacity; 
(v) a scaling parameter $\varphi$ defining the upper bound, 
$t_a^\star(\varphi)$ in the SLA constraints as a percentage of the 
actual completion time $t_a^c$, namely,
$t_a^\star = \varphi\, t_a^c$. 
The profiles range from strict high-coverage configurations to more flexible 
designs that allow greater congestion in exchange for reduced capacity 
requirements. The detailed parameter combinations are reported in 
Table~\ref{tab:profiles_and_scenarios} (left).

The six profiles represent distinct operational philosophies, namely:
\begin{itemize}
\item[·] \textbf{BAL (balanced).} A high coverage level ($\beta=0.95$) combined with moderate slack and neutral SLA scaling ($\varphi=1$) yields a benchmark configuration that balances reliability and capacity investment without emphasizing extreme tail protection.
\item[·] \textbf{COV (coverage-oriented).} A lower coverage target ($\beta=0.80$) but a slightly relaxed SLA bound ($\varphi=1.03$) shifts the focus toward average compliance rather than strict worst-case performance.
\item[·] \textbf{HARD (stringent compliance).} A very small tolerance ($\alpha=0.02$) and tighter SLA scaling ($\varphi=0.95$) impose strict response-time requirements and typically lead to higher required service capacity.
\item[·] \textbf{REL (reliability-focused).} Lower coverage ($\beta=0.70$) but strong tail aversion ($\psi=1.25$) emphasizes protection against extreme delays rather than universal compliance.
\item[·] \textbf{TAIL+ (tail protection under tight compliance).} Very small tolerance ($\alpha=0.01$) combined with neutral SLA scaling ($\varphi=1$) enforces tight compliance while focusing on delay dispersion control.
\item[·] \textbf{TIGHT+ (aggressive SLA tightening).} A substantially reduced SLA bound ($\varphi=0.88$) makes response-time targets significantly stricter than historical completion times and typically requires substantial capacity upgrades.
\end{itemize}

Overall, the profiles allow decision-makers to explicitly quantify the trade-off between investment in capacity (through $\kappa$ and the induced service rates) and reliability guarantees (through $\beta$, $\psi$, $\alpha$, and $\varphi$). Stricter SLA scaling and smaller slack parameters increase robustness and punctuality but at the expense of higher operational costs, while more relaxed configurations reduce capacity needs at the cost of greater congestion risk.
\item[-] \textbf{Conflict graphs.} From each filtered dataset, we build an undirected conflict graph whose nodes correspond to individual EMS incidents. Two incidents are connected by an edge if they occur sufficiently close in time and space, capturing short-term competition for shared resources. Incidents are grouped by dispatch area and ordered by time. Within each group, we connect any two incidents whose occurrence times differ by at most one minute. The resulting graph captures short temporal bursts of demand within the same operational area.
\end{itemize} 

\begin{table}[t]
\centering
\caption{Service-level profiles (left) and observation windows (right) used in the experiments.}
\label{tab:profiles_and_scenarios}
\begin{minipage}[t]{0.5\textwidth}
\centering
\small
\renewcommand{\arraystretch}{1.05}
\begin{tabular}{lccccc}
\toprule
\textbf{Profile} & \textbf{$\beta$} & \textbf{$\alpha$} & \textbf{$\psi$} & \textbf{$\varphi$} & \textbf{$\kappa$} \\
\midrule
BAL    & 0.95 & 0.05 & 1.10 & 1.00 & 0.10 \\
COV    & 0.80 & 0.07 & 1.15 & 1.03 & 0.10 \\
HARD   & 0.90 & 0.02 & 1.05 & 0.95 & 0.10 \\
REL    & 0.70 & 0.10 & 1.25 & 1.08 & 0.10 \\
TAIL+  & 0.75 & 0.01 & 1.00 & 1.00 & 0.10 \\
TIGHT+ & 0.80 & 0.05 & 1.08 & 0.88 & 0.10 \\
\bottomrule
\end{tabular}
\end{minipage}~
\begin{minipage}[t]{0.5\textwidth}
\centering
\small
\renewcommand{\arraystretch}{1.05}
\begin{tabular}{lcc}
\toprule
\textbf{Scenario} & \textbf{Start} & \textbf{End} \\
\midrule
night & 2025-03-03 20:00 & 2025-03-04 08:00 \\
peak  & 2025-03-03 08:00 & 2025-03-03 20:00 \\
D1 & 2025-03-01 00:00 & 2025-03-02 00:00 \\
D2 & 2025-03-02 00:00 & 2025-03-03 00:00 \\
D3 & 2025-03-03 00:00 & 2025-03-04 00:00 \\
3days & 2025-03-01 00:00 & 2025-03-04 00:00 \\
week  & 2025-03-01 00:00 & 2025-03-08 00:00 \\
year  & 2025-01-01 00:00 & 2026-01-01 00:00 \\
\bottomrule
\end{tabular}
\end{minipage}

\end{table}


\begin{table}[t]
\centering
\small
\caption{Summary of sizes, mixture probabilities, and conflict edges for each borough and scenario in our experiments.}
\label{tab:ems_pi_mu_rho}
\begin{adjustbox}{width=\linewidth}
\begin{tabular}{llccc @{\hspace{1.2em}} llccc}
\toprule
Borough & Scenario & $|\mathcal{A}|$ & $|\mathcal{E}|$ & $\pi$
& Borough & Scenario & $|\mathcal{A}|$ & $|\mathcal{E}|$ & $\pi$ \\
\midrule
\multirow{8}{*}{Bronx}  & night & 110 & 3 & \makecell{$(0.3182,0.2455,0.3273,0.1091)$}
& \multirow{8}{*}{Manhattan} &  night & 101 & 2 & \makecell{$(0.3168,0.2574,0.2772,0.1485)$} \\

 & peak & 178 & 10 & \makecell{$(0.3539,0.2303,0.2079,0.2079)$}
&  & peak & 165 & 8 & \makecell{$(0.2061,0.3697,0.2424,0.1818)$} \\

 & D1 & 243 & 6 & \makecell{$(0.2675,0.2387,0.3004,0.1934)$}
&  & D1 & 263 & 10 & \makecell{$(0.1711,0.3156,0.2928,0.2205)$} \\

 & D2 & 236 & 6 & \makecell{$(0.3390,0.3051,0.2331,0.1229)$}
&  & D2 & 233 & 2 & \makecell{$(0.2275,0.3219,0.2833,0.1674)$} \\

 & D3 & 287 & 15 & \makecell{$(0.3693,0.2125,0.2369,0.1812)$}
 & & D3 & 259 & 9 & \makecell{$(0.2394,0.3359,0.2548,0.1699)$} \\

& 3days & 576 & 19 & \makecell{$(0.3455,0.2535,0.2413,0.1597)$}
& & 3days & 526 & 11 & \makecell{$(0.2262,0.3213,0.2757,0.1768)$} \\

 & week & 1781 & 73 & \makecell{$(0.3088,0.2718,0.2673,0.1522)$}
&  & week & 1887 & 55 & \makecell{$(0.2194,0.3153,0.2719,0.1934)$} \\

\midrule
\multirow{8}{*}{Brooklyn}  & night & 104 & 5 & \makecell{$(0.2596,0.2981,0.3558,0.0865)$}
& \multirow{8}{*}{Queens} &  night & 86 & 0 & \makecell{$(0.1977,0.3372,0.2674,0.1977)$} \\

 & peak & 195 & 5 & \makecell{$(0.2821,0.2974,0.2410,0.1795)$}
&  & peak & 129 & 2 & \makecell{$(0.2403,0.2791,0.3101,0.1705)$} \\

 & D1 & 321 & 9 & \makecell{$(0.2087,0.3115,0.3084,0.1713)$}
&  & D1 & 209 & 6 & \makecell{$(0.2057,0.3254,0.3014,0.1675)$} \\

 & D2 & 299 & 11 & \makecell{$(0.2408,0.2575,0.2977,0.2040)$}
&  & D2 & 226 & 5 & \makecell{$(0.1991,0.3717,0.2699,0.1593)$} \\

 & D3 & 316 & 8 & \makecell{$(0.2911,0.2943,0.2690,0.1456)$}
&  & D3 & 228 & 3 & \makecell{$(0.2325,0.3246,0.2719,0.1711)$} \\

& 3days & 684 & 24 & \makecell{$(0.2617,0.2822,0.2880,0.1681)$}
& & 3days & 495 & 10 & \makecell{$(0.2141,0.3515,0.2727,0.1616)$} \\

 & week & 2209 & 73 & \makecell{$(0.2399,0.2947,0.2988,0.1666)$}
&  & week & 1613 & 42 & \makecell{$(0.2381,0.3292,0.2858,0.1469)$} \\
\bottomrule
\end{tabular}
\end{adjustbox}
\raggedright
\footnotesize
\\[.1cm]
\textit{Note: the summary for scenario \textsc{year} is detailed in Table~\ref{tab:case_pi_sizes}.}
\end{table}

For each borough, scenario, and profile, we solve the service system design problem using two
approaches. The first is the compact mixed-integer exponential cone formulation \ref{SSD(C/TR)-ECP} solved directly with the 
mixed-integer conic optimizer Mosek. The second is the Benders decomposition scheme, described in Section \ref{sec:benders} and detailed in Appendix \ref{app:benders}, in which the
master problem determines the set of protected users (under coverage and conflict constraints), while subproblem yields service rates, evaluates
the tail-risk and service-level constraints, and generate cutting planes.  Both approaches
are implemented in Mosek Fussion Python API  using the same solver settings. All instances were solved with Gurobi 12 on macOS running on a Mac Studio equipped with an Apple M1 Max chip (10-core CPU: 8 performance and 2 efficiency cores) and 64 GB of unified memory. 

The computational study compares the two approaches in terms of total running time, number
of instances solved to optimality within the time limit, and objective value consistency.
This allows us to assess the scalability and numerical robustness of the proposed
decomposition relative to the compact model.

In the second part of the analysis, we examine the structure of the optimized solutions.
We study how total response times (and its CVaR), congestion level, and cost vary across boroughs and profiles. These results provide managerial
insights into how congestion control and tail-risk constraints shape optimal EMS service
design decisions in large urban systems.

\subsection*{Computational Performance}

We begin by assessing the overall computational performance of the two solution approaches, namely the compact mixed-integer exponential cone formulation (\textbf{compact}) and the decomposition strategy described in Section~\ref{sec:benders} and Appendix~\ref{app:benders} (\textbf{decomposition}). 

Figure~\ref{fig:pp_all} presents the corresponding performance profiles, reporting the percentage of instances solved within a given CPU time. The left panel shows the full time horizon, while the right one zooms into instances solved within $10$ seconds (approximately $10\%$ of the sample), highlighting early-time behavior.

The \textbf{decomposition} approach exhibits a markedly steeper growth from the very first seconds, solving a substantial fraction of instances almost immediately and reaching near-complete coverage significantly earlier than the \textbf{compact} formulation. In contrast, the compact model displays slower progress and a pronounced computational tail, with a non-negligible subset of instances requiring substantially longer running times.

When all instances are considered jointly, the dominance of the decomposition strategy is evident across the entire time spectrum. Not only does it solve easy instances rapidly, but it also shows significantly greater stability when confronted with large-scale or structurally challenging cases. The compact formulation, by comparison, is considerably more sensitive to difficult parameter configurations, leading to delayed convergence and heavier tail behavior. Performance profiles disaggregated by borough are reported in Appendix~\ref{app:pp_borough}, where the same qualitative pattern consistently holds.

 \begin{figure}
 \centering
 {\includegraphics[width=0.5\textwidth]{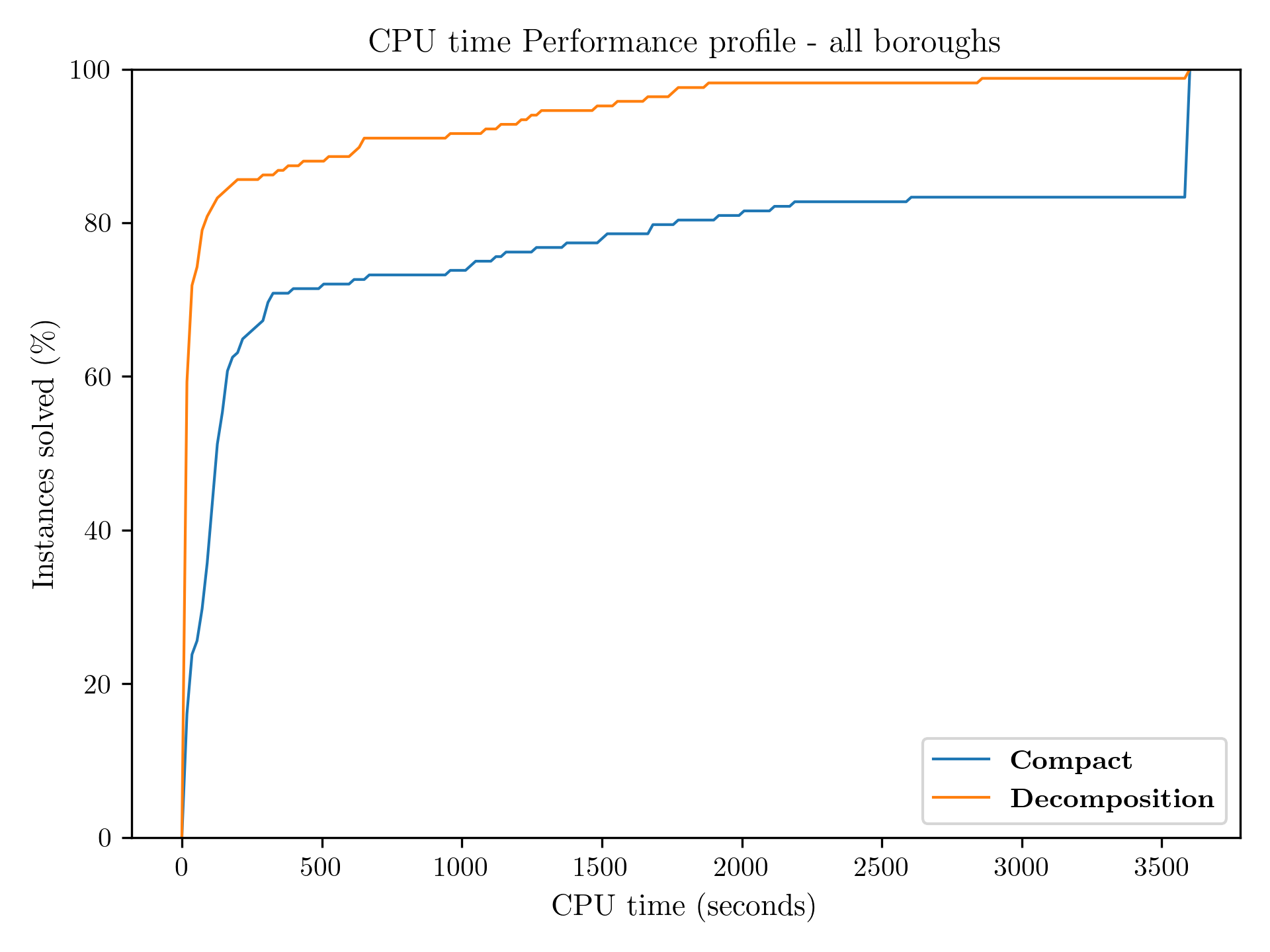}~
\includegraphics[width=0.5\textwidth]{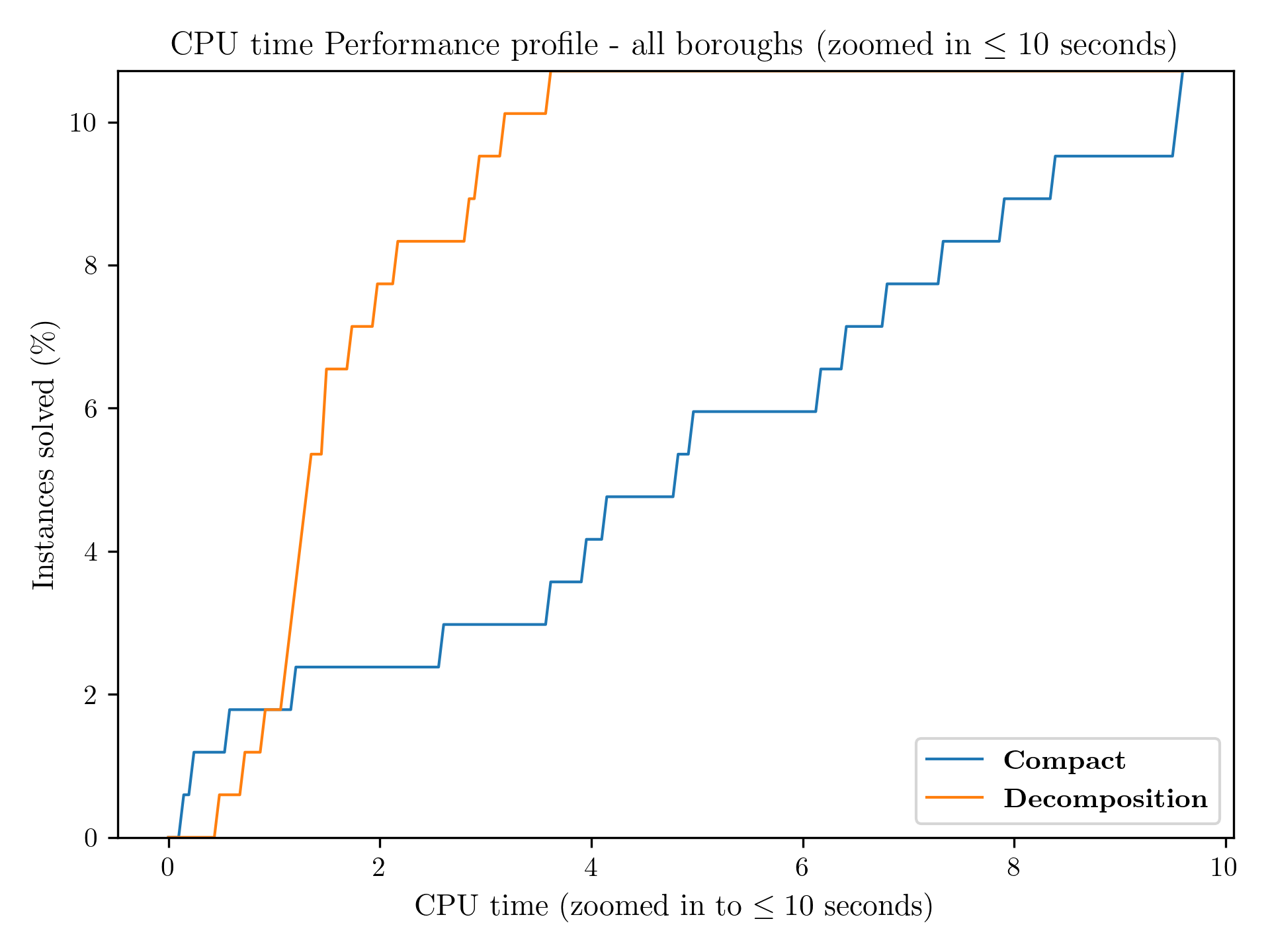}}
{Performance profile of CPU time for all instances (left) and zoomed view for those solved within $10$ seconds (right).\label{fig:pp_all}}
{}
\end{figure}

The global profile clearly illustrates the principal computational advantage of the decomposition approach: while both methods handle the easiest instances efficiently, the decomposition remains robust as instance size and difficulty increase. This feature is particularly valuable in exploratory, sensitivity, or scenario-based analyses, where a large number of related instances must be solved repeatedly and predictable computational effort is essential.

Consistently with the performance profiles, $25\%$ of the instances were not solved to proven optimality within the one-hour time limit when using the \textbf{compact} formulation. Among the unsolved instances, the average residual MIP gap equals $58.60\%$ in Bronx, $58.70\%$ in Brooklyn, $68.82\%$ in Manhattan, and $51.70\%$ in Queens. These gaps further confirm that the compact formulation struggles on a substantial subset of challenging instances, whereas the decomposition approach achieves markedly stronger overall time-to-solution performance with all instances optimally solved.

In order to assess the computational impact of the different profiles (and, indirectly, of the parameter configurations they induce), Figure~\ref{fig:bp_all} reports boxplots of CPU times (in seconds) for each profile (left) and time window scenario (right), comparing the \textbf{compact} and \textbf{decomposition} approaches removing outlier instances from both approaches. In Appendix~\ref{app:pp_borough}, the reader can found the analogous boxplots disaggregated by borough.

Across all profiles, the \textbf{decomposition} approach yields substantially smaller CPU times and a much tighter dispersion, with medians close to the origin and comparatively short upper whiskers. In contrast, the \textbf{compact} formulation exhibits large variability and a pronounced right tail in every profile, with many instances approaching the one-hour limit. The gap between both approaches becomes particularly visible for the most demanding profiles TAIL+, COV, and HARD, where the compact model shows heavy-tailed behavior and frequent time-limit hits, while the \textbf{decomposition} remains consistently fast and stable. Overall, these boxplots reinforce the conclusions from the performance profiles: the \textbf{decomposition} strategy is markedly less sensitive to adverse parameter combinations and provides substantially more predictable computational effort across scenarios.
 \begin{figure}
 \centering
 {\includegraphics[width=0.55\textwidth]{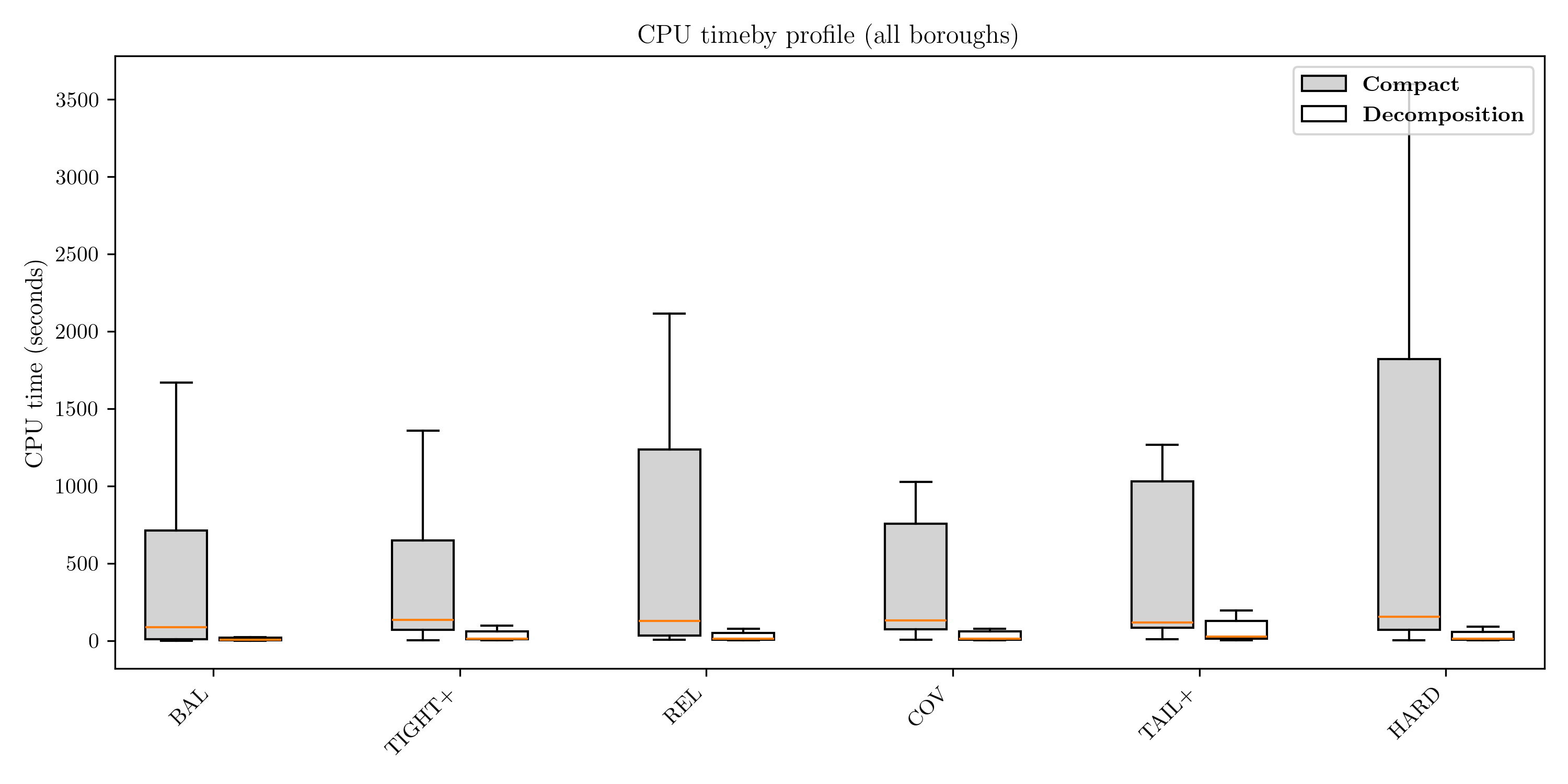}~\includegraphics[width=0.45\textwidth]{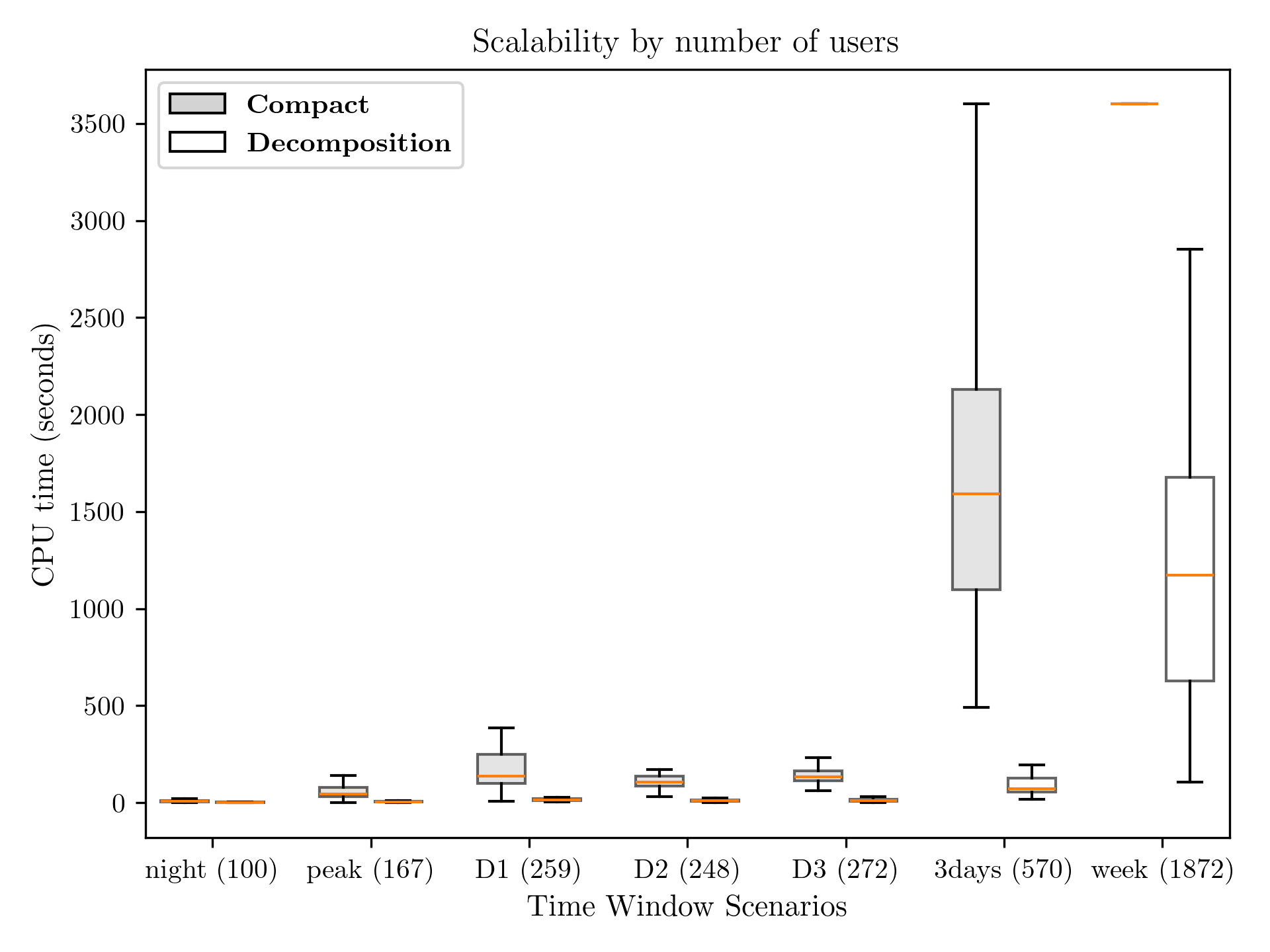}}
{Boxplots of CPU time by profile (left) and time window scenario (right). The distribution of CPU times is shown for the \textbf{compact} and \textbf{decomposition} approaches.\label{fig:bp_all}}
{}
\end{figure}
For the different time window scenarios, that also represent different number of users/incidents, a clear monotonic trend is observed for both approaches: CPU times increase as the average number of users grows, with the most pronounced effects occurring in the \textsc{3days} and \textsc{week} scenarios (triplicating in average the number of users/incidences). However, as we have already observed the growth rate differs substantially between formulations. The \textbf{compact} model exhibits a sharp increase in both median solution time and dispersion as the number of users rises, together with heavy-tailed behavior and frequent time-limit hits in the largest scenarios. In contrast, the \textbf{decomposition} approach scales significantly more gracefully: although computation times increase for larger windows, the median and interquartile ranges remain markedly smaller, and the variability is substantially reduced.

The difference becomes particularly striking in the \textsc{week} scenario, where the \textbf{compact} formulation frequently approaches consistently  the one-hour limit, whereas the \textbf{decomposition} method maintains substantially lower median times and improved stability. Overall, these results confirm that the \textbf{decomposition} strategy not only accelerates solution times in small and medium instances, but also provides superior scalability and robustness as the problem dimension increases.

Finally, we analyze the managerial implications of solving the proposed problem efficiently. In particular, we compare the quality of the objective values obtained by the two approaches on the same instances. Although the \textbf{compact} formulation is computationally more demanding than the \textbf{decomposition} approach to certify optimality, one might expect that both methods eventually identify solutions of comparable quality within the time limit. However, this is generally not the case.

To quantify these differences, we report the relative percent deviation of the best solution obtained with the \textbf{compact} formulation with respect to the optimal solution provided by the \textbf{decomposition} approach. Specifically, we define
$$
{\rm dev} = 100 \frac{{\rm BestObj}_{\rm compact} - {\rm OptObj}_{\rm decomposition}}{{\rm BestObj}_{\rm compact}}\% .
$$
Since the \textbf{decomposition} approach solves the instances to optimality within the time limit, this deviation is always nonnegative and measures the loss in solution quality incurred by relying on the \textbf{compact} formulation under time constraints.

Figure~\ref{fig:devs} presents boxplots of these deviations disaggregated by borough (left), profile (center), and time window scenario (right).
 \begin{figure}
 \centering
 {\includegraphics[width=0.32\textwidth]{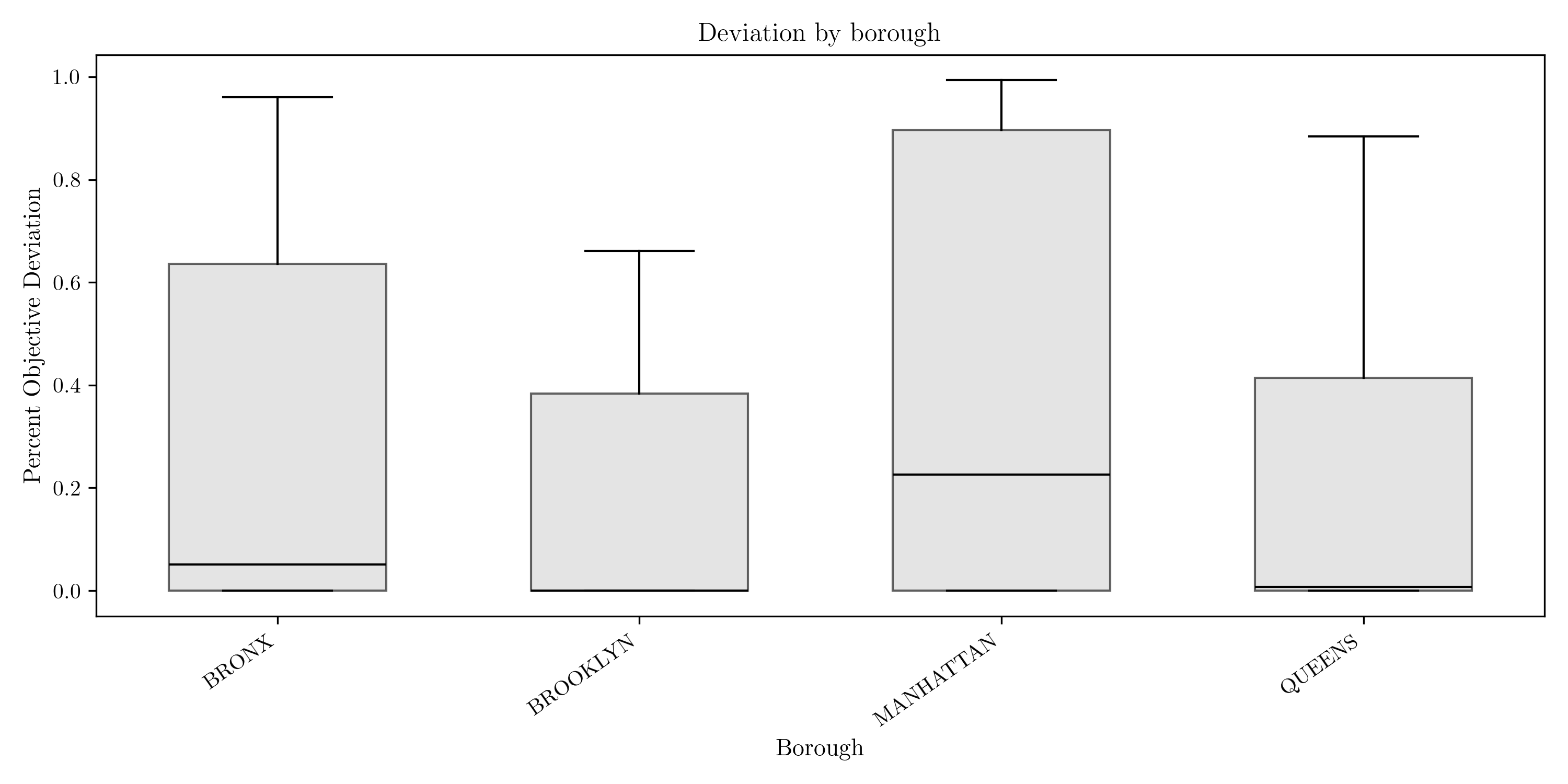}~\includegraphics[width=0.32\textwidth]{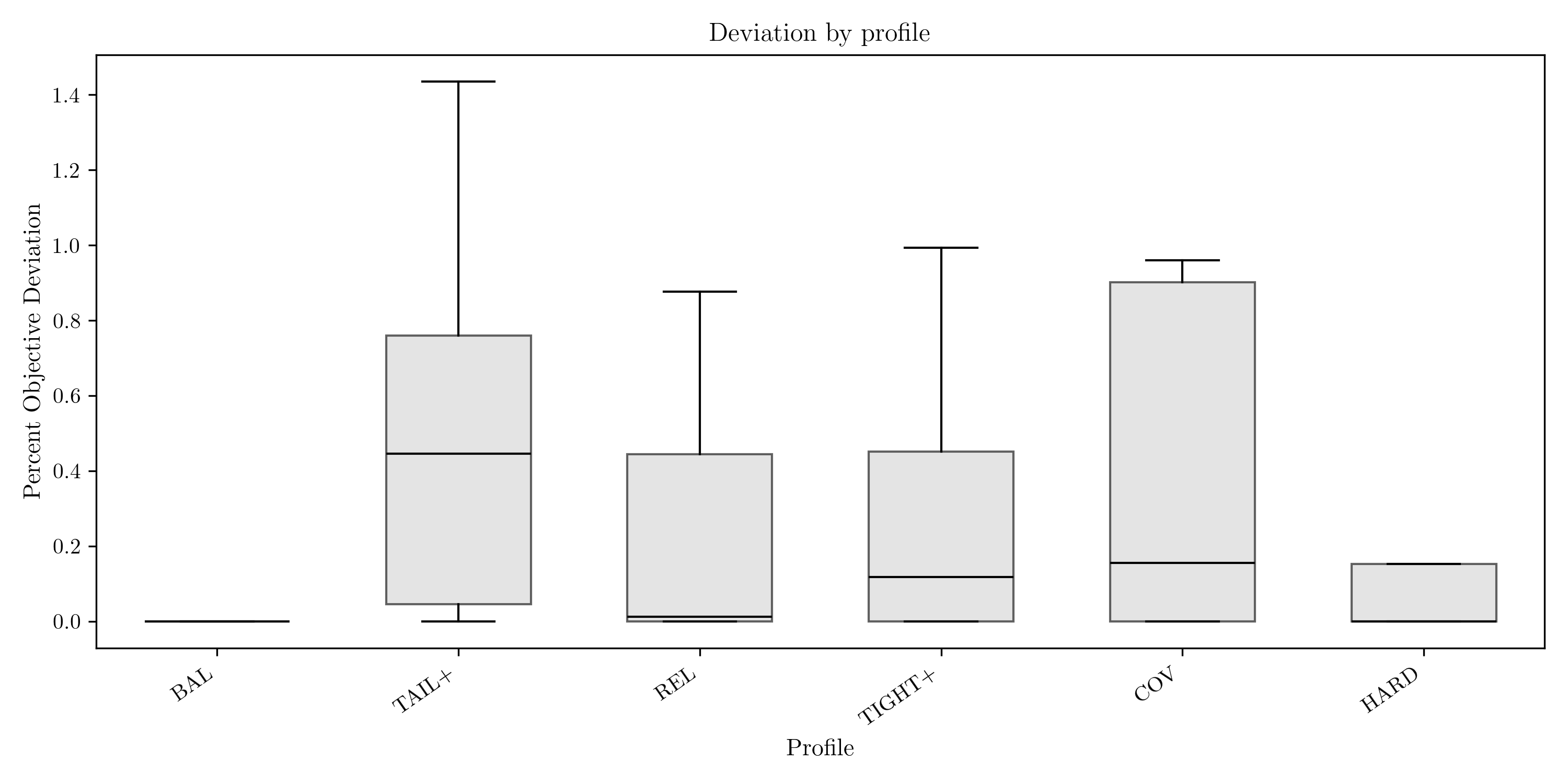}~\includegraphics[width=0.32\textwidth]{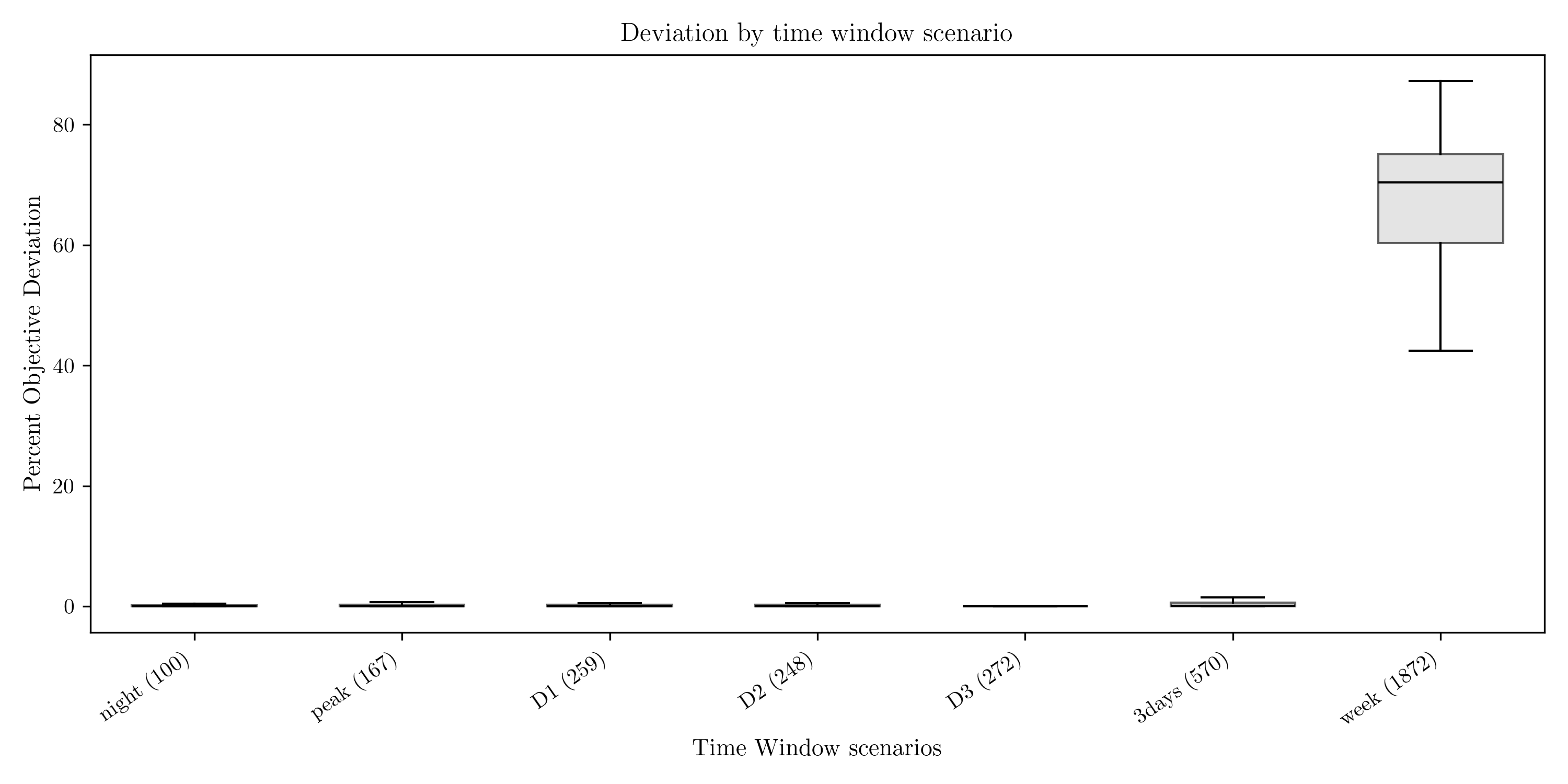}}
{Deviations of best solutions obtained with our approaches.\label{fig:devs}}
{}
\end{figure}

The results reveal substantial variability in the quality loss associated with the \textbf{compact} formulation. By profile, the largest deviations occur under the more demanding configurations (notably \textsc{TAIL+} and \textsc{COV}), where both the median and dispersion are markedly higher. In contrast, the \textsc{BAL} profile exhibits negligible deviations, indicating that easier parameter regimes allow both approaches to achieve similar solutions within the time limit.

The time-window analysis highlights an even more pronounced effect of problem scale. While short windows (e.g., \textsc{night}, \textsc{peak}, and daily scenarios) display near-zero deviations, the \textsc{week} scenario shows dramatically larger percent deviations, with consistently high medians and limited overlap with smaller windows. This confirms that as the number of users increases, the \textbf{compact} formulation struggles not only in computational time but also in solution quality.

Finally, the borough-level results indicate that Manhattan exhibits systematically larger deviations and greater dispersion, suggesting that its instances are structurally more challenging from a computational standpoint. Overall, these findings demonstrate that the \textbf{decomposition} approach delivers not only superior computational efficiency but also substantially improved solution quality in large-scale and demanding scenarios. From a managerial perspective, this distinction is critical: achieving proven optimality is essential, since prematurely terminating the solution process at the time limit may result in significant losses in objective value, particularly in the most complex and large-scale input data.

\subsection*{Case Study: Managerial Insights from a Borough Level Urban Service System}

After completing the computational performance analysis, we now turn to the managerial implications of the proposed framework. To this end, we consider a large-scale case study based on the complete EMS NYC dataset for the year 2025, extracted and processed analogously to the experimental setting described above. In contrast to the controlled benchmarking instances used for performance evaluation, this dataset reflects the full operational scale of the system over an entire year, thereby providing a realistic and practically meaningful environment for policy analysis. 

To facilitate interpretation and isolate the economic trade-offs induced by the service design decisions, we omit the conflict constraints in this analysis. In this setting, the resulting problem becomes polynomially solvable via Algorithm~\ref{alg:poly_no_conflicts}, allowing us to focus exclusively on structural and managerial insights rather than computational limitations. The problems for the different boroughs were solved in 45 seconds in average (max: 80 seconds, min: 28 seconds). The sizes of the borough-level instances and the estimated mixture weights $\pi$, capturing heterogeneous service-time regimes, are reported in Table~\ref{tab:case_pi_sizes}. 

\begin{table}[t]
\centering
\renewcommand{\arraystretch}{1.25}
\setlength{\tabcolsep}{6pt}
\caption{Estimated mixture weights $\pi$ by service regime and user population sizes by borough for the full-year 2025 dataset.}
\label{tab:case_pi_sizes}
{\small
\begin{tabular}{l c c c c c}
\hline
& \multicolumn{4}{c}{$\pi$ by service regime} & \\
\cline{2-5}
\textbf{Borough} 
& \texttt{CARDBR} & \texttt{INJURY} & \texttt{SICK} & \texttt{UNC} & \textbf{Incidents} \\
\hline
Manhattan & 0.1967 & 0.2893 & 0.2473 & 0.2667 & $66,780$ \\
Bronx     & 0.2897 & 0.2846 & 0.2536 & 0.1721 & $62,818$ \\
Brooklyn  & 0.2309 & 0.3013 & 0.2744 & 0.1934 & $79,975$ \\
Queens    & 0.2287 & 0.3229 & 0.2711 & 0.1773 & $55,681$ \\
\hline
\end{tabular}}
\end{table}

The goal of this section is to quantify and interpret the trade-offs embedded in the proposed service system design. In particular, we address four central managerial questions:

\begin{itemize}
    \item[-] \textbf{The price of efficiency.} What is the gain in total response time achieved by the proposed design relative to baseline configurations?
    \item[-] \textbf{The price of congestion control.} How does the design affect system load and operational stability?
    \item[-] \textbf{The price of fairness.} To what extent does the model protect the least-favored users?
    \item[-] \textbf{The price of robustness.} What is the cost of incorporating robustness in the service system design?
\end{itemize}

The terminology follows the seminal notion of the \emph{price of robustness} introduced by \citet{bertsimas2004price}, adapted here to the multi-criteria service system context. By analyzing these four dimensions jointly, we provide a comprehensive assessment of the operational, economic, and societal implications of the proposed optimization framework under realistic large-scale conditions.

To compare the proposed service system design with the one effectively implemented in practice, we first estimate the service rates of the queueing system from historical data. Specifically, for each borough and call type (service regime), we compute empirical service-time estimates based on observed incident durations. Let $S$ denote the sojourn time in minutes, obtained as the difference between the assignment time and the incident close time for valid observations (i.e., strictly positive and finite durations). For each call type $r$, we compute the sample mean of the sojourn time $
\widehat{s}_r \;=\; \frac{1}{n_r}\sum_{i=1}^{n_r} S_{ri}$, where $n_r$ is the number of valid observations for regime $r$. 

The estimated service rate is constructed as $
\widehat{\mu}_r \;=\; \frac{1}{\widehat{s}_r} + \Lambda_r$, where $\Lambda_r$ denotes the empirical arrival rate of the system under regime $r$. This adjustment accounts for the effective workload in the observed system configuration and ensures consistency with the steady-state balance condition of the queue. These borough- and regime-specific service rates provide a data-driven proxy of the performance of the current status quo system, which serves as the benchmark against which the proposed optimized design is assessed. The resulting estimates, based on the full-year 2025 dataset, are reported in Table~\ref{tab:lambda_mu_s}. 

\begin{table}[t]
\centering

\caption{Estimated arrival rates $\Lambda_r$, service rates $\widehat{\mu}_r$, and average service times $1/\widehat{\mu}_r$ by borough and regime.}
\label{tab:lambda_mu_s}
\adjustbox{width=\textwidth}{
\begin{tabular}{l | c c c c | c c c c | c c c c}
\hline
& \multicolumn{4}{c}{Arrival rates ($\Lambda$)} 
& \multicolumn{4}{c}{Service rate ($\widehat{\mu}$)} 
& \multicolumn{4}{c}{Avg.\ service time in minutes ($1/\widehat{\mu}$)} \\
\cline{2-5} \cline{6-9} \cline{10-13}
\textbf{Borough}
&\texttt{CARDBR} & \texttt{INJURY} & \texttt{SICK} & \texttt{UNC}
& \texttt{CARDBR} & \texttt{INJURY} & \texttt{SICK} & \texttt{UNC}
& \texttt{CARDBR} & \texttt{INJURY} & \texttt{SICK} & \texttt{UNC} \\
\hline
Manhattan 
& 0.0375 & 0.0552 & 0.0472 & 0.0509 
& 0.0505 & 0.0712 & 0.0613 & 0.0705 
& 19.82 & 14.05 & 16.30 & 14.18 \\

Bronx 
& 0.0520 & 0.0511 & 0.0455 & 0.0309 
& 0.0643 & 0.0669 & 0.0591 & 0.0474 
& 15.55 & 14.96 & 16.91 & 21.09 \\

Brooklyn 
& 0.0528 & 0.0689 & 0.0627 & 0.0442 
& 0.0656 & 0.0843 & 0.0762 & 0.0612 
& 15.24 & 11.86 & 13.13 & 16.34 \\

Queens 
& 0.0364 & 0.0514 & 0.0431 & 0.0282 
& 0.0495 & 0.0669 & 0.0570 & 0.0446 
& 20.18 & 14.95 & 17.55 & 22.43 \\
\hline
\end{tabular}}
\end{table}

The estimates reveal substantial heterogeneity across boroughs and regimes, both in arrival and service rates. Brooklyn exhibits the highest service rates (particularly under regime~2), whereas Queens and Manhattan display comparatively lower service rates in regimes~1 and~4, which translate into longer average service times. At the same time, arrival rates vary significantly across boroughs, with Brooklyn and Manhattan experiencing consistently higher demand intensities in several regimes. 

The joint variation of $\Lambda_r$ and $\widehat{\mu}_r$ implies heterogeneous utilization levels and congestion exposure across the system. In particular, regimes combining high arrival rates with relatively modest service rates are structurally more vulnerable to congestion. These differences highlight the operational imbalance across boroughs and reinforce the importance of borough-specific and regime-aware service system design decisions.

\begin{itemize}
\item[-] {\bf The price of efficiency (gains in total response time). } The \emph{price of efficiency} quantifies the reduction in total response time obtained by optimally redesigning the service system relative to the current data-driven configuration. Since response time directly impacts service quality, reliability, and, particularly in emergency settings, safety outcomes, improving this metric is of primary managerial relevance.

Figure~\ref{fig:violinplots} presents the global comparison. The left panel shows the full distribution of total response times, and the right panel reports the corresponding conditional value-at-risk (CVaR). The optimized design produces a dramatic and systematic leftward shift of the entire distribution. Median response times decrease by approximately one hour, dispersion is slightly reduced, and extreme delays are significantly mitigated. Statistical tests confirm overwhelming significance. The plots additionally report the results of a parametric paired $t$-test for equality of total response time means (left bottom corner), including the test statistic ($t$), the corresponding $p$-value, the achieved statistical power, and a $95\%$ confidence interval for the effect size ($\delta$), interpreted as the true difference in means. Analogously, in the right plot, we report a summary of the nonparametric paired Wilcoxon test for the homogeneity of CVaR of the total response time groups, including the test statistic ($w$), the $p$-value, and a $95\%$ confidence interval for the true median of the differences ($\eta$). For the total response time, values of  $p=0.0^{***}$ and power $=100\%$, with an estimated mean reduction of roughly 60 minutes and an exceptionally tight confidence interval point out and quantify the price of efficiency. Hence, the improvement is not only statistically significant but also operationally large.

 \begin{figure}
 {\includegraphics[width=0.5\textwidth]{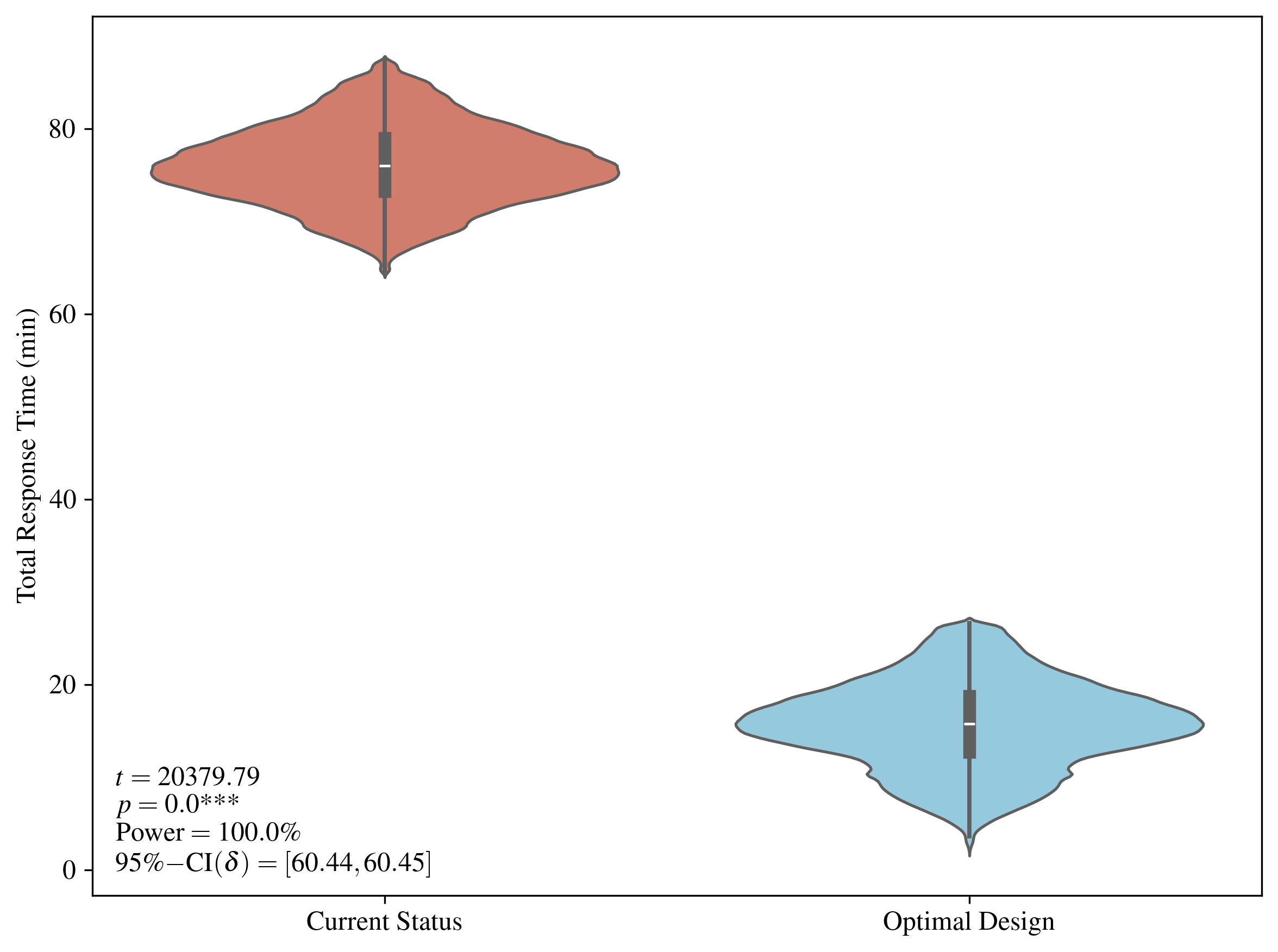}~
\includegraphics[width=0.5\textwidth]{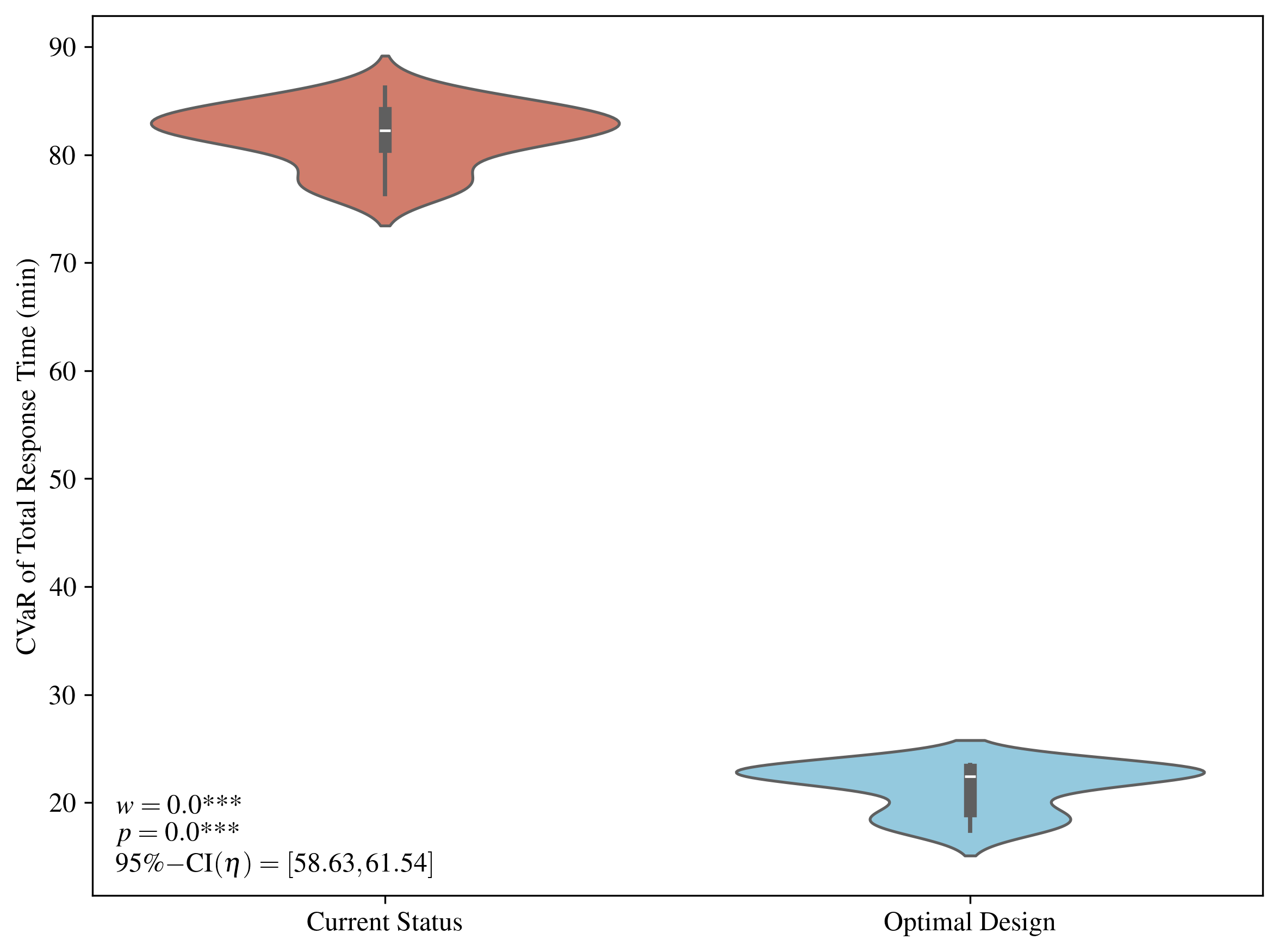}}
{Global comparison of total response time (left) and its CVaR (right), in minutes, between the current system and the optimal design.\label{fig:violinplots}}
{}
\end{figure}

The spatial robustness of these gains is illustrated in Figure~\ref{fig:boxplot_response_time_borough} (in Appendix \ref{sec:case_study}), which disaggregates total response times by borough. In every borough the optimal design uniformly reduces median response times from approximately 70–80 minutes to roughly 15–20 minutes. The reductions are statistically significant in all cases ($p=0.0^{***}$) with extremely large test statistics and narrow confidence intervals for the effect size, confirming that the improvements are consistent across heterogeneous demand environments. Importantly, the effect is uniform in direction: no borough experiences a deterioration, and congestion is not shifted geographically. Instead, the optimized configuration improves efficiency system-wide.

A similar pattern emerges when results are examined by operational profile (see Appendix~\ref{sec:case_study}). Under the current configuration, median response times remain high across profiles. Under the optimized design, medians fall consistently to the range of 10–20 minutes, with controlled variability. The largest relative gains occur in stringent and tail-focused profiles (\textsc{HARD}, \textsc{TAIL+}, \textsc{TIGHT+}), where the current system incurs substantial efficiency penalties. Even in robustness-oriented profiles (\textsc{REL}, \textsc{COV}), the optimized design achieves clear reductions in both central tendency and dispersion.

Overall, the price of efficiency is uniformly favorable. The optimized design dominates the current configuration globally, spatially, and across design regimes, delivering large, consistent, and statistically robust reductions in total response time. From a managerial perspective, this translates into a faster, more stable, and more reliable service system with significantly lower exposure to extreme delays.

\item[-] {\bf The price of congestion control (impact on system load and stability). } Congestion is a primary source of instability and extreme delays in urban service systems. The \emph{price of congestion control} captures the trade-off between reducing system load and the potential cost or efficiency adjustments required to achieve stability.

Under the historically estimated configuration, the system operates persistently close to saturation, leaving little buffer to absorb demand fluctuations. In contrast, the optimal design shifts the distribution of congestion levels downward, introducing structural slack and reducing the likelihood of near-instability regimes. This represents a transition from a brittle operating point to a more resilient and stable configuration.

Figure~\ref{fig:barplots} reports congestion and cost deviations by borough. In the left panel, congestion level reductions are observed uniformly across all boroughs, indicating that improvements are systemic rather than spatially reallocated. The reductions are statistically significant and accompanied by controlled variability, confirming that stability gains are achieved without transferring load between regions. The right panel shows the associated cost deviations, quantifying the operational adjustment required to secure these congestion reductions.

 \begin{figure}
 \centering
 {\includegraphics[width=0.5\textwidth]{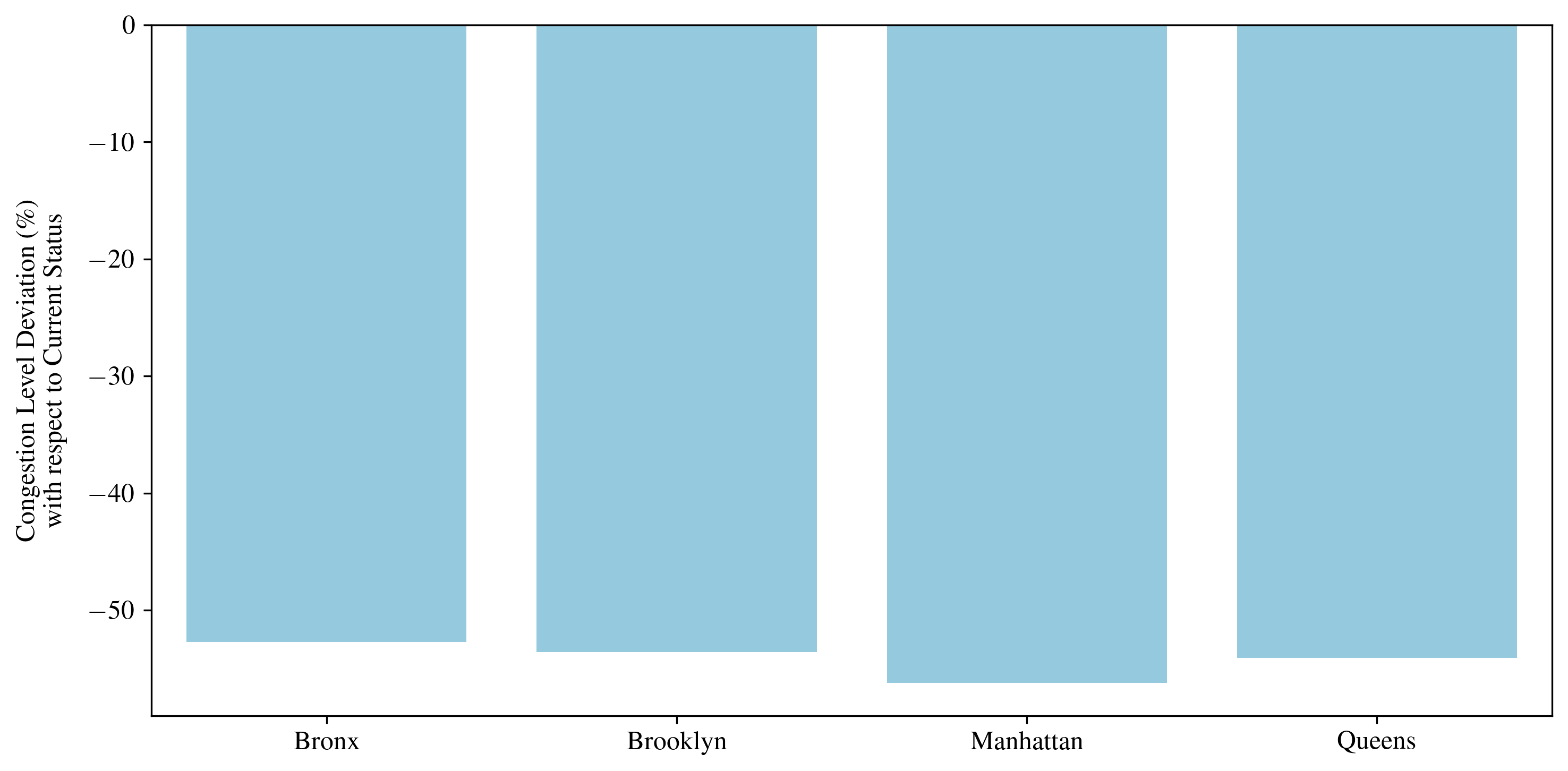}~
\includegraphics[width=0.5\textwidth]{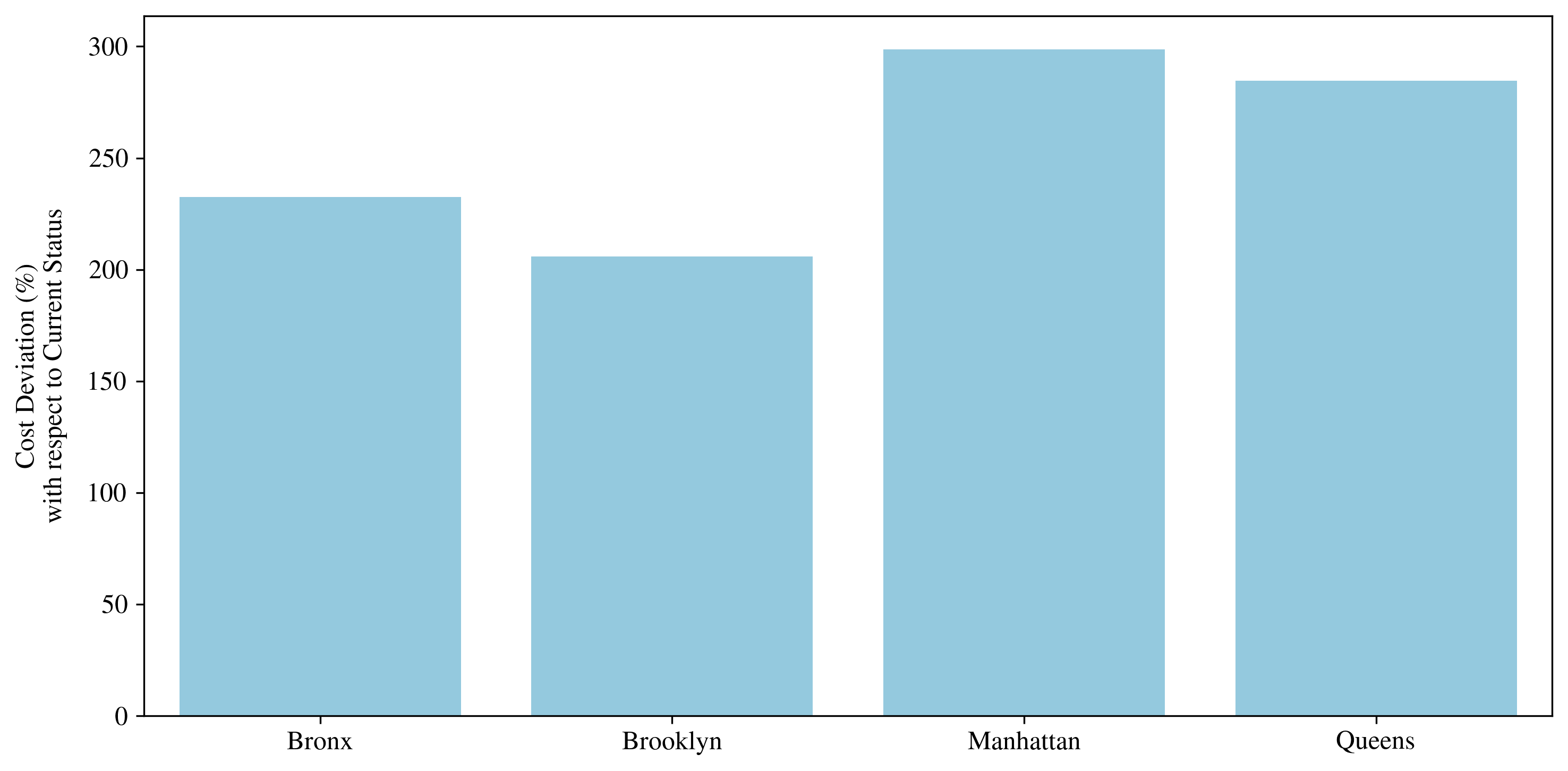}}
{Congestion level (left) and cost deviations (right), in \%, relative to the current system by borough. Negative values indicate reduced congestion under the optimal design.\label{fig:barplots}}
{}
\end{figure}

A similar pattern emerges across operational profiles (see Appendix~\ref{sec:case_study}). Profiles imposing stringent guarantees or tail protection exhibit the highest congestion under the current system, reflecting the inefficiency of enforcing strict targets without coordinated optimization. Under the optimal design, congestion levels decrease markedly across all profiles, often moving the system away from the saturation boundary. Even balanced profiles experience systematic improvements.

Overall, the cost of congestion control is favorable. The optimized design achieves substantial and statistically significant reductions in system-wide congestion, with reductions exceeding $50\%$ in every borough and ranging from $40\%$ to nearly $80\%$ across the profiles. It enhances system stability without inducing spatial imbalances. These results demonstrate that congestion control can be embedded directly into system design, yielding a structurally more robust and resilient operation than the current configuration.

\item[-] {\bf The price of fairness (protection of least-favored users). } Fairness in service systems concerns the protection of users most exposed to extreme delays. The \emph{price of fairness} measures the performance trade-offs required to reduce tail risk, typically assessed through risk-sensitive metrics such as the conditional value-at-risk (CVaR). In this context, fairness does not imply equalizing averages, but limiting exposure to severe service outcomes.

Figure~\ref{fig:barplot_cvar_borough} reports borough-level changes in CVaR relative to the current system. Negative values indicate reductions in extreme delays under the optimal design. Substantial and systematic CVaR reductions are observed across all boroughs, with particularly pronounced improvements in high-demand regions such as Brooklyn, Manhattan, and Queens. These results indicate that tail-risk mitigation is structural and not confined to isolated scenarios.

Figure~\ref{fig:barplot_cvar_profile} provides the same analysis by operational profile. Under the current configuration, profiles imposing strict guarantees (\textsc{HARD}) exhibit substantial tail exposure, highlighting the inefficiency of enforcing fairness objectives without coordinated optimization. The optimal design markedly reduces CVaR across nearly all profiles, including tail-focused configurations (\textsc{TAIL+}), without transferring risk to other regimes.

The optimized design delivers large and consistent reductions in tail risk across both spatial and operational dimensions. To formally assess the differences in CVaR values reported in Figure~\ref{fig:violinplots} (right), we apply the Wilcoxon signed-rank test. The test yields a $p$-value below conventional significance levels ($p<0.001$) and a test statistic of $w=0$, indicating a systematic reduction in tail risk under the optimized design. The estimated median reduction is approximately one hour.  These results support that the proposed design not only cuts the time for the average user but also narrows the gap experienced by the least-favored users, preserving homogeneity and yielding a service system that meets efficiency at the time that promotes robustness and fairness. These findings demonstrate that such objectives can be embedded directly into system design through explicit risk-sensitive optimization, yielding measurable protection for least-favored users without degrading overall performance.

\item[-] {\bf The price of robustness (cost of the proposed service system design). } Robustness refers to the ability of a service system to sustain acceptable performance under demand variability and adverse operating conditions. The \emph{price of robustness} measures the additional cost required to transition from the current estimated configuration to a design that delivers improved stability, lower congestion, and reduced tail risk. Since implementation decisions ultimately depend on budget considerations, quantifying this cost is essential for assessing the practical viability of the proposed framework.

Figure~\ref{fig:barplots} (right panel) reports borough-level cost deviations, expressed as percentages relative to the current system. Cost increases are not uniform across boroughs: high-demand areas require larger adjustments, reflecting the need for additional capacity to prevent persistent saturation. However, all increases remain within the same order of magnitude, indicating that robustness is achieved through coordinated system-level redesign rather than localized overprovisioning.

A similar pattern emerges across operational profiles (see Appendix~\ref{sec:case_study}). Profiles emphasizing strict guarantees or tail protection entail higher cost increments, as expected, since reducing extreme delays requires structural slack. In contrast, balanced or efficiency-oriented profiles exhibit more moderate cost adjustments. 

The congestion- and tail-aware system architecture leads to an increase of up to $300\%$ in every borough and ranges from $150\%$ to over $600\%$ across the different profiles. Nevertheless, such deviation might be readily tailored when decision-making is driven by the effective cost considerations. Note that suitable smaller values of parameter $\kappa$ make the loss function~\eqref{eq:obj_tradeoff} prioritizes the cost, being able to reach the most conservative service capacity while still keep the quality requirements when $\kappa$ equals zero. On the contrary, larger values of $\kappa$ prevail the emphasis on reducing congestion levels. The motivations behind the decision-making process and the state of nature are captured within this flexible modeling framework, enabling the system to achieve proper equilibrium. The evidence shows that substantial gains in efficiency, congestion control, and fairness can be achieved at a measurable and economically interpretable cost, supporting informed adoption in real urban service systems.
\end{itemize}

\section{Conclusion}\label{sec:conclusions} 

We develop an optimization-based framework for the design of multi-regime parallel service queueing systems under congestion and risk considerations. The core contribution is a flexible mixed-integer exponential conic formulation that integrates (i) service-level agreement (SLA) chance constraints protecting a prescribed fraction of users, (ii) structural design constraints modeled through conflict graphs, and (iii) explicit tail-risk control under the CVaR paradigm. 

Although the resulting problem is NP-hard, we design an efficient decomposition approach that substantially alleviates the computational burden in large-scale instances. Moreover, we show that under suitable structural simplifications, a constructive polynomial-time procedure can be derived, thereby identifying tractable regimes of the model and clarifying the boundary between complexity and structure.

From a practical standpoint, the framework provides a unified decision-support tool for service system design. The case study demonstrates that the optimized design delivers substantial and systematic improvements in efficiency, congestion control, fairness, and robustness. In particular, the ``prices of" efficiency, congestion control, fairness, and robustness are shown to be measurable and transparent, allowing decision makers to explicitly quantify trade-offs between response time reduction, system stability, tail-risk mitigation, and budgetary adjustments. 

Overall, the proposed methodology bridges modern conic optimization modeling with actionable service system design, offering both theoretical insight and operational value for large-scale urban service applications.

A natural avenue for future research is the extension of the framework to multi-facility service systems in which assignment decisions endogenously determine arrival rates. In such settings, allocation policies would interact with congestion dynamics, leading to joint design and routing models with coupled arrival and service mechanisms.

 \section*{Acknowledgements}

The authors acknowledge financial support by  grants PID2020-114594GB-C21, 
PID2022-139219OB-I00, PID2024-156594NB-C21, and 
RED2022-134149-T  (Thematic Network on Location Science and Related Problems) funded by MICIU/AEI/10.13039/501100011033; FEDER+Junta de Andalucía projects C‐EXP‐139‐UGR23, and AT 21\_00032; SOL2024-31596
and SOL2024-31708 funded by US;  the IMAG-Maria de Maeztu grant CEX2020-001105-M / AEI / 10.13039 / 501100011033; and the IMUS--Maria de Maeztu grant CEX2024-001517-M. 

\appendix

\section{SSD(C/TR)-ECP Full Model}\label{ap:fullmodel}

\setlength{\abovedisplayskip}{6pt}
\setlength{\belowdisplayskip}{6pt}
\setlength{\abovedisplayshortskip}{4pt}
\setlength{\belowdisplayshortskip}{4pt}

\begin{align}\tag*{\rm SSD(C/TR)-ECP}\label{SSD(C/TR)-ECP}
\min \quad
& c^{t}\mu - \kappa \sum_{r\in\mathcal{R}} \nu_r
\nonumber\\
\text{s.t.}\quad
& \mu_r - \Lambda_r \ge \varepsilon, &\forall r\in\mathcal{R},
\label{ctr:1}\\
& \sum_{r\in\RR} \pi_r \zeta_{ar}
\le 1 + (\alpha_a-1)s_a, & \forall a\in\mathcal{A},\nonumber
\\
& u_{ar} = -(\mu_r-\Lambda_r)(t_a^\star-t_a), & \forall a\in\mathcal{A}, r \in \RR,\label{ctr:u}\\
&(\zeta_{ar},\, s_a,\,
u_{ar})
\in \mathcal{K}_{\exp},& \forall a\in\mathcal{A},\ \forall r\in\mathcal{R},\nonumber\\
&\sum_{a\in\mathcal{A}} s_a \ge \lceil \beta |\mathcal{A}| \rceil,\nonumber\\
& s_a + s_{a'} \;\le\; 1, &\forall \{a,a'\}\in\mathcal{E},
\nonumber\\
& r_a \ge t_a + \sum_{r\in\mathcal{R}} \pi_r \tau_r, & \forall a\in\mathcal{A},
\label{ctr:6}\\
& U_a \ge r_a - \eta, &\forall a\in\mathcal{A},
\label{ctr:7}\\
& \eta + \frac{\sum_{a\in\mathcal{A}} U_a}{\lfloor (1-\gamma)|\A|\rfloor} \le \Gamma,
\label{ctr:8}\\
& \nu_r + \ell_r \ge 0, & \forall r\in\mathcal{R},
\label{ctr:9}\\
& (\tau_r,\,1,\,\ell_r) \in \mathcal{K}_{\exp},
&\forall r\in\mathcal{R},\label{ctr:9b}\\
& (\mu_r-\Lambda_r,\,1,\,\nu_r)
\in \mathcal{K}_{\exp}, & \forall r\in\mathcal{R},\label{ctr:10}\\
& \mu_r, \nu_r, \ell_r, \tau_r \geq 0, & \forall r \in \RR,\nonumber\\
& U_a, r_a \geq 0, & \forall a \in \A,\nonumber\\
& \eta \geq 0, \nonumber\\
& s_a \in \{0,1\}, & \forall a\in\mathcal{A}.\nonumber
\end{align}
\setlength{\abovedisplayskip}{10pt}
\setlength{\belowdisplayskip}{10pt}

\section{Benders Decomposition for SSD(C/TR)-ECP}
\label{app:benders}

This appendix provides the detailed derivation of the Benders decomposition
used to solve problem~\ref{SSD(C/TR)-ECP}.

Fixing $s=\bar s\in\{0,1\}^{\mathcal A}$, the continuous subproblem defining
$\Theta(\bar s)$ is given by

\setlength{\abovedisplayskip}{6pt}
\setlength{\belowdisplayskip}{6pt}
\setlength{\abovedisplayshortskip}{4pt}
\setlength{\belowdisplayshortskip}{4pt}

\begin{align}
\min \quad
& c^{t}\mu - \kappa \sum_{r\in\mathcal{R}} \nu_r
\nonumber\\
\text{s.t.}\quad
& \eqref{ctr:1}, \eqref{ctr:u}, \eqref{ctr:6}-\eqref{ctr:10},\nonumber\\
& \sum_{r\in\RR} \pi_r \zeta_{ar} \le 1 + (\alpha_a-1)\bar s_a, & \forall a\in\mathcal{A}, \label{linear_s}\\
&(\zeta_{ar},\, \bar s_a,\,
u_{ar})
\in \mathcal{K}_{\exp},& \forall a\in\mathcal{A},\ \forall r\in\mathcal{R},\label{exp_s}\\
& \bar{s}_a + \bar{s}_{a'} \;\le\; 1, & \forall \{a,a'\}\in\mathcal{E},
\label{conflict_benders}\\
& \mu_r, \nu_r, \ell_r, \tau_r \geq 0, & \forall r \in \RR,\nonumber\\
& U_a, r_a \geq 0, & \forall a \in \A,\nonumber\\
& \eta \geq 0, \nonumber
\end{align}
\setlength{\abovedisplayskip}{10pt}
\setlength{\belowdisplayskip}{10pt}
Note that all constraints in the problem above  are convex. Thus, assuming that the subproblem is feasible, and since the problem clearly verifies strong duality (e.g., by Slater's condition), let $f_a, g_{\{a,a'\}}\ge 0$ denote the dual multipliers associated with
\eqref{linear_s} and \eqref{conflict_benders}, respectively, and let
$$
h_{ar}=(h^{(1)}_{ar},h^{(2)}_{ar},h^{(3)}_{ar})\in\mathcal K_{\exp}^\ast
$$
be the dual vector associated with \eqref{exp_s}.
The second component $h^{(2)}_{ar}$ corresponds to the coordinate in which
$\bar s_a$ appears. Collecting the terms of the Lagrangian that depend on $\bar s$, we obtain
$$
\sum_{a\in\mathcal A} (1-\alpha_a)f_a \bar s_a
+
\sum_{\{a,a'\}\in\mathcal{E}} g_{\{a,a'\}}(\bar{s}_a + \bar{s}_{a'})
+
\sum_{a\in\mathcal A}\sum_{r\in\mathcal R} h^{(2)}_{ar}\bar s_a.
$$
Therefore, a subgradient of $\Theta(\cdot)$ at $\bar s$ is given componentwise by
\begin{equation*}
q_a(\bar s)
=
(1-\alpha_a)f_a
+
\sum_{e\in\mathcal{E}:\, a\in e}g_e
+
\sum_{r\in\mathcal R} h^{(2)}_{ar},
\label{eq:subgradient}
\end{equation*}
for all $a\in\mathcal A.$ By convexity of $\Theta(\cdot)$, the supporting hyperplane inequality yields, 
$$
\Theta(s)\ge
\Theta(\bar s)
+
\sum_{a\in\mathcal A} q_a(\bar s)(s_a-\bar s_a),
\,\forall s\in[0,1]^{\mathcal A}.
$$
Introducing the master variable $\theta$, this inequality gives rise to the
Benders optimality cut
\begin{equation*}
\theta \ge
\Theta(\bar s)
+
\sum_{a\in\mathcal A} q_a(\bar s)(s_a-\bar s_a).
\label{eq:benders_cut}
\end{equation*}

If the subproblem is infeasible for $\bar s$, a feasibility cut excluding this
solution is added
\begin{equation*}
\sum_{a\in\mathcal A:\bar s_a=1} s_a \le
|\{a \in \A: \bar s_a=1\}| - 1.
\label{eq:feas_cut}
\end{equation*}

\section{Computational Performance. Additional Plots}
\label{app:pp_borough}

 \begin{figure}[h!]
 {\includegraphics[width=0.5\textwidth]{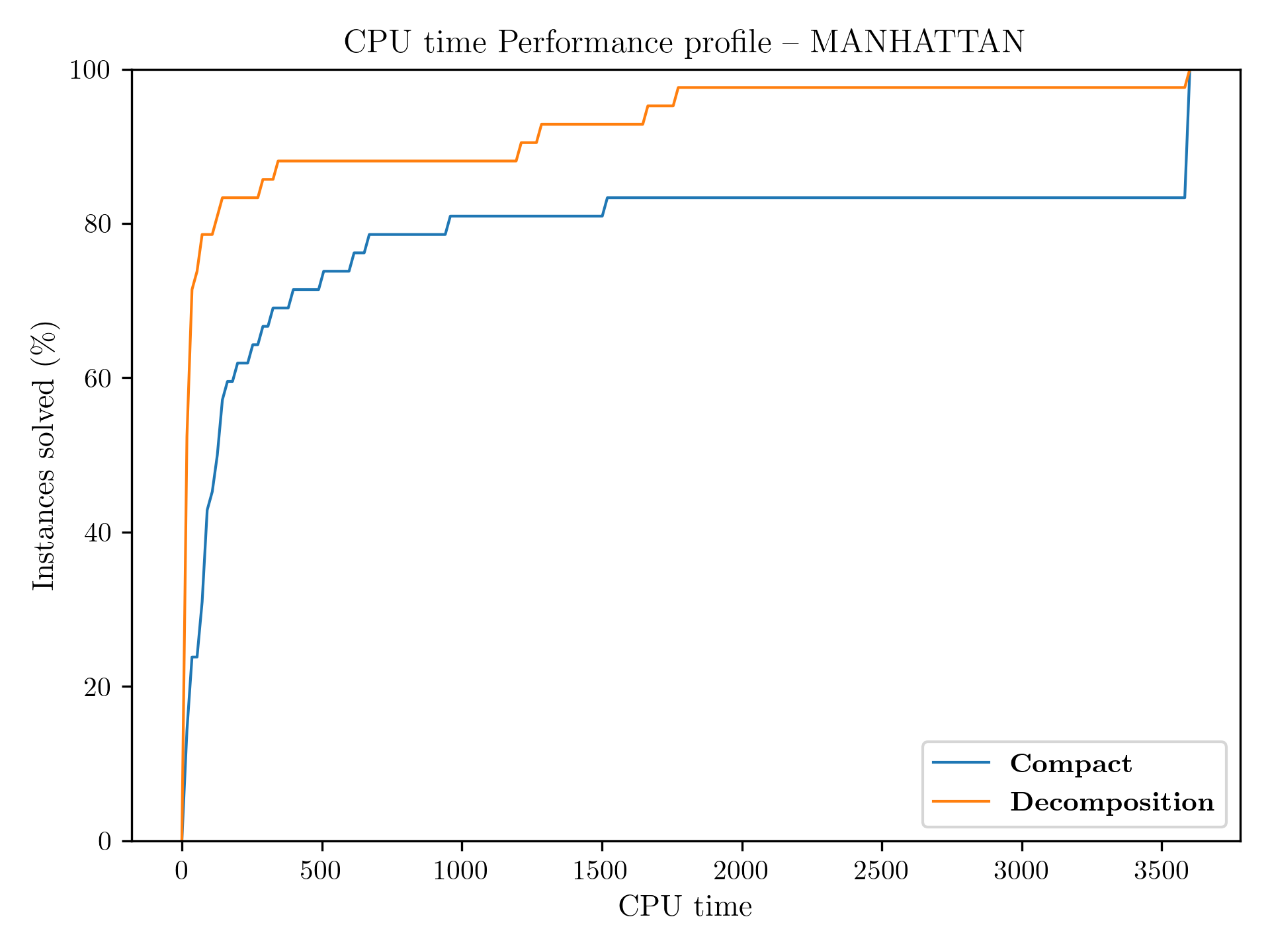}~\includegraphics[width=0.5\textwidth]{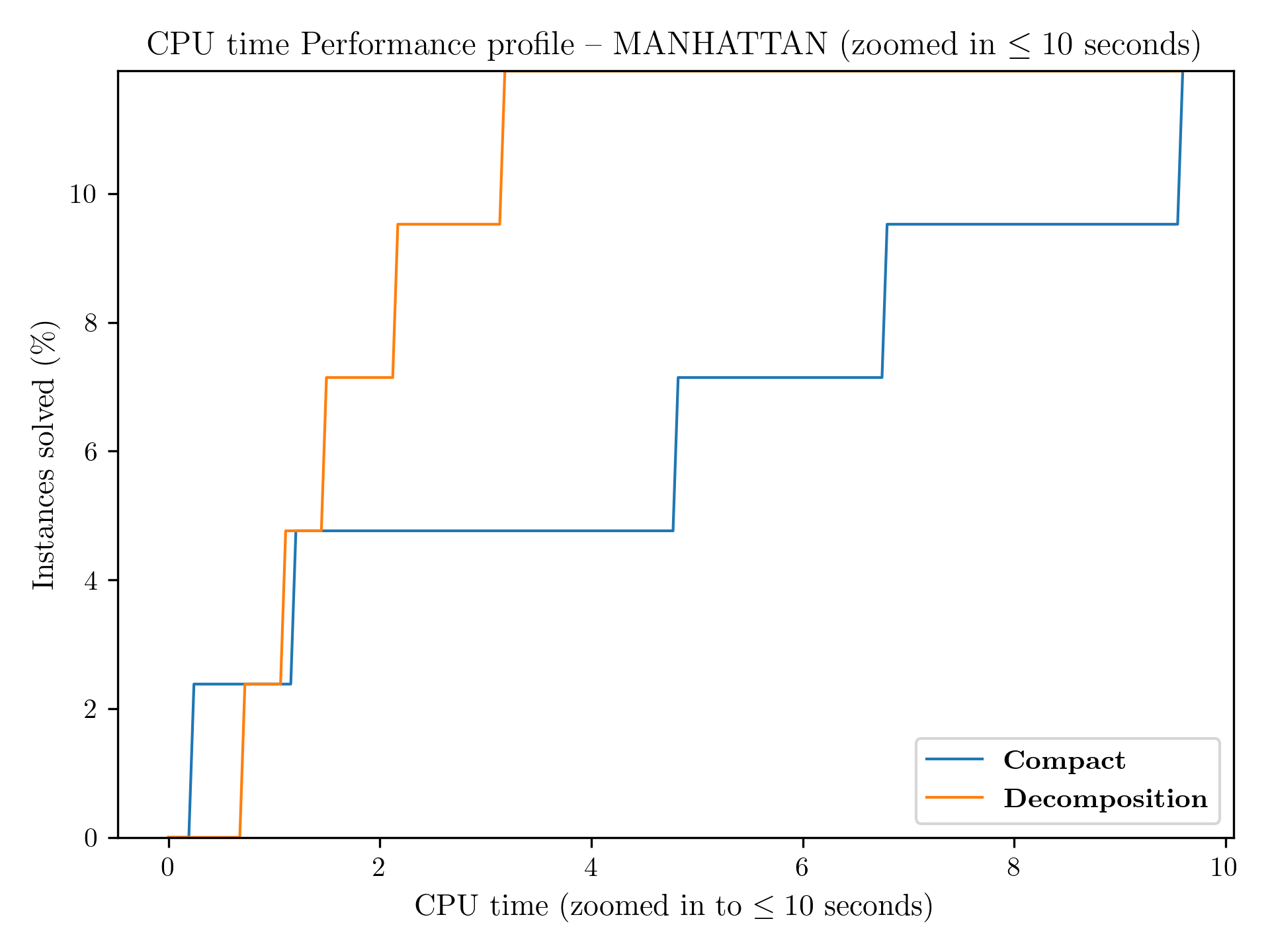}}
{Performance Profile of CPU Time for Manhattan in our experiments (left) and zoomed to those solved in less than $10$ seconds (right).\label{fig:pp_manhattan}}
{}
\end{figure}

 \begin{figure}[h!]
 {\includegraphics[width=0.5\textwidth]{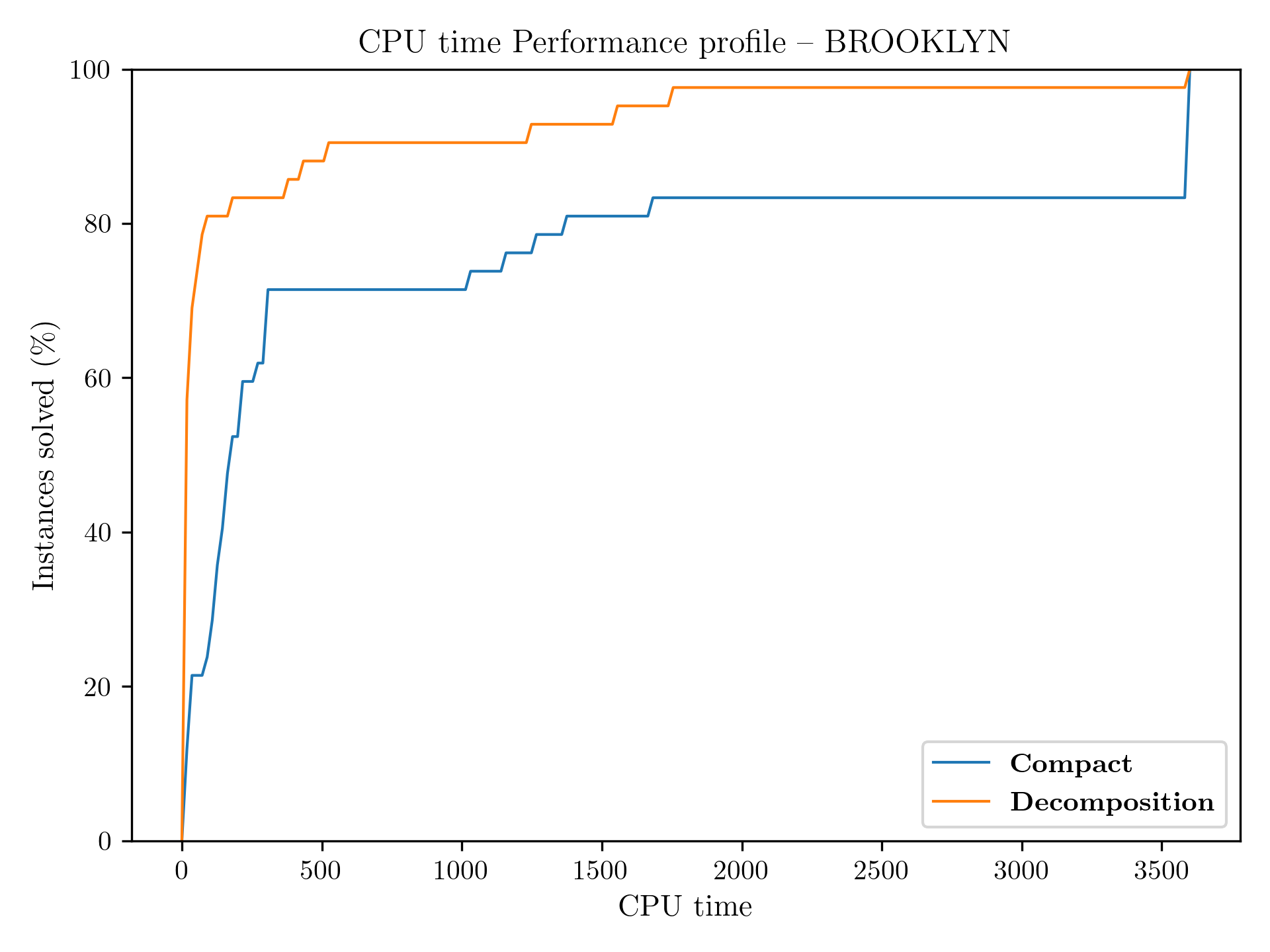}~\includegraphics[width=0.5\textwidth]{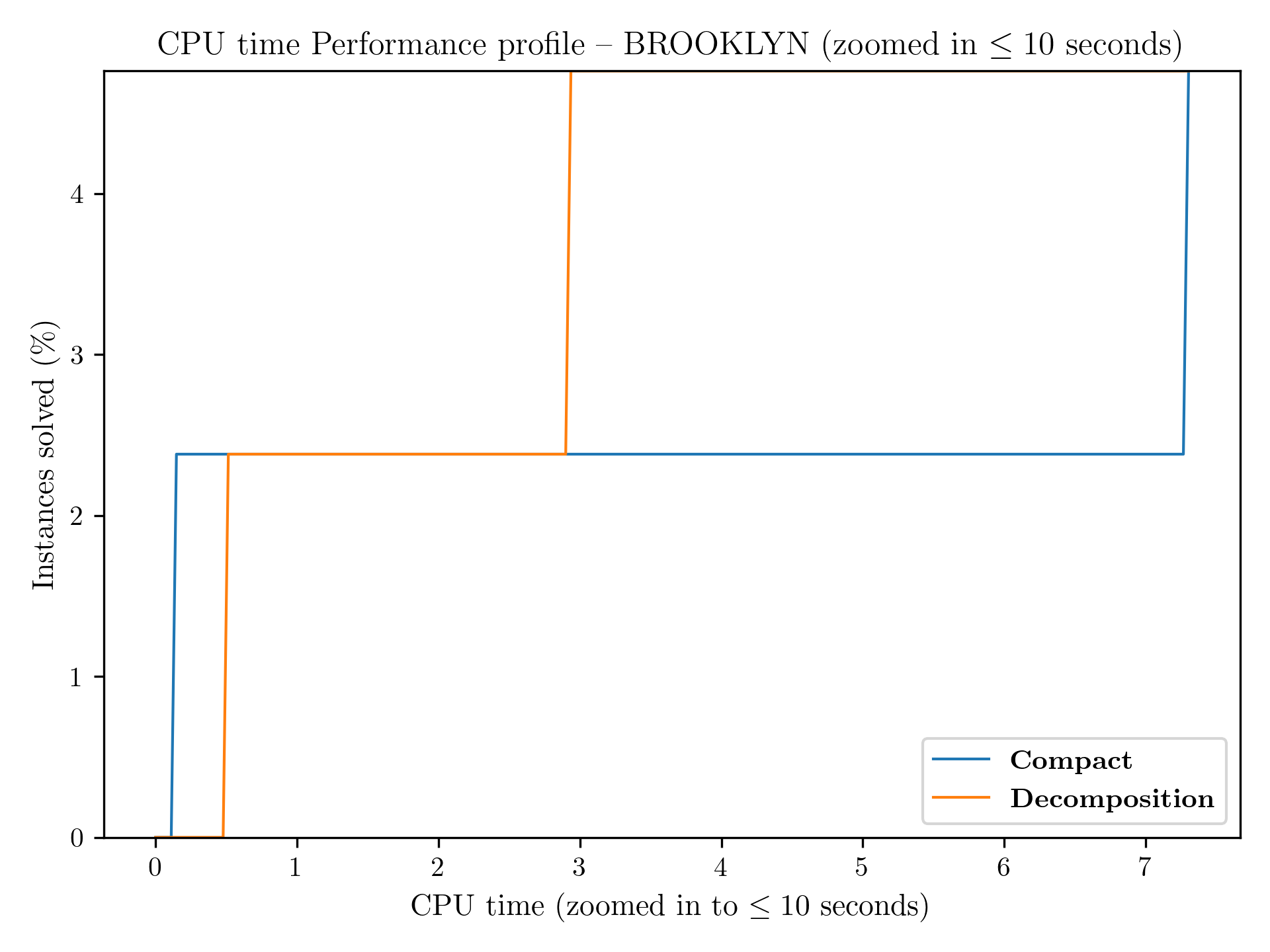}}
{Performance Profile of CPU Time for Brooklyn in our experiments (left) and zoomed to those solved in less than $10$ seconds (right).\label{fig:pp_brooklyn}}
{}
\end{figure}

 \begin{figure}
 {\includegraphics[width=0.5\textwidth]{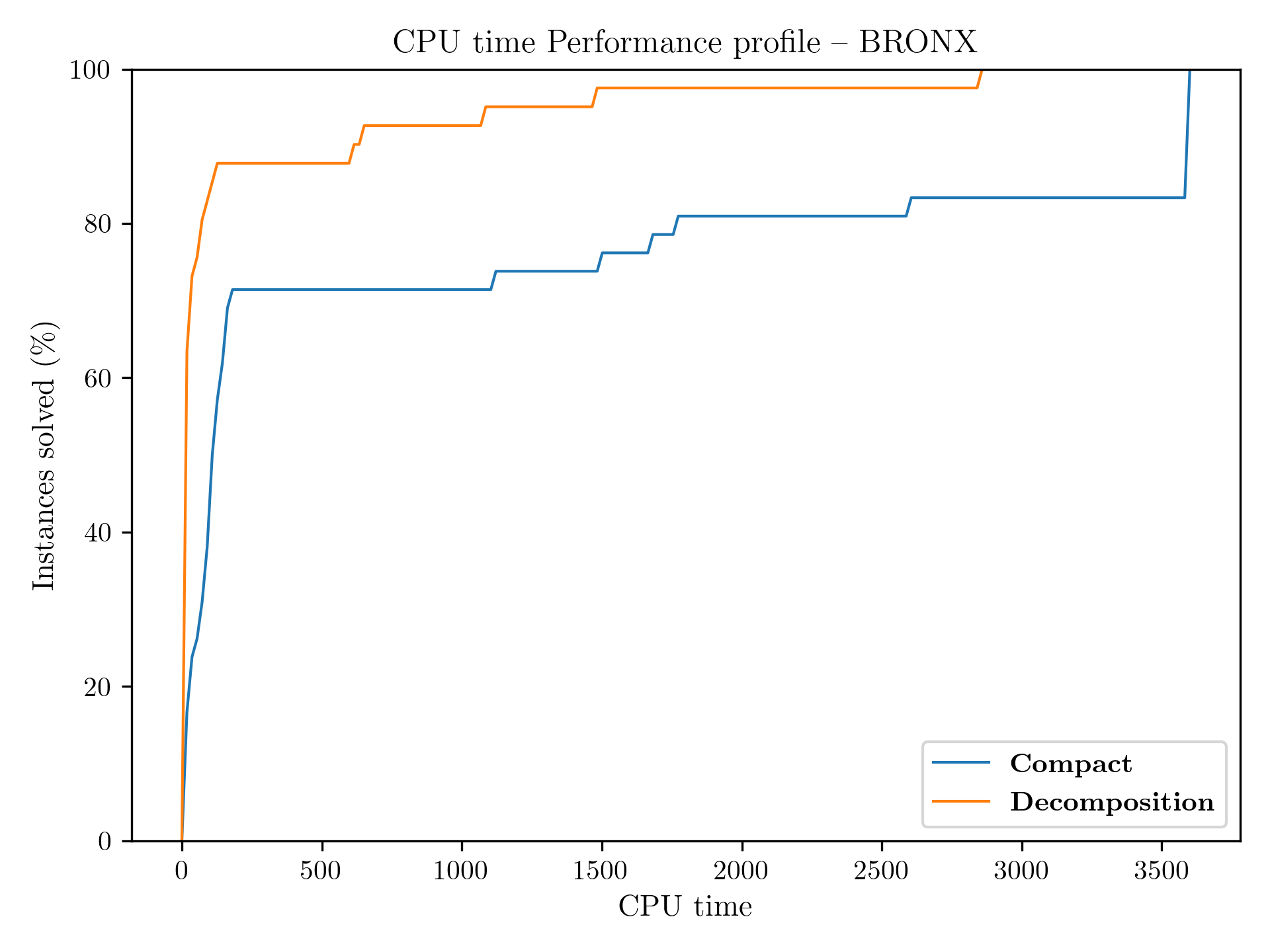}~\includegraphics[width=0.5\textwidth]{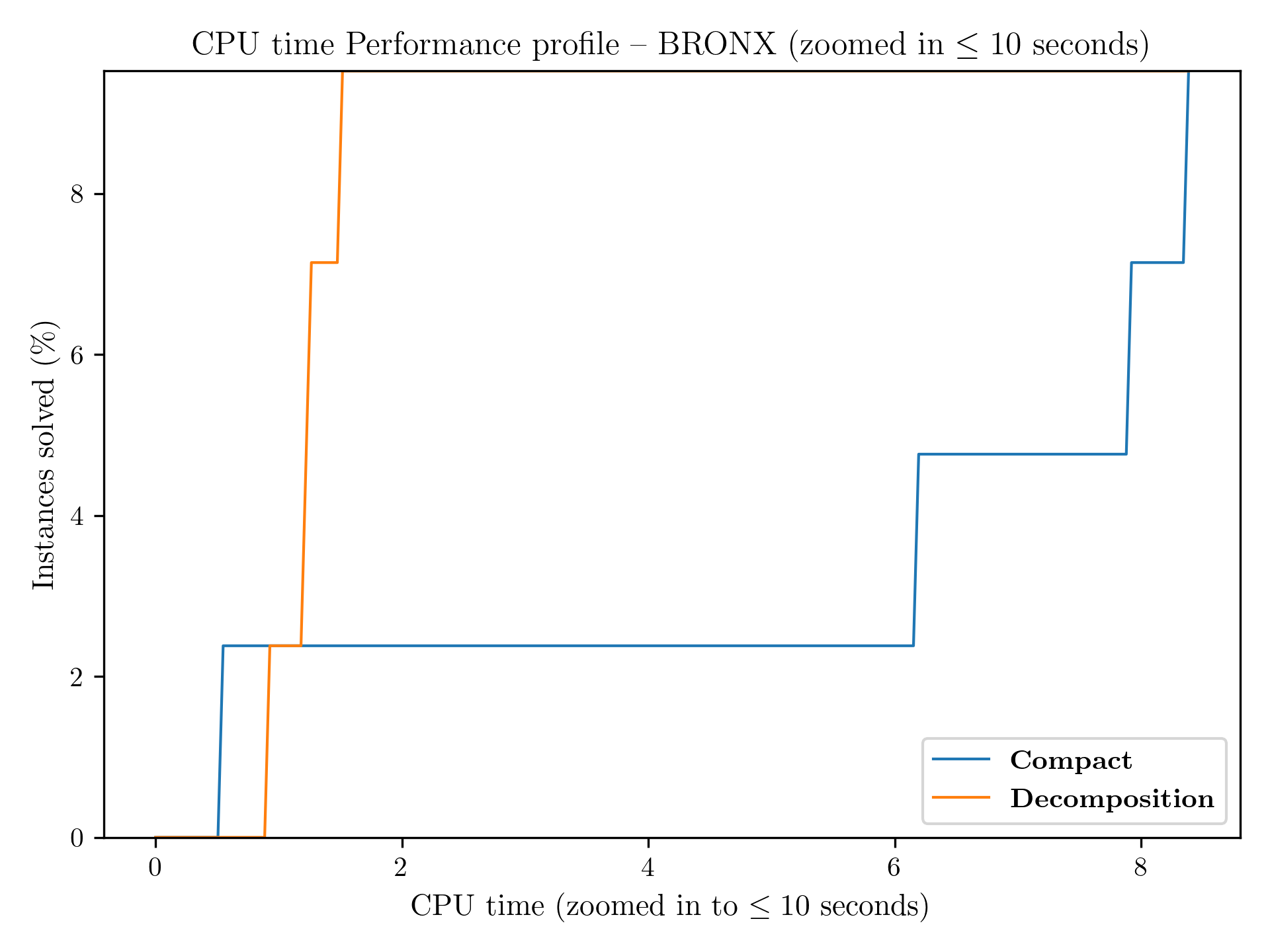}}
{Performance Profile of CPU Time for Bronx in our experiments (left) and zoomed to those solved in less than $10$ seconds (right).\label{fig:pp_bronx}}
{}
\end{figure}

 \begin{figure}
 {\includegraphics[width=0.5\textwidth]{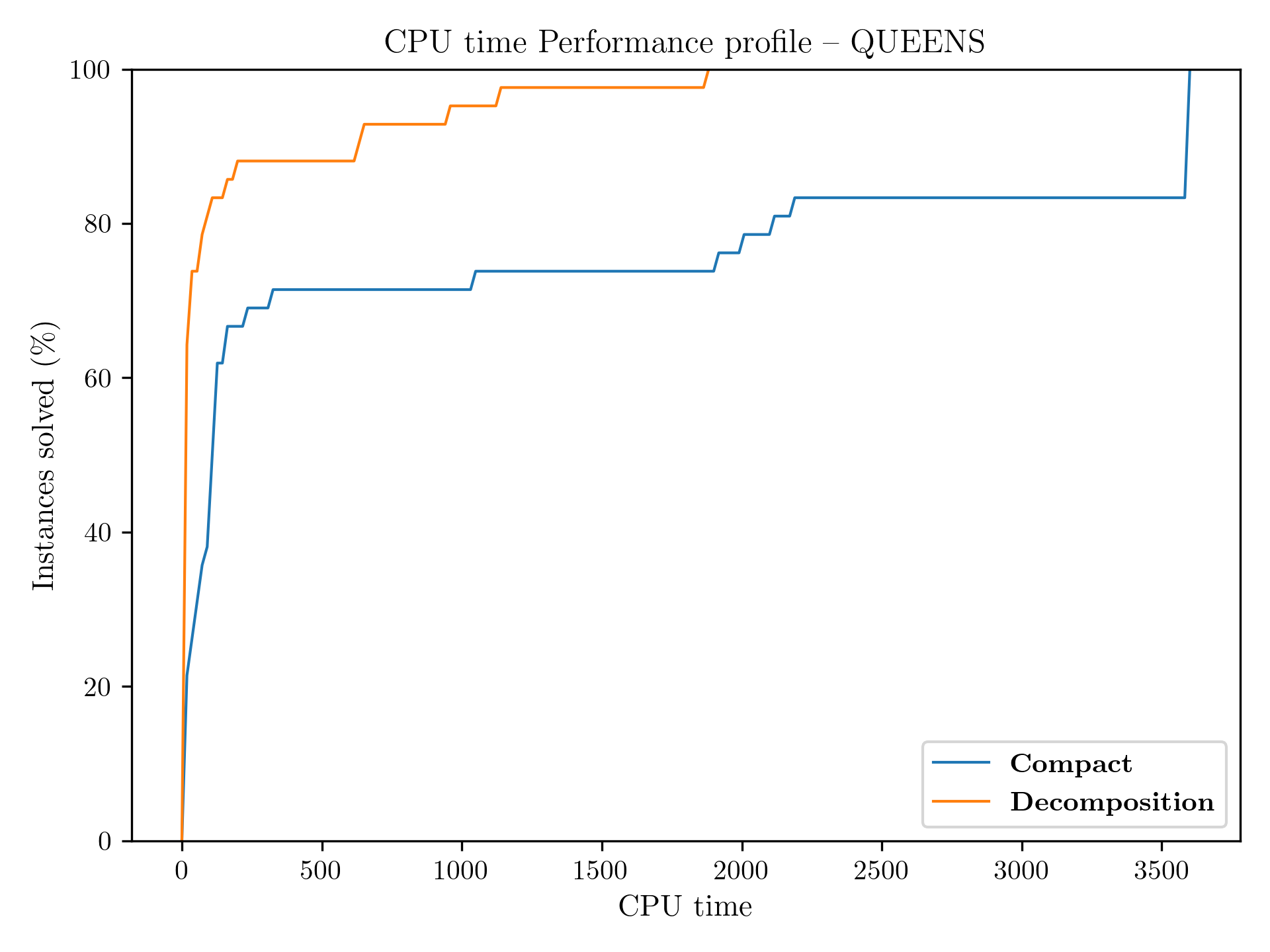}~\includegraphics[width=0.5\textwidth]{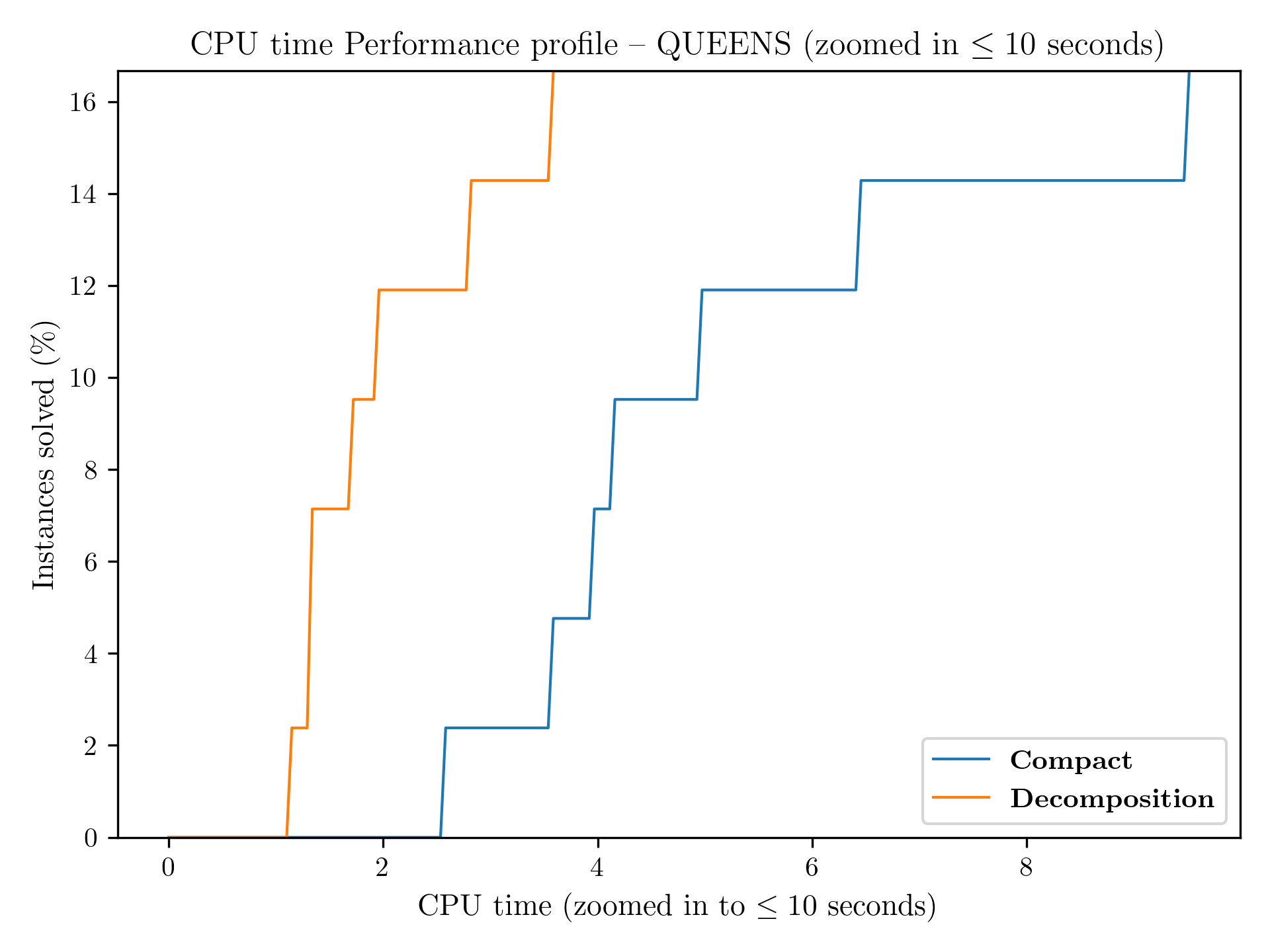}}
{Performance Profile of CPU Time for Queens in our experiments (left) and zoomed to those solved in less than $10$ seconds (right).\label{fig:pp_queens}}
{}
\end{figure}

 \begin{figure}
 {\includegraphics[width=0.75\textwidth]{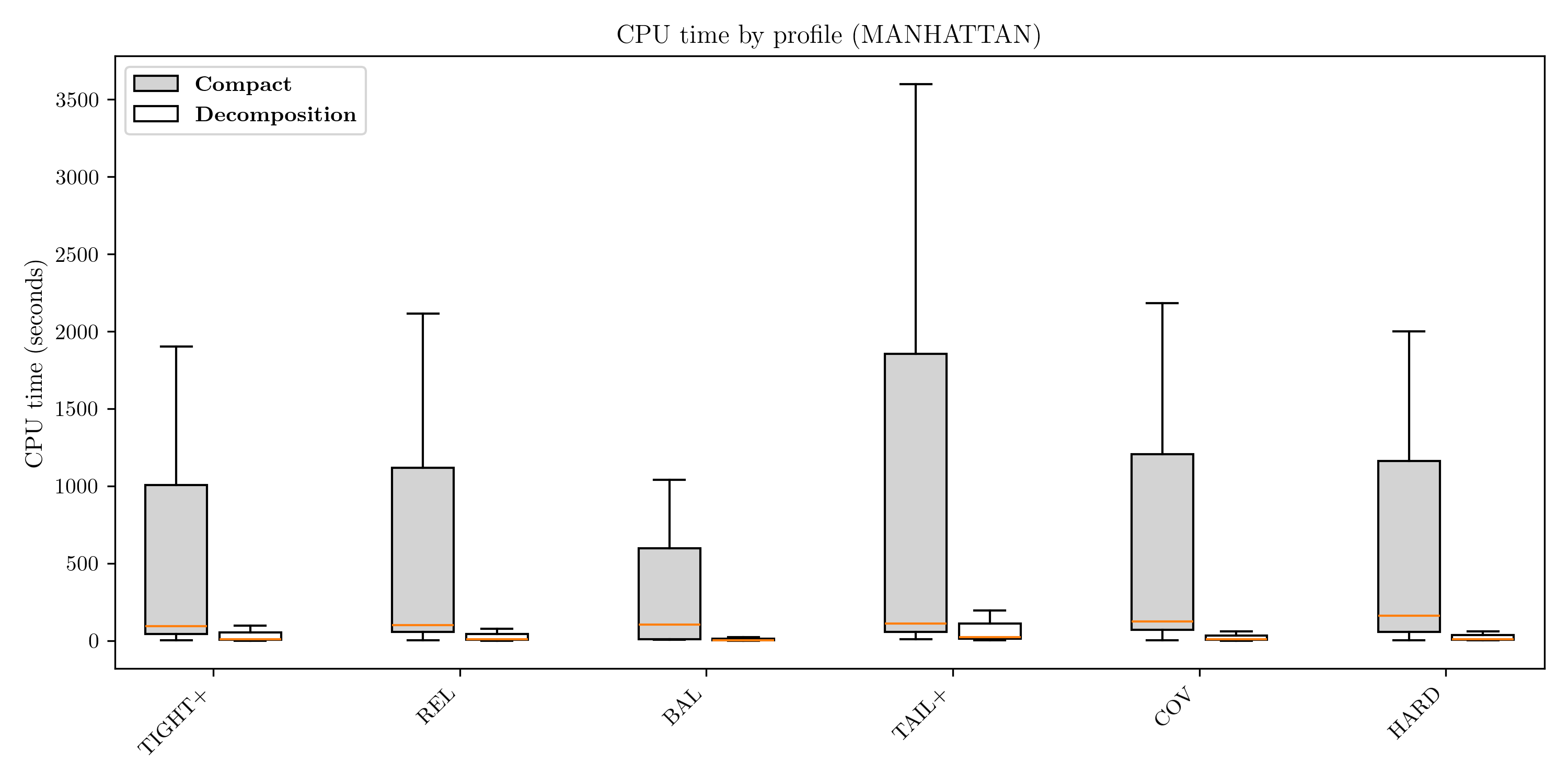}}
{Boxplots of CPU time by profile for Manhattan. For each profile, the distribution of CPU times is shown for the \textbf{compact} and \textbf{decomposition} approaches.\label{fig:bp_manhattan}}
{}
\end{figure}

 \begin{figure}
 {\includegraphics[width=0.75\textwidth]{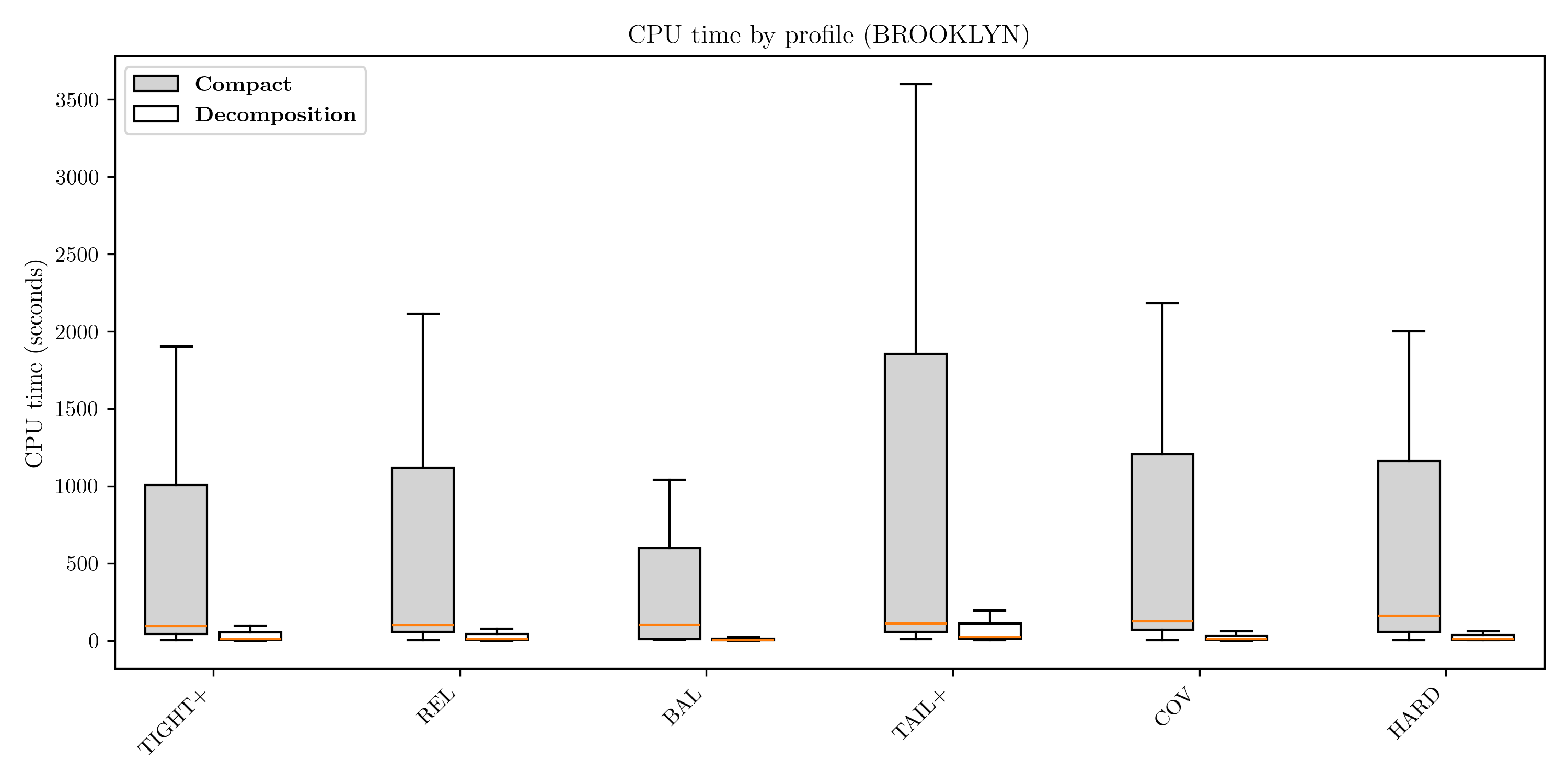}}
{Boxplots of CPU time by profile for Brooklyn. For each profile, the distribution of CPU times is shown for the \textbf{compact} and \textbf{decomposition} approaches.\label{fig:bp_brooklyn}}
{}
\end{figure}

 \begin{figure}
 {\includegraphics[width=0.75\textwidth]{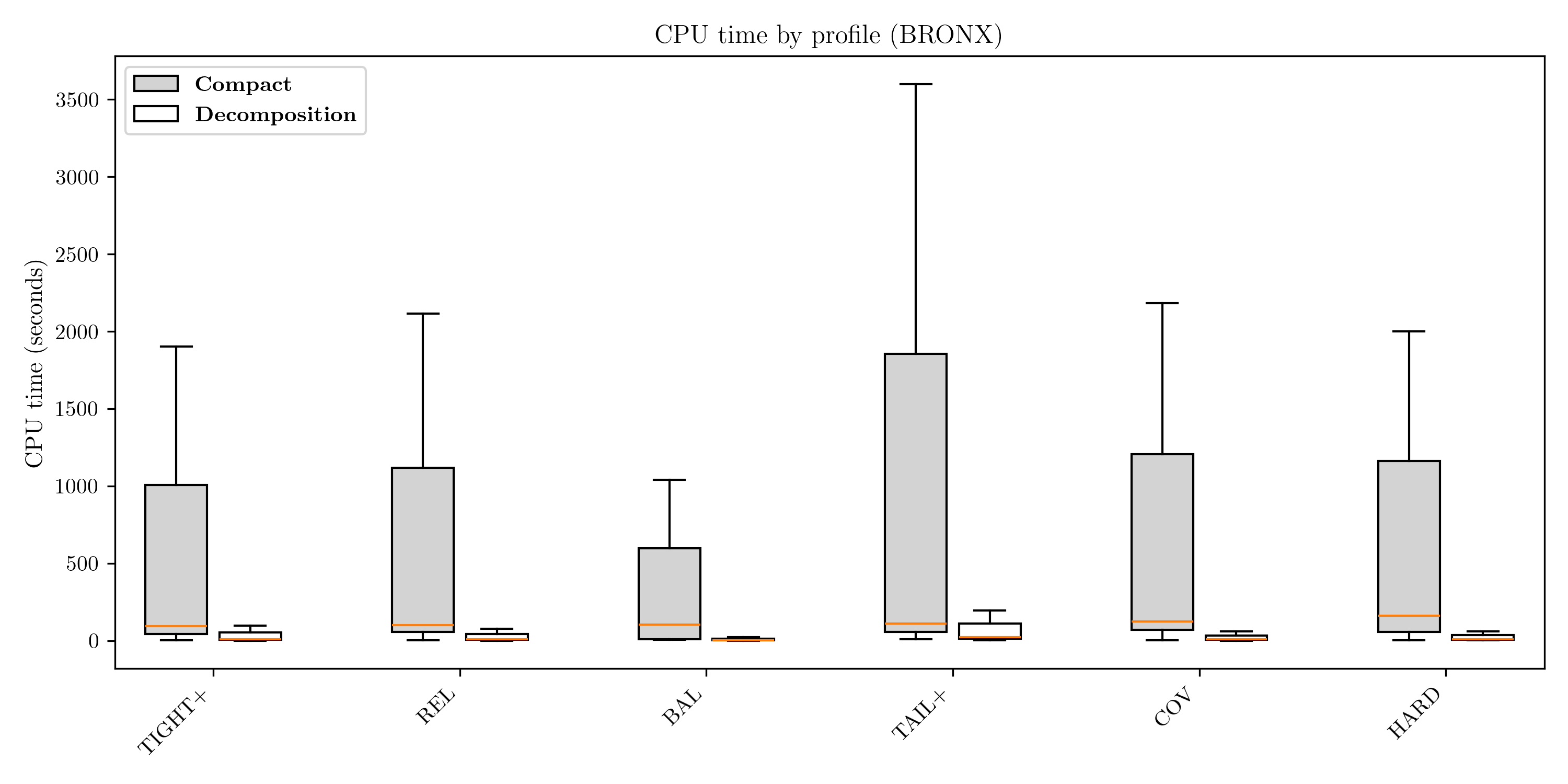}}
{Boxplots of CPU time by profile for Bronx. For each profile, the distribution of CPU times is shown for the \textbf{compact} and \textbf{decomposition} approaches.\label{fig:bp_bronx}}
{}
\end{figure}

 \begin{figure}
 {\includegraphics[width=0.75\textwidth]{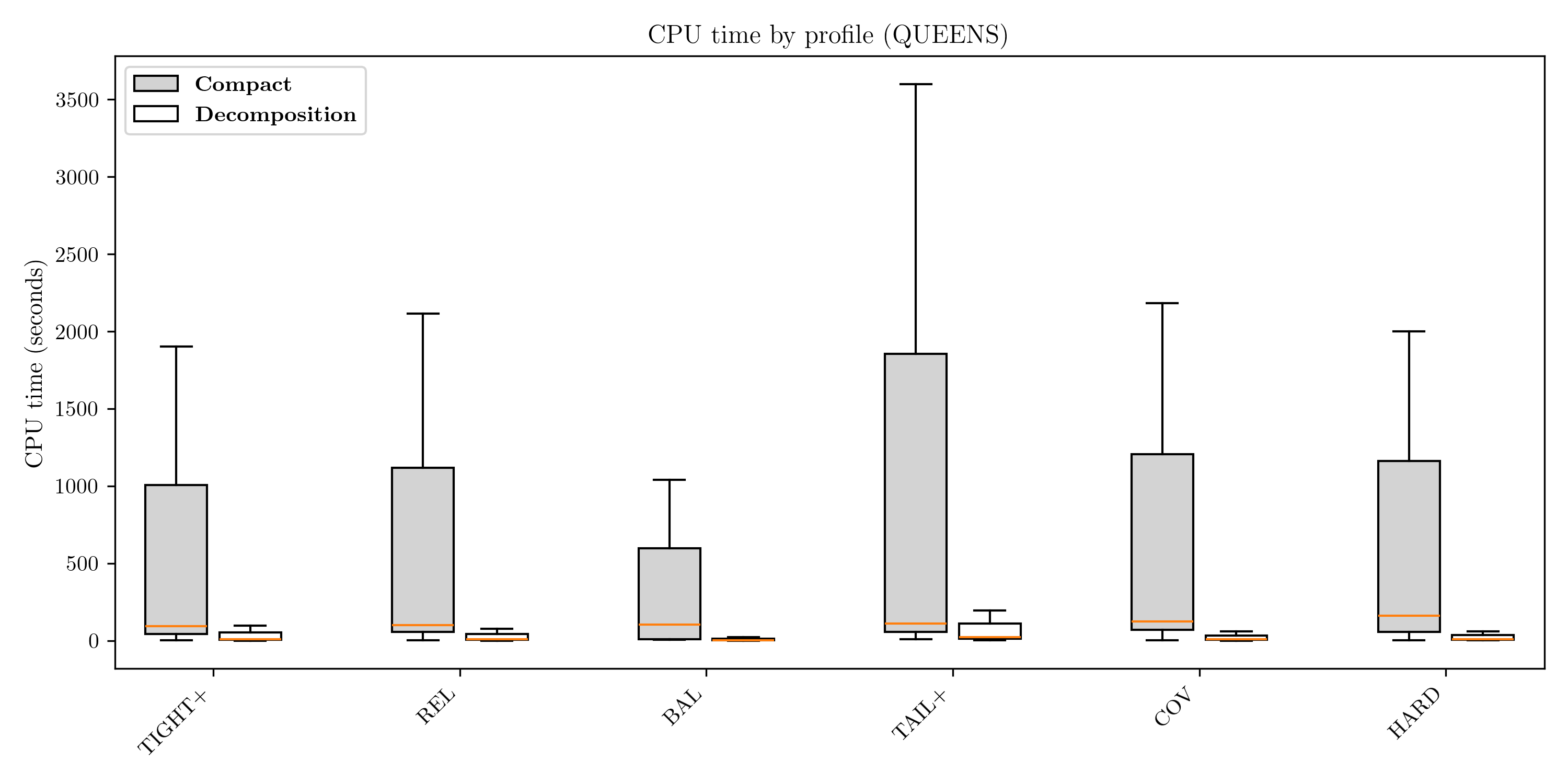}}
{Boxplots of CPU time by profile for Queens. For each profile, the distribution of CPU times is shown for the \textbf{compact} and \textbf{decomposition} approaches.\label{fig:bp_queens}}
{}
\end{figure}

\section{Case Study. Additional Plots}
\label{sec:case_study}

 \begin{figure}
 {\includegraphics[width=.75\textwidth]{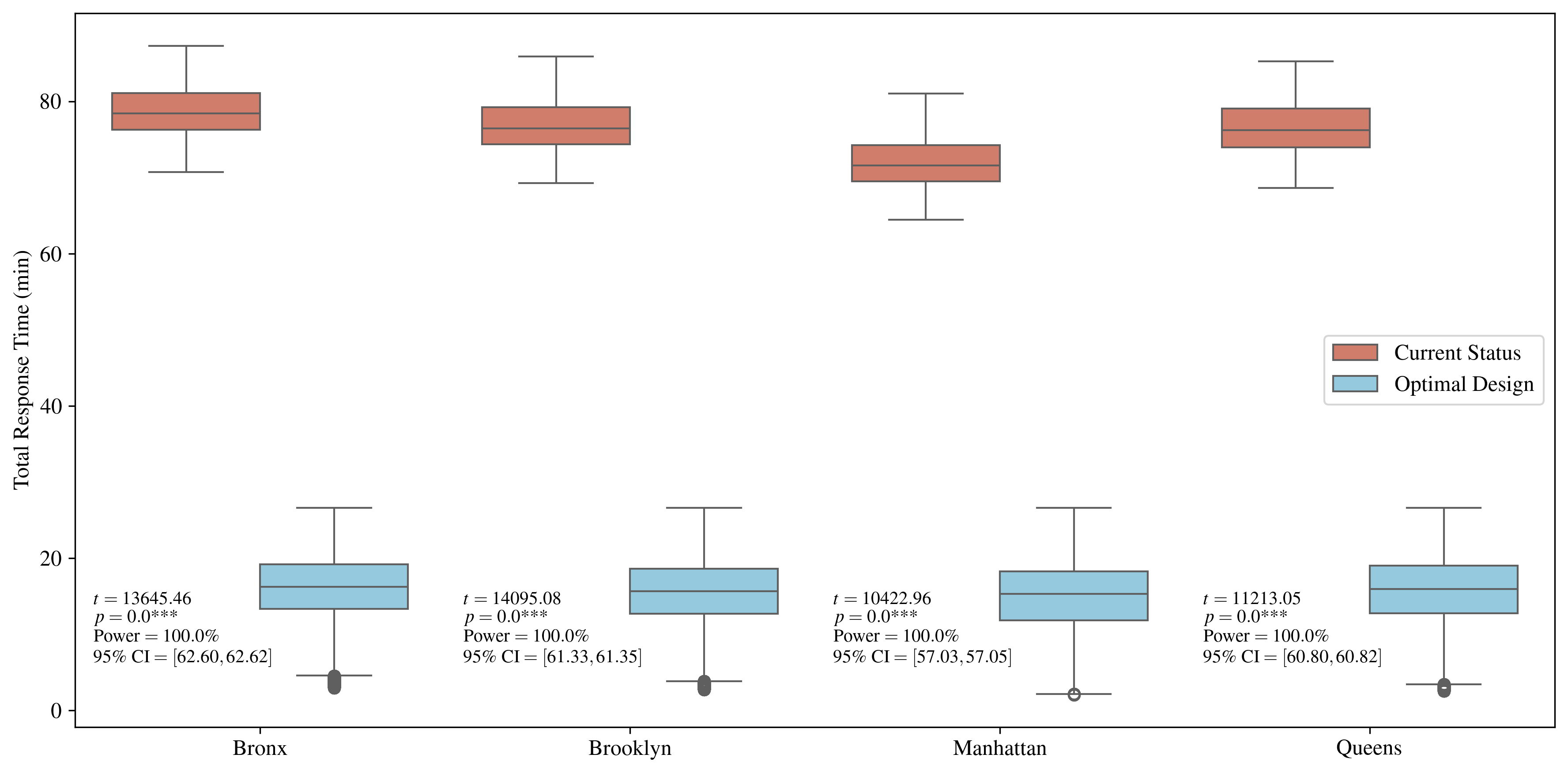}}
{Boxplots of obtained response times with the estimated rates and those obtained with our models for each borough.\label{fig:boxplot_response_time_borough}}
{}
\end{figure}

 \begin{figure}
 {\includegraphics[width=.75\textwidth]{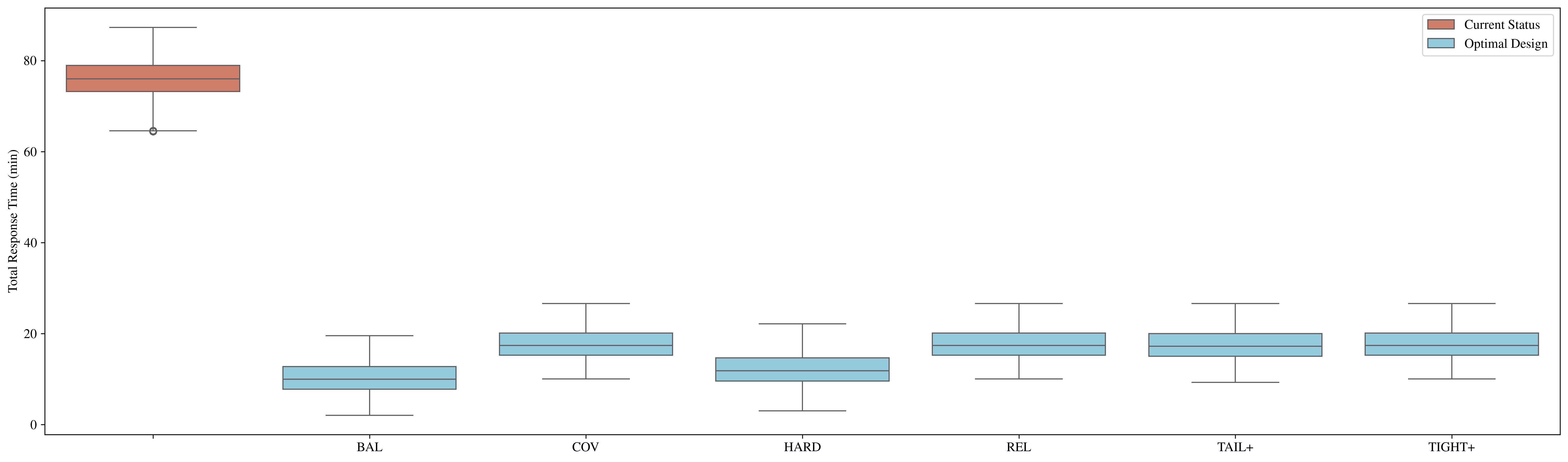}}
{Boxplots of obtained response times with the estimated rates and those obtained with our models for each parameters profile.\label{fig:efficiency_profiles}}
{}
\end{figure}

 \begin{figure}
 {\includegraphics[width=0.75\textwidth]{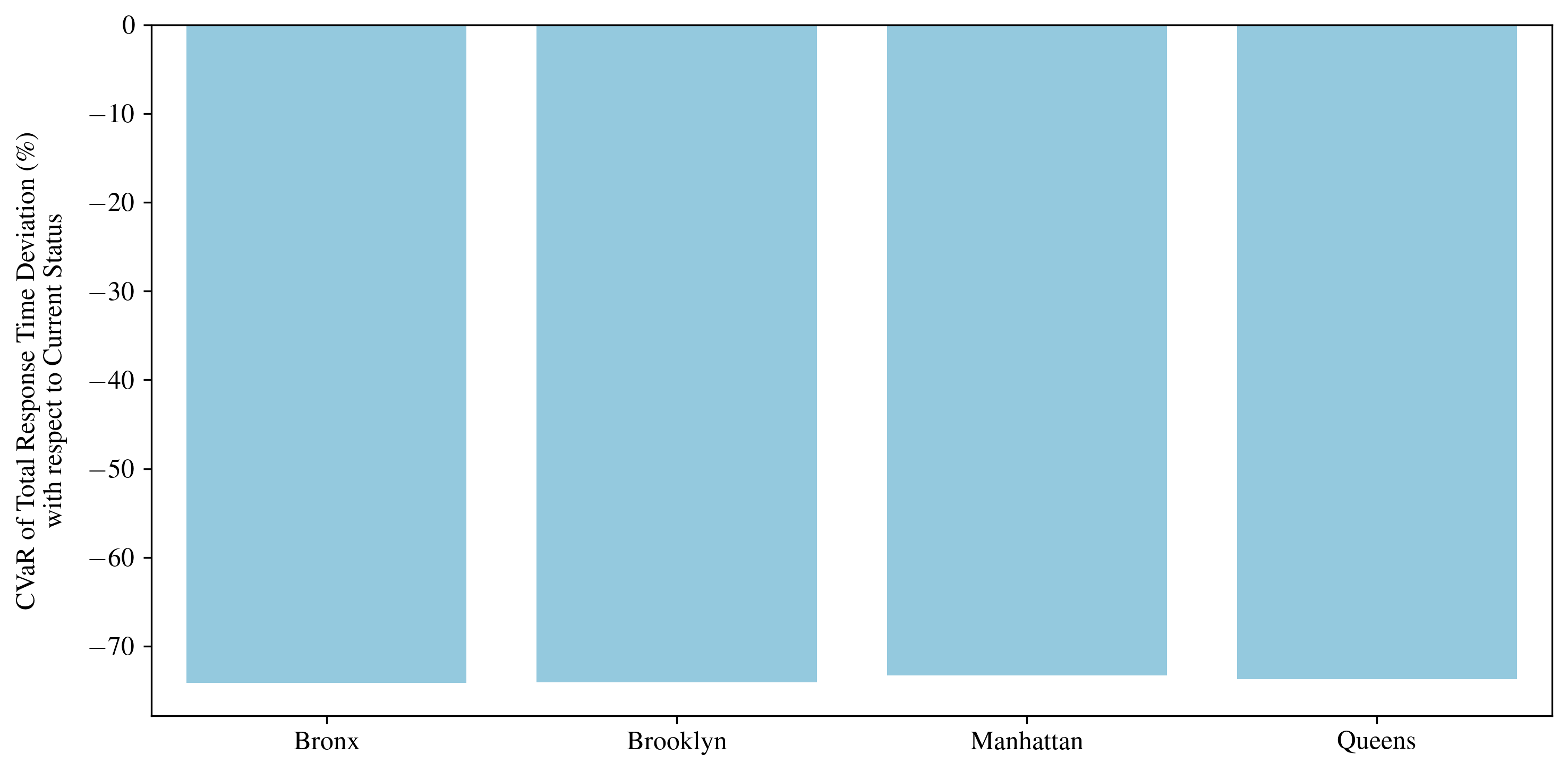}}
{Relative change in the CVaR of total response time by borough with respect to the current system.
Negative values indicate a reduction in extreme response times under the optimal design.\label{fig:barplot_cvar_borough}}
{}
\end{figure}

 \begin{figure}
 {\includegraphics[width=.75\textwidth]{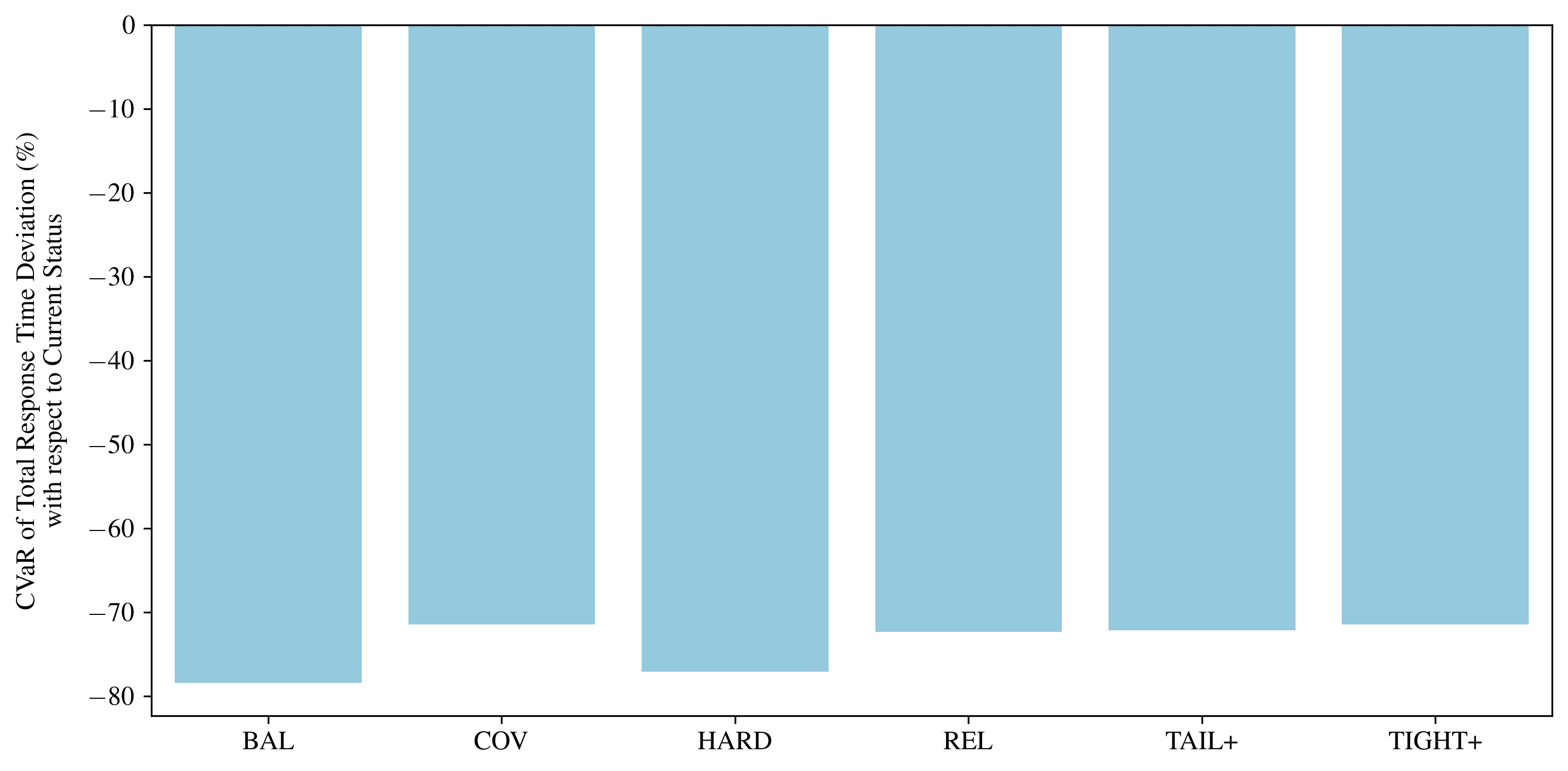}}
{Relative change in the CVaR of total response time by operational profile with respect to the current system.
Negative values indicate improved protection against extreme delays under the optimal design.\label{fig:barplot_cvar_profile}}
{}
\end{figure}

 \begin{figure}
 {\includegraphics[width=0.75\textwidth]{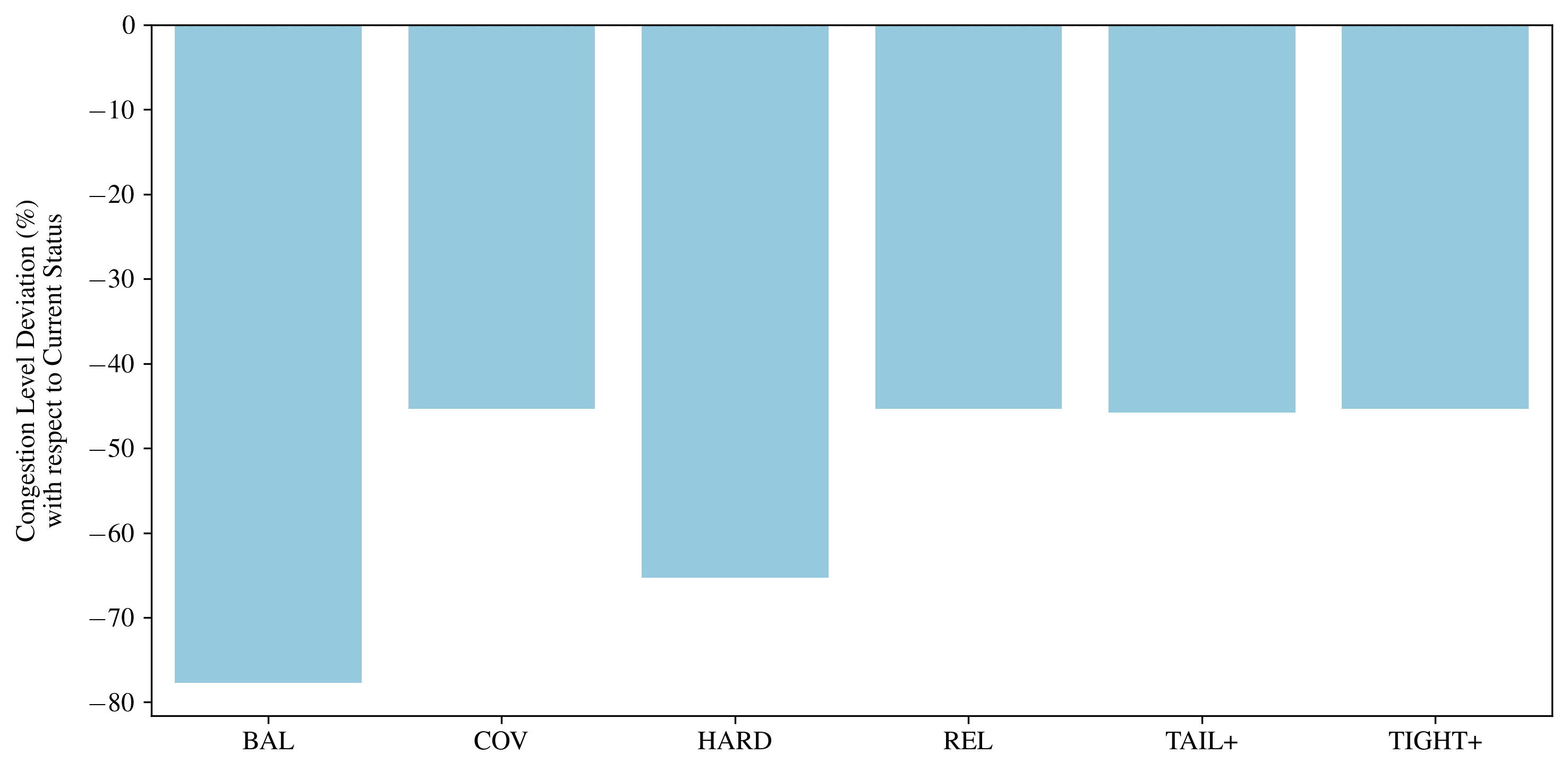}}
{Congestion level deviation with respect to the current system by profile.
Negative values indicate a reduction in congestion level under the optimal design.\label{fig:barplot_cong_profile}}
{}
\end{figure}

 \begin{figure}
 {\includegraphics[width=.75\textwidth]{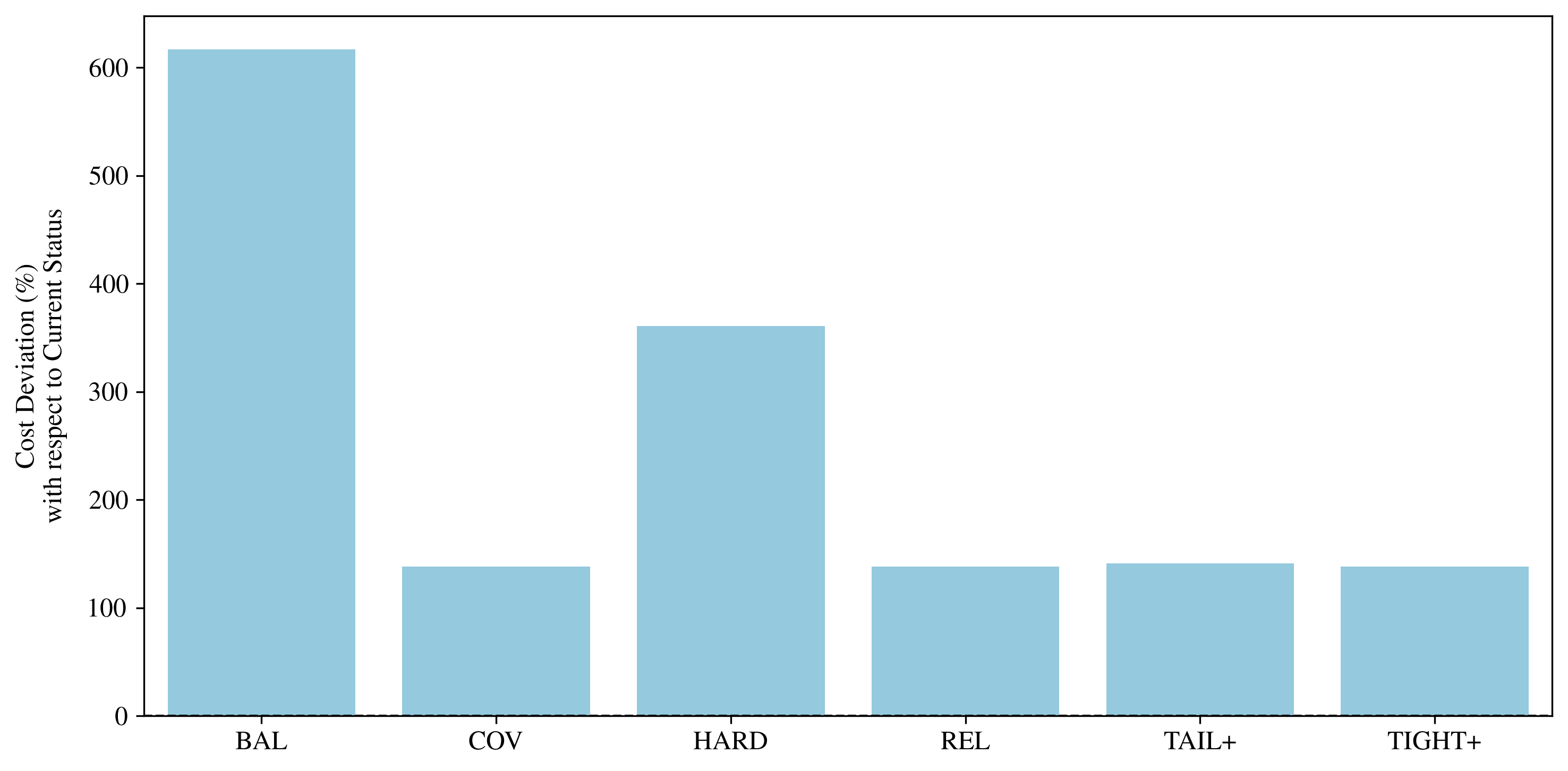}}
{Cost deviation with respect to the current system by operational profile.
Values above zero indicate an increase in cost under the optimal design.\label{fig:barplot_cost_profile}}
{}
\end{figure}

\bibliographystyle{plainnat}
 \bibliography{facility_chance}

\end{document}